\documentclass{amsart}

\usepackage{afterpage}
\usepackage{amsfonts}
\usepackage{amsmath}
\usepackage{amssymb}
\usepackage{amsthm}
\usepackage[style=numeric, backend=biber, maxbibnames=99, doi=false, isbn=false, url=false]{biblatex}
\usepackage{booktabs}
\usepackage{hyperref}
\usepackage{longtable}
\usepackage{mathrsfs}
\usepackage{multirow}
\usepackage{tikz-cd}
\usepackage{xparse}

\theoremstyle{plain}
\newtheorem{theorem}{Theorem}[section]
\newtheorem{corollary}[theorem]{Corollary}
\newtheorem{lemma}[theorem]{Lemma}
\newtheorem{proposition}[theorem]{Proposition}
\newtheorem{conjecture}[theorem]{Conjecture}

\theoremstyle{definition}
\newtheorem{definition}[theorem]{Definition}

\theoremstyle{remark}
\newtheorem{notation}[theorem]{Notation}
\newtheorem{remark}[theorem]{Remark}
\newtheorem{example}[theorem]{Example}

\numberwithin{equation}{subsection}

\DeclareMathOperator{\AAbel}{\mathbf{A}}
\DeclareMathOperator{\Aut}{Aut}
\DeclareMathOperator{\Br}{Br}
\DeclareMathOperator{\Div}{Div}
\DeclareMathOperator{\DDiv}{\mathbf{Div}}
\DeclareMathOperator{\Gal}{Gal}
\DeclareMathOperator{\id}{id}
\DeclareMathOperator{\Jac}{Jac}
\DeclareMathOperator{\Norm}{N}
\DeclareMathOperator{\Pic}{Pic}
\DeclareMathOperator{\PPic}{\mathbf{Pic}}
\DeclareMathOperator{\Res}{Res}
\DeclareMathOperator{\Stab}{Stab}
\DeclareMathOperator{\Spec}{Spec}
\DeclareMathOperator{\Sym}{Sym}

\DeclareDocumentCommand{\GL}{O{2} m}{\operatorname{GL}_{#1}(#2)}
\DeclareDocumentCommand{\O}{}{\mathcal{O}}
\DeclareDocumentCommand{\Pone}{}{\(\mathbb{P}^{1}\)}
\DeclareDocumentCommand{\Q}{}{\mathbb{Q}}
\DeclareDocumentCommand{\Qab}{}{{\Q^{\text{ab}}}}
\DeclareDocumentCommand{\Qbar}{}{\overline{\Q}}
\DeclareDocumentCommand{\SL}{O{2} m}{\operatorname{SL}_{#1}(#2)}
\DeclareDocumentCommand{\W}{}{\mathbf{W}}
\DeclareDocumentCommand{\Z}{s o}{\IfBooleanT{#1}{(}\mathbb{Z}\IfValueT{#2}{/ #2 \mathbb{Z}}\IfBooleanT{#1}{)^{\times}}}
\DeclareDocumentCommand{\Zhat}{s}{\widehat{\Z}\IfBooleanT{#1}{^{\times}}}

\title{Isolated points on modular curves}
\date{}
\author{Kenji Terao}

\begin{document}	
	\begin{abstract}
		We study isolated points on the modular curves $X_{H}$, for $H$ a subgroup of $\operatorname{GL}_{2}(\mathbb{Z}/n \mathbb{Z})$ for some $n \geq 1$. In particular, we prove a single-sink theorem for such isolated points, which traces the existence of all such isolated points with the same $j$-invariant back to an isolated point on a single curve. Building on this result, we also present a uniform strategy for determining the isolated points on any family of modular curves. As an example, we use this strategy to classify the isolated points with rational $j$-invariant on all modular curves of level 7, as well as the modular curves $X_{0}(n)$, the latter assuming a conjecture on images of Galois representations of elliptic curves over $\Q$. Underpinning all of this, we develop a theory of isolated divisors on geometrically disconnected varieties, which may be of independent interest.
	\end{abstract}
	
	\maketitle
	
	\section{Introduction} \label{sec:introduction}

\subsection{Isolated points on modular curves} \label{sec:introduction:introduction}

In 1991, Abramovich and Harris \cite{abramovich1991abelian} built upon Faltings's theorem to give necessary conditions for the existence of infinitely many points of a given degree on a geometrically irreducible curve defined over a number field. Namely, if $C/k$ has infinitely many degree $d$ points, then either there exists a degree $d$ map $C \to \mathbb{P}^{1}$, or the image of the map $\Sym^{d}(C) \to \Jac(C)$ contains the translate of a positive rank abelian subvariety $A$ of $\Jac(C)$ by the image of a point of degree $d$.

An argument of Bourdon, Ejder, Liu, Odumodu and Viray \cite[Theorem 4.2]{bourdon2019} shows that these two conditions are also sufficient. In the first case, one obtains an infinite family of degree $d$ points on $C$ given by the pullbacks of the $k$-rational points of $\mathbb{P}^{1}$. In the second case, the $k$-rational points of $A$ give infinitely many $k$-rational points on $\Sym^{d}(C)$, which gives an infinite family of degree $d$ points on $C$. By virtue of their construction, the points of these families are called \Pone-parametrized and AV-parametrized respectively. The remaining points of $C$, which, by \cite[Theorem 4.2]{bourdon2019}, are finite in number, are said to be isolated. These isolated points may be interpreted as exceptional points of $C$, whose existence is not justified by the geometric properties of the curve $C$.

The study of isolated points was initiated by Bourdon, Ejder, Liu, Odumodu and Viray in \cite{bourdon2019}, with particular emphasis on computing the isolated points on the modular curves $X_{1}(n)$. Since then, much work has been done on this specific case. For instance, Bourdon, Gill, Rouse and Watson have classified all odd degree isolated points with rational $j$-invariant on the modular curves $X_{1}(n)$ in \cite{bourdon2024odd}, while Ejder has classified all non-CM isolated points with rational $j$-invariant on the modular curves $X_{1}(\ell^{n})$ of prime-power level, with $\ell > 7$, in \cite{ejder2022}. Recently, Bourdon, Hashimoto, Keller, Klagsbrun, Lowry-Duda, Morrison, Najman and Shukla, in \cite{bourdon2025algorithm}, have given an algorithm for computing the isolated points on the modular curves $X_{1}(n)$ with given rational $j$-invariant, and gave a conjecture for the finite list of rational $j$-invariants corresponding to non-CM isolated points on the modular curves $X_{1}(n)$.

While the aforementioned papers work exclusively with the modular curves $X_{1}(n)$, one can define a number of other modular curves. For instance, given a subgroup $H$ of $\GL{\Z[n]}$ for some $n \geq 1$, or equivalently, an open subgroup $H$ of $\GL{\Zhat}$, one can define a modular curve $X_{H}$. This modular curve parametrizes elliptic curves whose mod-$n$ Galois representation has image contained in the subgroup $H$, c.f. Section \ref{sec:modular_curves}. The set of modular curves $X_{H}$ subsumes the family of modular curves $X_{1}(n)$, as well as a number of other such families, such as $X_{0}(n)$, $X_{s}(p)$ and $X_{ns}(p)$. Working with the modular curves $X_{H}$ thus allows us to study these diverse families of modular curves in a uniform manner.

The starting point of this paper is to generalize many of the results and techniques developed for studying isolated points on $X_{1}(n)$ to arbitrary modular curves $X_{H}$. As indicated above, this opens the door to studying isolated points on other families of modular curves, such as $X_{0}(n)$, $X_{s}(p)$ or $X_{ns}(p)$, without rehashing the work done for $X_{1}(n)$. Subsequently, we explore some of the consequences of these results, which leads us to the two main contributions of this paper.

\subsection{A single-sink theorem} \label{sec:introduction:single_sink}

The first main goal of the paper is to prove a single-sink theorem for isolated points on modular curves. The statement of the theorem is as follows.

\begin{theorem}[Single-sink theorem]
	Let $H$ be a subgroup of $\GL{\Z[n]}$ for some $n \geq 1$, and let $x$ be a non-cuspidal, non-CM isolated point on $X_{H}$. Let $E / \Q(j(x))$ be an elliptic curve such that $j(E) = j(x)$, and let $G = \rho_{E}(G_{\Q(j(x))})$ be the image of the adelic Galois representation of $E$. Then the modular curve $X_{G}$ contains an isolated point $y$ with $j(x) = j(y)$ and $\Q(y) = \Q(j(x))$.
\end{theorem}

This theorem demonstrates that the existence of isolated points with given $j$-invariant on any modular curve can be traced back to a single curve, namely the curve $X_{G}$ as defined in the theorem. Strikingly, this latter curve depends solely on the $j$-invariant of the isolated point, which determines the group $G$, rather than the original modular curve $X_{H}$ which contains the isolated point. This provides a very effective technique for determining whether there are any isolated points on modular curves with a given $j$-invariant. Indeed, given a non-CM $j$-invariant $j \in \Q$, one can compute the image $G$ of the adelic Galois representation described in the theorem using the algorithm given in \cite{zywina2024}. For $j$-invariants $j \in \Qbar$ defined over larger number fields, an outline of such an algorithm is given in \cite{zywina2025}, though there is at present no concrete implementation of such an algorithm. By the single-sink theorem, it then suffices to consider the points defined over $\Q(j)$ with $j$-invariant equal to $j$ on the modular curve $X_{G}$, and determine whether these are isolated. While there are no universal techniques for this latter problem, in many situations it is quite straightforward to solve in an ad-hoc manner.

\begin{example} \label{ex:introduction:serre_curves}
	As an example, let $E / \Q$ be a Serre curve, that is to say, such that the index $[\GL{\Zhat} : \rho_{E}(G_{\Q})]$ of the image of the adelic Galois representation of $E$ is equal to 2. These correspond to the elliptic curves whose adelic Galois image is as large as possible, since, by \cite[Proposition 22]{serre1972}, the index of the adelic Galois image must be divisible by 2.

	Denote by $G$ the image $\rho_{E}(G_{\Q})$ of the adelic Galois representation of $E$. By \cite[Section 1.6]{zywina2024}, we know that the intersection $G \cap \SL{\Zhat}$ is equal to the commutator subgroup $[\GL{\Zhat}, \GL{\Zhat}]$. Therefore, using the formula in \cite[Section 2.3]{rouse2022}, one can show that the genus of the modular curve $X_{G}$, which depends solely on the intersection $G \cap \SL{\Zhat}$, is zero. In particular, there is an isomorphism $X_{G} \cong \mathbb{P}^{1}$. From the definitions, it is straightforward to see that any $\Q$-rational point on $X_{G}$ is \Pone-parametrized, and therefore, is not isolated. Thus, by the single-sink theorem, there are no isolated points with $j$-invariant equal to $j(E)$ on any modular curve. 

	Notably, by \cite{jones2010}, almost all elliptic curves $E / \Q$ are Serre curves. Therefore, the above argument shows that, for almost all $j \in \Q$, there are no isolated points with $j$-invariant equal to $j$ on any modular curve. This justifies the expectation that isolated points are rare occurrences, even if one considers all modular curves simultaneously.
\end{example}

The single-sink theorem also highlights a strong link between finding isolated points and determining the possible images of Galois representations of elliptic curves, so-called Mazur's Program B. Indeed, as the above example shows, knowledge of the image of the Galois representation associated to an elliptic curve $E$ allows one to understand whether there are any isolated points on a modular curve with $j$-invariant equal to $j(E)$, which serves as an important first step towards finding all isolated points on modular curves. As there has been much work on Mazur's Program B, the single-sink theorem provides an explicit mechanism to transfer existing work on Mazur's Program B to the problem of finding isolated points on modular curves. This will be one of the key features in our approach to the latter problem.

Before moving on to the problem of finding isolated points on modular curves, we briefly explain the nomenclature of the single-sink theorem.

\begin{remark}
	As will be described in Section \ref{sec:introduction:mapping_isolated_points}, the proof of the single-sink theorem relies on results allowing one to map isolated points between different curves. In particular, given two points $x$ and $y$, with the same $j$-invariant and satisfying certain group-theoretic conditions, on two different modular curves, these results show that if $x$ is isolated, so is $y$. One can thus create a directed graph whose vertex set is the set of all points on all modular curves with a given $j$-invariant, and whose edges represent these implications. We call such a graph an \emph{isolation graph}, see Definition \ref{def:level_7:isolation_graph}. In this context, the single-sink theorem then states that such an isolation graph has a single terminal vertex, or sink, for any $j$-invariant.
\end{remark}

\subsection{Finding isolated points on modular curves} \label{sec:introduction:finding_isolated_points}

The second goal of this paper is to present a new method for finding isolated points on modular curves. This method expands on both existing techniques developed for studying isolated points on $X_{1}(n)$, as well as new results such as the single-sink theorem, in order to give a systematic framework for finding isolated points on modular curves. Moreover, by working with all of the modular curves $X_{H}$ in a uniform manner, it enables the study of isolated points on other families of modular curves, such as $X_{0}(n)$, $X_{s}(p)$ or $X_{ns}(p)$.

This method proceeds in three main steps. Suppose that we have some set of modular curves $S$ and some subset $J$ of $\Qbar$, both possibly infinite, and we wish to find all of the isolated points on the modular curves in $S$ whose $j$-invariant is contained in the set $J$. For instance, in the case considered by \cite{bourdon2024odd}, \cite{bourdon2025algorithm} and \cite{ejder2022}, the set $S$ consists of the modular curves $X_{1}(n)$ for $n \geq 1$, while $J$ is the set of rational numbers $\Q$.

The first step is to compute a finite subset $J' \subset J$ containing the $j$-invariant of any isolated point on a modular curve in $S$ with $j$-invariant in $J$. This imposes some restrictions on the set $S$ of modular curves for which this method can be applied, as it may not be possible to find such a finite set $J'$. For instance, this is impossible if $S$ is the set of all modular curves $X_{H}$ and $J$ the set of rational numbers. However, in practice, this is not a particularly stringent restriction on the set $S$, as many families of modular curves satisfy this requirement. For example, when $J$ is the set of rational numbers and $S$ is the set of modular curves $X_1(n)$, such a finite set $J'$ exists by \cite[Corollary 1.7]{bourdon2019}, assuming Serre's uniformity conjecture.

However, it does not suffice to know of the existence of such a finite set $J'$; we must also explicitly compute such a set. To do so, we exploit the restrictions given by the single-sink theorem. Let $j \in J$ be the $j$-invariant of an isolated point on a modular curve in $S$, and let $G_{j}$ be the image of the adelic Galois representation associated to the $j$-invariant $j$, as defined in the single-sink theorem. By this latter theorem, the modular curve $X_{G_{j}}$ must contain an isolated point $y$ with $j(y) = j$ and $\Q(y) = \Q(j)$. Therefore, we may consider the set of $j$-invariants
\[
	\{j \in J : X_{G_{j}} \text{ has an isolated point $y$ with $j(y) = j$ and $\Q(y) = \Q(j)$}\}.
\]
As illustrated in Example \ref{ex:introduction:serre_curves}, it is likely that almost all of the $j$-invariants of $J$ do not satisfy the above criterion, and so the above set will be much smaller than $J$. However, this set can still be infinite, even when $J$ is the set of rational numbers.

To circumvent this, rather than using the single-sink theorem itself to construct the set $J'$, we employ generalizations of this result. These generalizations utilise extra information about the modular curves in the set $S$ in order to deduce the existence of isolated points on other modular curves. As an example, if the set $S$ consists of modular curves of a fixed level $n$, that is to say, if $S$ consists solely of modular curves $X_{H}$, where $H$ is a subgroup of $\GL{\Z[n]}$ for some fixed $n$, then one can employ the following result.

\begin{theorem} \label{thm:introduction:mod_n_image_isolated}
	Let $n \geq 1$, let $H \leq \GL{\Z[n]}$ be a subgroup, and let $x$ be a non-cuspidal isolated point on $X_{H}$ with $j(x) \notin \{0, 1728\}$. Let $E / \Q(j(x))$ be an elliptic curve such that $j(E) = j(x)$, and let $G_{n} = \rho_{E, n}(G_{\Q(j(x))})$ be the image of the mod-$n$ Galois representation of $E$. Then the modular curve $X_{G_{n}}$ contains an isolated point $y$ with $j(x) = j(y)$ and $\Q(y) = \Q(j(x))$.
\end{theorem}

Rather than giving an isolated point on a modular curve $X_{G}$ which is independent of the initial modular curve $X_{H}$, as in the single-sink theorem, this result gives an isolated point on a modular curve $X_{G_{n}}$ which depends on the level of the subgroup $H$. However, as the set $S$ consists of modular curves of a fixed level $n$, this modular curve $X_{G_{n}}$ is the same for all modular curves $X_{H} \in S$, provided the $j$-invariant of the isolated point is fixed. Therefore, if we denote by $G_{n, j}$ the image of the mod-$n$ Galois representation associated to the $j$-invariant $j$, as in the above theorem, we can consider the set
\[
	J' = \{j \in J : X_{G_{n, j}} \text{ has an isolated point $y$ with $j(y) = j$ and $\Q(y) = \Q(j)$}\}.
\]
This set $J'$ is finite, as there are finitely many modular curves $X_{G_{n, j}}$, each with finitely many isolated points, by \cite[Theorem 4.2]{bourdon2019}. Moreover, computing this set $J'$ can be done explicitly. The main difficulty is to compute the possible mod-$n$ Galois images $G_{n, j}$ as $j$ varies through all $j$-invariants in $J$. This problem has been extensively studied in the literature, especially when $J$ is the set of rational numbers, and the solution is known for many small values of $n$.

While Theorem \ref{thm:introduction:mod_n_image_isolated} works well when the set $S$ consists of modular curves of the same level, this is not necessarily true for other families of modular curves. For instance, the modular curves $X_{1}(n)$ or $X_{0}(n)$ are unbounded in level, which implies that one cannot apply Theorem \ref{thm:introduction:mod_n_image_isolated} in the same manner. Instead, we prove analogous results for each family of modular curves that we consider, which allow us to apply a similar strategy for computing such a finite set $J'$.

Once we have determined the finite set $J'$, the second step in the method is to restrict the set $S$: namely, for every $j$-invariant $j \in J'$, we find a finite subset $S_{j} \subset S$ such that there is a modular curve in $S$ with an isolated point with $j$-invariant $j$ if and only if there exists such a modular curve in $S_{j}$. To find such a finite subset $S_{j}$, we use theorems which are, in a sense, dual to the single-sink theorem and its generalizations, as described above. In particular, while the generalizations of the single-sink theorem use information about the modular curves in $S$ to restrict the $j$-invariant of the isolated points, these theorems use information about the $j$-invariant of the isolated point to reduce the set of modular curves which needs to be considered. For instance, the analogue to Theorem \ref{thm:introduction:mod_n_image_isolated} above is the following.

\begin{theorem}
	Let $H$ be a subgroup of $\GL{\Z[n]}$ for some $n \geq 1$, and let $x$ be a non-cuspidal, non-CM isolated point on $X_{H}$. Let $E / \Q(j(x))$ be an elliptic curve such that $j(E) = j(x)$. Let $G = \rho_{E}(G_{\Q(j(x))})$ be the image of the adelic Galois representation of $E$, and let $m$ be the level of $G$. Let $H_{m} \leq \GL{\Z[(n, m)]}$ be the reduction mod-$m$ of $H$. Then the modular curve $X_{H_{m}}$ contains an isolated point with $j$-invariant equal to $j(x)$.
\end{theorem}

This theorem demonstrates that, given a $j$-invariant, and knowing the image of the adelic Galois representation of an elliptic curve with said $j$-invariant, one can map isolated points with the given $j$-invariant on any modular curve $X_{H}$ to an isolated point on a modular curve $X_{H_{m}}$ of bounded level. Therefore, instead of searching for isolated points on all modular curves $X_{H}$, it suffices to find the isolated points with given $j$-invariant on a set of modular curves of bounded level. As there are finitely many such modular curves of bounded level, this set of modular curves gives us the finite subset $S_{j}$ which we need to consider.

In order to apply this technique successfully however, we require knowledge of the image of the adelic Galois representation of an elliptic curve with the given $j$-invariant. As there are only finitely many $j$-invariants to consider, these can be computed individually for each $j$-invariant, for instance using the algorithm described in \cite{zywina2024}.

When working with other families of modular curves, such as $X_{1}(n)$ and $X_{0}(n)$, we also prove analogues of the above theorem which allow us to restrict the level further. This exploits the additional knowledge of the subgroup $H$, and is instrumental in reducing the computations further.

We have now reduced the problem of finding all isolated points with $j$-invariant contained in $J$ on the modular curves of $S$ to finding the isolated points with $j$-invariant $j$ on the modular curves in $S_{j}$, for all $j \in J'$. The last step consists of iterating through all such pairs of $j$-invariants and modular curves, and for each pair, finding any isolated points with the given $j$-invariant on the given modular curve. This is done in a more ad-hoc way, using different techniques based on the $j$-invariant and modular curve at hand. However, as the sets $S_{j}$ and $J'$ are, in practice, rather small, this does not prove to be too difficult.

In order to give a practical example of the method, as well as a demonstration of its effectiveness in tackling various families of modular curves, we give two applications of this method to finding isolated points. Firstly, we classify the non-CM isolated points with rational $j$-invariant on the modular curves of level 7, and obtain the following.

\begin{theorem} \label{thm:introduction:level_7}
	Let $H \leq \GL{\Zhat}$ be an open subgroup of level 7, and let $x \in X_{H}$ be a non-cuspidal, non-CM isolated point such that $j(x) \in \Q$. Then $j(x) = \frac{3^{3} \cdot 5 \cdot 7^{5}}{2^{7}}$, and $H$ is conjugate to one of nine known subgroups.
\end{theorem}

The nine conjugacy classes of subgroups, as well as precise information about the closed points $x$ is also determined; see Theorem \ref{thm:level_7:level_7_isolated_points}. We note that this $j$-invariant has appeared previously in \cite{sutherland2012}, as the $j$-invariant of the only elliptic curves over $\Q$ giving rise to counterexamples of a local-global principle for rational isogenies of prime degree. Secondly, we classify the non-CM isolated points with rational $j$-invariant on the modular curves $X_{0}(n)$, assuming a conjecture of Zywina on images of Galois representations of elliptic curves over $\Q$ (Conjecture \ref{thm:isolated_points_x0:exceptional_j_invariants} below). This is much closer in flavor to the case of the modular curves $X_{1}(n)$ described above, and the techniques developed here, alongside a substantial computational effort, may also be applicable to the latter. This motivates the choice of the modular curves $X_{0}(n)$ as a simpler step towards understanding the modular curves $X_{1}(n)$. In this case, we obtain the following classification.

\begin{theorem} \label{thm:introduction:isolated_x0}
	Suppose that Conjecture \ref{thm:isolated_points_x0:exceptional_j_invariants} holds. Let $n \geq 1$, and let $x \in X_{0}(n)$ be a non-cuspidal, non-CM isolated point such that $j(x) \in \Q$. Then the $j$-invariant $j(x)$ belongs to the set
	\[\def\arraystretch{1.3}
		\left\{\begin{array}{c}
			-121, -24729001, -\frac{25}{2}, -\frac{121945}{32}, \frac{46969655}{32768}, -\frac{349938025}{8}, \\
			-\frac{297756989}{2}, -\frac{882216989}{131072}, \frac{3375}{2}, -\frac{140625}{8}, -\frac{1159088625}{2097152}, \\
			-\frac{189613868625}{128}, -9317, -162677523113838677
		\end{array}\right\}.
	\]
\end{theorem}

We note that, unlike in the case of the modular curves of level 7, we do not obtain a full classification of all isolated points with rational $j$-invariant on the modular curves $X_{0}(n)$; rather, we classify the $j$-invariants of these points, without specifying the set of modular curves on which the isolated points appear. While each $j$-invariant given above does appear as the $j$-invariant of an isolated point on some modular curve $X_{0}(n)$, the number of times it appears thusly is still unknown, see Remark \ref{rmk:isolated_points_x0:higher_levels}.

\subsection{Maps between isolated points} \label{sec:introduction:mapping_isolated_points}

In order to prove the single-sink theorem and the other results which form the basis of our method for finding isolated points, we rely heavily on mapping isolated points between curves. This technique, first developed by Bourdon, Ejder, Liu, Odumodu and Viray in \cite{bourdon2019}, allows one to obtain an isolated point on a curve $D$ given the existence of an isolated point on a second curve $C$ and a morphism $f : C \to D$ satisfying specific degree conditions. More precisely, the aforementioned authors prove the following theorem.

\begin{theorem}[{\cite[Theorem 4.3]{bourdon2019}}] \label{thm:introduction:pullback_isolated}
	Let $f : C \to D$ be a finite map of geometrically integral curves, let $x \in C$ be a closed point, and let $y = f(x) \in D$. Suppose that $\deg(x) = \deg(y) \cdot \deg(f)$. Then if $x$ is isolated, then $y$ is isolated.
\end{theorem}

We build on this technique by showing that there is a second such mapping theorem, which operates in the other direction. Namely, while the previous theorem maps an isolated point on the curve $C$ to an isolated point on the curve $D$, this theorem maps an isolated point on the curve $D$ to an isolated point on the curve $C$. The statement of this theorem is as follows.

\begin{theorem} \label{thm:introduction:pushforward_isolated}
	Let $f : C \to D$ be a finite map of curves, let $x \in C$ be a closed point, and let $y = f(x) \in D$. Suppose that $\deg(x) = \deg(y)$. Then if $y$ is isolated, then $x$ is isolated.
\end{theorem}

This theorem acts as an excellent foil to the previous result, and leveraging both simultaneously is the central argument of the proof of the single-sink theorem and the other results described in the previous sections.

One important obstacle to these applications is the fact that the modular curves $X_{H}$, with $H$ a subgroup of $\GL{\Z[n]}$ for some $n \geq 1$, can be geometrically disconnected. This impedes the direct application of \cite[Theorem 4.3]{bourdon2019}, as the proof of the single-sink theorem requires geometrically disconnected modular curves $X_{H}$, even if the modular curves described in the theorem itself are geometrically connected. To circumvent this issue, we reprove \cite[Theorem 4.3]{bourdon2019} in the case of geometrically disconnected curves. This requires defining and studying isolated points on geometrically disconnected curves, a task which occupies a significant fraction of this paper.

We also use this as an opportunity to recast the above theorems in a more geometrical perspective. More precisely, rather than directly considering isolated points on curves, we define and use the more general notion of isolated divisors. These isolated divisors are better-behaved under morphisms of curves, as they are stable under pushforwards and pullbacks. By restricting to isolated points in a second step, we recover the aforementioned mapping theorems, and in particular the necessary degree conditions, from the properties of pushforwards and pullbacks of divisors. This gives a geometric interpretation for these degree conditions, and justifies why we are not able to find other such conditions.

We note that, while the primary focus of this paper will be on isolated points on curves, the theory of isolated divisors does not require the base variety to be one-dimensional. As such, we state many of our results for arbitrary varieties, which opens the door for studying isolated divisors on higher-dimensional varieties in future work.

\subsection{Outline} \label{sec:introduction:outline}

The outline of the paper is as follows. In Section \ref{sec:isolated_divisors}, we develop the theory of isolated divisors on geometrically disconnected varieties, which culminates in the proof of Theorems \ref{thm:introduction:pullback_isolated} and \ref{thm:introduction:pushforward_isolated} about mapping isolated points. This section relies on some results about the divisor and Picard schemes of geometrically disconnected varieties, the proofs of which are given in Appendix \ref{sec:divisor_picard_schemes}.

Section \ref{sec:group_theory} contains a collection of preliminary results on group products and the profinite group $\GL{\Zhat}$, which will be used intermittently throughout the remainder of the paper. As such, the reader is invited to skip this section on first reading, and return to the results therein as and when they are needed.

Section \ref{sec:modular_curves} contains a detailed treatment of the modular curves $X_{H}$, with a particular emphasis on the moduli interpretation and its relationship with the group theory of $H$. While most of this is well-known, and can be found throughout the literature, this section serves as a central reference for the results and notation used in remainder of the paper, some of which is slightly non-standard. We also draw attention to Section \ref{sec:modular_curves:closed_points_as_double_cosets}, which provides a purely group-theoretic framework for working with points on the modular curves $X_{H}$ with given $j$-invariant, and may be of independent interest.

In Section \ref{sec:isolated_points_modular_curves}, the results of Section \ref{sec:isolated_divisors} and Section \ref{sec:modular_curves} are brought together to study isolated points on modular curves. In particular, we give the proof of the single-sink theorem, as well as the other results which form the backbone of our method for finding isolated points, as described above.

Finally, we conclude the paper with the two applications of our method described above. Section \ref{sec:level_7} is devoted to the classification of the isolated points on the modular curves of level 7, while Section \ref{sec:isolated_points_x0} focuses on the modular curves $X_{0}(n)$.

\subsection{Notation} \label{sec:introduction:notation}

Given a number field $k$, we denote by $G_{k}$ the absolute Galois group $\Gal(\bar{k}/k)$. For a variety $X$ over $k$, and a field extension $K/k$, we denote $X_{K}$ the base change $X \times_{k} \Spec K$. Throughout, we restrict ourselves to using left actions of groups on sets. This signifies, for instance, that the group $\GL{R}$ acts on $R^{2}$ as column vectors, for any commutative ring $R$. Moreover, the product $fg$ of elements of $\GL{R}$ corresponds to the composition $f \circ g$, when $f$ and $g$ are viewed as automorphisms of $R^{2}$. Similarly, the product $\sigma \eta$ of elements of $G_{k}$ corresponds to the composition $\sigma \circ \eta$ of automorphisms. Additional group-theoretic notation, namely for the group $\GL{\Zhat}$, is given in Section \ref{sec:group_theory}.

\subsection{Code}

The computations in Sections \ref{sec:level_7} and \ref{sec:isolated_points_x0} were performed using the computer algebra system \texttt{Magma} \cite{magma}. The code can be found in
the following GitHub repository:
\begin{center}
	\url{https://github.com/kenjiterao/maps-between-isolated-points}
\end{center}

\subsection{Acknowledgments}

This work was undertaken under the supervision of Samir Siksek, whom the author thanks for their continued guidance and support. The author also thanks Maarten Derickx, David Holmes, James Rawson and David Zywina for helpful and insightful conversations which have greatly influenced the direction of this work, and the referees for useful comments on earlier versions of this work. This work was supported by the Additional Funding Programme for Mathematical Sciences, delivered by EPSRC (EP/V521917/1) and the Heilbronn Institute for Mathematical Research. For the purpose of open access, the author has applied a Creative Commons Attribution (CC-BY) licence to any Author Accepted Manuscript version arising from this submission.
	\section{Isolated divisors} \label{sec:isolated_divisors}

As explained in Section \ref{sec:introduction:mapping_isolated_points}, throughout this paper we will require a notion of isolated points on geometrically disconnected curves. We dedicate this section to defining such a notion, as well as proving many analogues of results known for isolated points on geometrically connected curves.

This is achieved by extending, in Section \ref{sec:isolated_divisors:definition}, the existing notion of isolated points to effective divisors on smooth, projective, possibly geometrically disconnected varieties over number fields. In Section \ref{sec:isolated_divisors:stein_factorization}, the Stein factorization is shown to give a tight link between isolated divisors on geometrically disconnected and geometrically connected varieties. This explains why the behavior of isolated divisors on geometrically disconnected varieties is so similar to that of isolated divisors on geometrically connected varieties. In Section \ref{sec:isolated_divisors:morphisms}, it is shown that isolated divisors are well-behaved under pushforwards and pullbacks of finite morphisms. This justifies the choice of isolated divisors rather than isolated points as the primary object of study. Finally, in Section \ref{sec:isolated_divisors:isolated_points}, we restrict to the case of isolated points on curves, making use of the notion of degrees of points to reinterpret the results from the previous sections. This culminates with the proof of the mapping theorems described in Section \ref{sec:introduction:mapping_isolated_points}, which will be central to the remainder of the paper.

We set the following notation, which will persist throughout this section. Let $X$ be a smooth projective variety over a number field $k$, that is to say, a smooth, projective, integral scheme over $k$. Let $\Div(X/k)$ be the set of relative effective Cartier divisors on $X/k$; in other words, the set of effective Cartier divisors $D \subset X$ such that the morphism $D \to \Spec k$ is flat. Note that, by \cite[Exercise 3.5]{kleiman2005}, the set $\Div(X/k)$ is a commutative monoid. By \cite[Theorem 3.7]{kleiman2005}, there exists a $k$-scheme $\DDiv_{X/k}$ such that $\DDiv_{X/k}(T) = \Div(X_{T}/T)$ for all $k$-schemes $T$, which we call the divisor scheme of $X$.

\begin{remark} \label{rmk:isolated_divisors:divisor_cycle_equivalence}
	The commutative monoid $\Div(X/k)$ can equally be described in terms of cycles on $X$ of codimension one. Indeed, as all schemes over the spectrum of a field are flat, $\Div(X/k)$ is equal to the monoid of effective Cartier divisors on $X$. Moreover, since $X$ is a smooth scheme over the number field $k$, the monoid of effective Cartier divisors on $X$ is isomorphic to the monoid of effective cycles on $X$ of codimension one, by \cite[Th\'eor\`eme 21.6.9]{egaiv4}. Note that, as $X$ is Noetherian, the latter is isomorphic to the free commutative monoid generated by the irreducible closed subschemes of $X$ of codimension one, by \cite[21.6.2]{egaiv4}. In particular, each element of $\Div(X/k)$ is represented by a unique positive formal sum $\sum_{i = 1}^{n} n_{i} Z_{i}$, where each $Z_{i}$ is an irreducible closed subscheme of $X$ of codimension one.
\end{remark}

Recall that the Picard group $\Pic(X)$ is the set of isomorphism classes of invertible sheaves $\mathcal{L}$ on $X$. Define the relative Picard functor $\Pic_{X/k}$ by
\[
	\Pic_{X/k}(T) = \Pic(X_{T})/\pi_{T}^{\ast}\Pic(T),
\]
for all $k$-schemes $T$, where $\pi_{T} : X_{T} \to T$ denotes the structure map. The sheafification of the relative Picard functor under the fppf-topology is representable by the Picard scheme $\PPic_{X/k}$, by \cite[Corollaire 6.6]{fga5}. Note that, by the properties of sheafification, there is a natural transformation $\Pic_{X/k} \to \PPic_{X/k}$. By abuse of notation, for any invertible sheaf $\mathcal{L}$ on $X$, we denote $[\mathcal{L}] \in \PPic_{X/k}(k)$ the image of the equivalence class $[\mathcal{L}] \in \Pic_{X/k}(k)$.

The Picard scheme $\PPic_{X/k}$ is a group scheme, and we let $\PPic_{X/k}^{0}$ denote the connected component of the identity. The latter is an abelian variety by \cite[Remark 5.6]{kleiman2005}. More generally, for any invertible sheaf $\mathcal{L}$ of $X$, we denote by $\PPic_{X/k}^{\mathcal{L}}$ the connected component of $\PPic_{X/k}$ containing the point $[\mathcal{L}] \in \PPic_{X/k}(k)$. Note that, since $\PPic_{X/k}$ is a group scheme, $\PPic_{X/k}^{\mathcal{L}}$ is isomorphic to $\PPic_{X/k}^{0}$ for all invertible sheaves $\mathcal{L}$ of $X$. By \cite[Proposition 5.10]{kleiman2005}, the connected component $\PPic_{X/k}^{\mathcal{L}}$ parametrizes invertible sheaves algebraically equivalent to $\mathcal{L}$; that is to say, two invertible sheaves $\mathcal{L}$ and $\mathcal{L}'$ on $X$ are algebraically equivalent if and only if $\PPic_{X/k}^{\mathcal{L}} = \PPic_{X/k}^{\mathcal{L}'}$.

Consider the natural transformation $\DDiv_{X/k} \to \Pic_{X/k}$ defined by
\begin{align*}
	\DDiv_{X/k}(T) & \to \Pic_{X/k}(T)             \\
	D              & \mapsto [\mathcal{O}_{X}(D)],
\end{align*}
for all $k$-schemes $T$. Composing this with the natural transformation $\Pic_{X/k} \to \PPic_{X/k}$, we obtain a morphism of schemes $\AAbel_{X/k} : \DDiv_{X/k} \to \PPic_{X/k}$, called the Abel map. The Abel map is a proper morphism of schemes by Theorem \ref{thm:divisor_picard_schemes:abel_proper}. In particular, the image $\AAbel_{X/k}(\DDiv_{X/k})$ is a closed subscheme of $\PPic_{X/k}$, which we denote $\W_{X/k}$.

For any relative effective Cartier divisor $D \in \DDiv_{X/k}(k)$ on $X$, we let $\DDiv_{X/k}^{D}$ be the preimage $\AAbel_{X/k}^{-1} \! \left(\PPic_{X/k}^{\O_{X}(D)} \right)$. Since the Abel map is of finite type, $\DDiv_{X/k}^{D}$ is a finite union of connected components of $\DDiv_{X/k}$. As is the case for the Picard scheme, the set $\DDiv_{X/k}^{D}(k)$ is in bijection with the set of relative effective Cartier divisors on $X$ algebraically equivalent to $D$. As before, the image $\AAbel_{X/k}(\DDiv_{X/k}^{D})$ is a closed subscheme of $\PPic_{X/k}^{\O_{X}(D)}$, which we denote $\W_{X/k}^{D}$.

Finally, let $X \to \Spec K \to \Spec k$ be the Stein factorization of the structure map $X \to \Spec k$, where $K/k$ is a finite extension of number fields. The field $K$ can be explicitly described as $H^{0}(X, \O_{X})$, or equivalently, as the intersection $k(X) \cap \bar{k}$, where $k(X)$ denotes the function field of $X$. The map $X \to \Spec K$ makes $X$ into a smooth, projective, geometrically integral scheme over the number field $K$, and we use analogous notation to denote the divisor and Picard schemes of $X/K$. Note that, if $X \to \Spec k$ is geometrically integral, then $K = k$.

\subsection{Isolated divisors} \label{sec:isolated_divisors:definition}

We begin by defining the notion of isolated divisors on the variety $X$. This extends the definition of isolated points given in \cite[Definition 4.1]{bourdon2019} to arbitrary divisors, and removes the condition that the variety $X$ is one-dimensional and geometrically integral.

\begin{definition}
	Let $D \in \DDiv_{X/k}(k)$ be a relative effective Cartier divisor on $X$. We say that
	\begin{enumerate}
		\item $D$ is \textbf{\Pone-parametrized} if there exists $D' \neq D \in \DDiv_{X/k}(k)$ such that $\AAbel_{X/k}(D) = \AAbel_{X/k}(D')$, and \textbf{\Pone-isolated} otherwise.
		\item $D$ is \textbf{AV-parametrized} if there exists a positive rank abelian subvariety $A \subset \PPic_{X/k}^{0}$ such that
		      \[
			      \AAbel_{X/k}(D) + A \subset \W_{X/k},
		      \]
		      and \textbf{AV-isolated} otherwise.
		\item $D$ is \textbf{isolated} if it is both \Pone-isolated and AV-isolated.
	\end{enumerate}
\end{definition}

The notion of isolated points on geometrically irreducible curves defined in \cite{bourdon2019} is motivated by the observation that a curve $C$ over a number field has infinitely many closed points of degree at most $d$ if and only there exists a non-isolated point of such degree. This property generalizes to isolated divisors, as the following theorem demonstrates.

\begin{theorem} \label{thm:isolated_divisors:non_isolated_infinite}
	Let $D \in \DDiv_{X/k}(k)$ be a relative effective Cartier divisor on $X$. Then, there exist infinitely many relative effective Cartier divisors on $X$ algebraically equivalent to $D$ if and only if there exists a non-isolated such divisor.
\end{theorem}

The proof of this result proceeds in much the same way as the proof of \cite[Theorem 4.2]{bourdon2019}. However, there are a few nuances which apply in the case of geometrically disconnected varieties, the proofs of which can be found in Appendix \ref{sec:divisor_picard_schemes}. Firstly, the fibers of the Abel map are no longer given by projective spaces, as is the case for geometrically integral varieties. However, these fibers can still be described in terms of Weil restrictions.

\begin{theorem}[{see Theorem \ref{thm:divisor_picard_schemes:abel_fiber}}] \label{thm:isolated_divisors:abel_fiber}
	Let $X, k$ and $K$ be as above, and let $D \in \DDiv_{X/k}(k)$ be a relative effective Cartier divisor on $X$. Then, the fiber $\AAbel_{X/k}^{-1}(\AAbel_{X/k}(D))$ is isomorphic to $\Res_{K/k}(\mathbb{P}_{K}^{n-1})$, where $n = \dim_{K} H^{0}(X, \O_{X}(D))$.
\end{theorem}

Moreover, the image of the $k$-rational points of $\DDiv_{X/k}$ under the Abel map cannot be explicitly determined. However, we can still determine some of its properties, as the following theorem shows.

\begin{theorem}[{see Theorem \ref{thm:divisor_picard_schemes:abel_rational_image}}] \label{thm:isolated_divisors:abel_rational_image}
	Let $X, k$ and $K$ be as above. Then there exists a subgroup $G \leq \PPic_{X/k}(k)$ such that
	\[
		\AAbel_{X/k}(\DDiv_{X/k}(k)) = \AAbel_{X/k}(\DDiv_{X/k})(k) \cap G.
	\]
	Moreover, the quotient group $\PPic_{X/k}(k) / G$ is torsion.
\end{theorem}

As is the case for \cite[Theorem 4.2]{bourdon2019}, the proof of Theorem \ref{thm:isolated_divisors:non_isolated_infinite} relies heavily on Faltings's theorem on subvarieties of abelian varieties, which we state here.

\begin{theorem}[\cite{faltings1994}, Main theorem] \label{thm:isolated_divisors:faltings}
	Let $A$ be an abelian variety over a number field $k$, and $X \subset A$ a closed subvariety. Then, there exist finitely many abelian subvarieties $A_{1}, \dots, A_{m} \subset A$, and finitely many points $y_{1}, \dots, y_{m} \in X(k)$ such that $y_{i} + A_{i} \subset X$ for all $i$, and
	\[
		X(k) = \bigcup_{i = 1}^{m} y_{i} + A_{i}(k).
	\]
\end{theorem}

Equipped with these results, we can now give the proof of Theorem \ref{thm:isolated_divisors:non_isolated_infinite}. As indicated above, this proof follows just as the proof of \cite[Theorem 4.2]{bourdon2019}, with some minor changes to accommodate the aforementioned nuances.

\begin{proof}[Proof of Theorem \ref{thm:isolated_divisors:non_isolated_infinite}]
	Suppose first that there exist infinitely many relative effective Cartier divisors on $X/K$ algebraically equivalent to $D$, or equivalently, $\DDiv_{X/k}^{D}(k)$ is infinite. If the Abel map $\AAbel_{X/k} : \DDiv_{X/k} \to \PPic_{X/k}$ is not injective on $\DDiv_{X/k}^{D}(k)$, then there exists two distinct effective Cartier divisors $E, E' \in \DDiv_{X/k}^{D}(k)$ such that $\AAbel_{X/k}(E) = \AAbel_{X/k}(E')$. In particular, $E$ is \Pone-parametrized, and so is a non-isolated relative effective Cartier divisor on $X/k$ algebraically equivalent to $D$.

	On the other hand, suppose that the Abel map $\AAbel_{X/k} : \DDiv_{X/k} \to \PPic_{X/k}$ is injective on $\DDiv_{X/k}^{D}(k)$. The image $\W_{X/k}^{D}$ is a closed subscheme of the connected component $\PPic_{X/k}^{\O_{X}(D)}$, the latter being isomorphic to the abelian variety $\PPic_{X/k}^{0}$. As the Abel map is proper, the image $\W_{X/k}^{D}$ is Noetherian, and so has finitely many irreducible components. Therefore, by Theorem \ref{thm:isolated_divisors:faltings}, there exist finitely many abelian subvarieties $A_{1}, \dots, A_{m} \subset \PPic_{X/k}^{0}$, and finitely many points $y_{1}, \dots, y_{m} \in \W_{X/k}^{D}(k)$ such that $y_{i} + A_{i} \subset \W_{X/k}^{D}$ for all $i$, and
	\[
		\W_{X/k}^{D}(k) = \bigcup_{i = 1}^{m} y_{i} + A_{i}(k).
	\]
	Since the Abel map $\AAbel_{X/k}$ is injective on $\DDiv_{X/k}^{D}(k)$, we know that the set $\AAbel_{X/k}(\DDiv_{X/k}^{D}(k))$ is infinite. By the above, there exists a translate $y_{i} + A_{i}(k)$ containing infinitely many points of $\AAbel_{X/k}(\DDiv_{X/k}^{D}(k))$. In particular, $A_{i}$ has positive rank, and we may assume that $y_{i} \in \AAbel_{X/k}(\DDiv_{X/k}^{D}(k))$. Therefore, there exists some relative effective Cartier divisor $E \in \DDiv_{X/k}^{D}(k)$ such that $y_{i} = \AAbel_{X/k}(E)$. Since $y_{i} + A_{i} \subset \W_{X/k}^{D}$, it follows that
	\[
		\AAbel_{X/k}(E) + A_{i} \subset \W_{X/k},
	\]
	and so $E$ is AV-parametrized. Therefore, there exists a non-isolated relative effective Cartier divisor on $X/k$ algebraically equivalent to $D$.

	Suppose now that there exists a non-isolated relative effective Cartier divisor $E$ on $X/k$ algebraically equivalent to $D$. Suppose first that $E$ is \Pone-parametrized. Therefore, there exists $E' \neq E \in \DDiv_{X/k}(k)$ such that $\AAbel_{X/k}(E) = \AAbel_{X/k}(E')$. Consider the fiber $F = \AAbel_{X/k}^{-1}(\AAbel_{X/k}(E))$. By definition, $E$ and $E'$ are distinct $k$-points on $F$. By Theorem \ref{thm:isolated_divisors:abel_fiber}, we know that $F = \Res_{K/k}(\mathbb{P}^{n-1}_{K})$ for some $n \geq 0$. In particular, since $F(k)$ has two distinct elements, it follows that $n \geq 2$ and so $F(k)$ is infinite. By definition, $F(k) \subset \DDiv_{X/k}^{D}(k)$, and so there are infinitely many relative effective Cartier divisors on $X/k$ algebraically equivalent to $D$.

	Suppose now that $E$ is AV-parametrized, so there exists a positive rank abelian subvariety $A \subset \PPic_{X/k}^{0}$ such that
	\[
		\AAbel_{X/k}(E) + A \subset \W_{X/k}.
	\]
	Note that, since $A \subset \PPic_{X/k}^{0}$, we have $\AAbel_{X/k}(E) + A \subset \W_{X/k}^{E}$. By Theorem \ref{thm:isolated_divisors:abel_rational_image}, there is a subgroup $G \subset \PPic_{X/k}(k)$ such that
	\[
		\AAbel_{X/k}(\DDiv_{X/k}(k)) = \W_{X/k}(k) \cap G.
	\]
	Moreover, the quotient group $\PPic_{X/k}(k) / G$ is torsion. By definition, we have $\AAbel_{X/k}(E) \in G$, and so
	\[
		\AAbel_{X/k}(E) + (A(k) \cap G) \subset \W_{X/k}^{E}(k) \cap G = \AAbel_{X/k}(\DDiv_{X/k}^{E}(k)).
	\]
	The quotient group $A(k) / (A(k) \cap G)$ is a finitely-generated subgroup of the torsion group $\PPic_{X/k}(k) / G$, and so is finite. In particular, since $A(k)$ has positive rank, the group $A(k) \cap G$ is infinite. Therefore, by the above, the set $\DDiv_{X/k}^{E}(k)$ must be infinite. Since $E$ is algebraically equivalent to $D$, we have $\DDiv_{X/k}^{D} = \DDiv_{X/k}^{E}$. Therefore, there are infinitely many relative effective Cartier divisors on $X/k$ algebraically equivalent to $D$.
\end{proof}

\subsection{Isolated divisors and Stein factorization} \label{sec:isolated_divisors:stein_factorization}

As evidenced in the appendix, the divisor and Picard schemes of a geometrically disconnected variety are closely related to those of its Stein factorization. This correspondence extends to isolated divisors, as the following theorem shows.

\begin{theorem} \label{thm:isolated_divisors:stein_factorization_isolated}
	Let $D \in \DDiv_{X/k}(k)$ be a relative effective Cartier divisor on $X / k$, and note that $D$ also defines a relative effective Cartier divisor $D \in \DDiv_{X/K}(K)$ on $X / K$. Then, the following hold:
	\begin{enumerate}
		\item $D$ is \Pone-isolated if and only if the same divisor $D$ on $X / K$ is \Pone-isolated.
		\item $D$ is AV-isolated if and only if the same divisor $D$ on $X / K$ is AV-isolated.
	\end{enumerate}
	In particular, $D$ is isolated if and only if the same divisor $D$ on $X / K$ is isolated.
\end{theorem}

\begin{proof}
	By Theorem \ref{thm:divisor_picard_schemes:divisor_weil_restriction}, we have a commutative diagram
	\[\begin{tikzcd}
			\DDiv_{X/k}(k) \arrow[d, "\psi_{k}"] \arrow[r, "\AAbel_{X/k}"] & \PPic_{X/k}(k) \arrow[d, "\varphi_{k}"] \\
			\DDiv_{X/K}(K) \arrow[r, "\AAbel_{X/K}"] & \PPic_{X/K}(K),
		\end{tikzcd}\]
	where both $\varphi_{k}$ and $\psi_{k}$ are bijections. Moreover, the proof of said theorem shows that the map $\psi_{k}$ takes a relative effective Cartier divisor $D$ on $X/k$ to the same relative effective Cartier divisor $D$ on $X/K$. Therefore, $D \in \DDiv_{X/k}(k)$ is \Pone-isolated if and only if $D \in \DDiv_{X/K}(K)$ is \Pone-isolated.

	Suppose that $D$ on $X / K$ is AV-parametrized, so there exists a positive rank abelian subvariety $A \subset \PPic_{X/K}^{0}$ such that
	\[
		\AAbel_{X/K}(D) + A \subset \W_{X/K}.
	\]
	To ease notation, denote by $T_{D/K} : \PPic_{X/K} \to \PPic_{X/K}$ the translation by $\AAbel_{X/K}(D)$. Therefore, the above becomes $T_{D/K}(A) \subset \W_{X/K}$. Applying the $\Res_{K/k}$ functor, we obtain
	\[
		\Res_{K/k}(T_{D/K}(A)) \subset \Res_{K/k} \W_{X/K}.
	\]
	By Theorem \ref{thm:divisor_picard_schemes:divisor_weil_restriction}, we have a commutative diagram
	\[\begin{tikzcd}
		\DDiv_{X/k} \arrow[d, "\psi"] \arrow[rr, "\AAbel_{X/k}"] & & \PPic_{X/k} \arrow[d, "\varphi"] \\
		\Res_{K/k} \DDiv_{X/K} \arrow[rr, "\Res_{K/k}\AAbel_{X/K}"] & & \Res_{K/k} \PPic_{X/K},
	\end{tikzcd}\]
	where $\psi$ is an isomorphism of schemes and $\varphi$ an isomorphism of group schemes. Since $K/k$ is a finite separable extension, it follows from Theorem \ref{thm:divisor_picard_schemes:weil_restriction_surjectivity} that
	\begin{align*}
		\Res_{K/k} \W_{X/K} & = \Res_{K/k} \AAbel_{X/K} (\Res_{K/k} \DDiv_{X/K}) \\
		& = \varphi(\W_{X/k}).
	\end{align*}
	By \cite[Proposition A.5.1]{conrad2015}, we know that $\Res_{K/k} \PPic_{X/K}$ is a group scheme, with group structure induced from the group structure of $\PPic_{X/K}$. Using the explicit description of $\varphi$ from the proof of Theorem \ref{thm:divisor_picard_schemes:picard_weil_restriction}, it follows that $\Res_{K/k} T_{D/K} = \varphi \circ T_{D/k} \circ \varphi^{-1}$, where $T_{D/k} : \PPic_{X/k} \to \PPic_{X/k}$ is the translation by $\AAbel_{X/k}(D)$. Therefore, we have
	\begin{align*}
		\Res_{K/k}(T_{D/K}(A)) & = (\Res_{K/k} T_{D/K})(\Res_{K/k} A) \\
		& = \varphi(T_{D/k}(\varphi^{-1}(\Res_{K/k} A))).
	\end{align*}
	Thus, we obtain that
	\[
		\varphi(T_{D/k}(\varphi^{-1}(\Res_{K/k} A))) \subset \varphi(\W_{X/k}).
	\]
	Since $\varphi$ is an isomorphism of schemes, it follows that
	\[
		\AAbel_{X/k}(D) + \varphi^{-1}(\Res_{K/k} A) \subset \W_{X/k}.
	\]
	By \cite[Proposition A.5.1]{conrad2015} and \cite[Propositions 6.2.9, 6.2.10 and 6.3.7]{spec}, the Weil restriction $\Res_{K/k} A$ is an abelian variety. Moreover, by definition, the group $(\Res_{K/k} A)(k)$ is isomorphic to $A(K)$, and so $\Res_{K/k} A$ has positive rank. Finally, since $\Res_{K/k} A$ is connected, it follows that $\Res_{K/k} A \subset \PPic_{X/k}^{0}$. Therefore, the divisor $D \in \DDiv_{X/k}(k)$ on $X/k$ is AV-parametrized.

	Suppose now that $D$ on $X/k$ is AV-parametrized. Therefore, there exists a positive rank abelian subvariety $A \subset \PPic_{X/k}^{0}$ such that
	\[
		\AAbel_{X/k}(D) + A \subset \W_{X/k}.
	\]
	As before, we can write this as $T_{D/k}(A) \subset \W_{X/k}$. Composing with $\varphi : \PPic_{X/k} \to \Res_{K/k} \PPic_{X/K}$, we obtain that
	\[
		\varphi(T_{D/k}(A)) \subset \varphi(\W_{X/k}).
	\]
	As was shown above, we know that $\varphi(\W_{X/k}) = \Res_{K/k} \W_{X/K}$ and $\varphi \circ T_{D/k} = \Res_{K/k} T_{D/K} \circ \varphi$. Therefore, we have
	\[
		\Res_{K/k} T_{D/K}(\varphi(A)) \subset \Res_{K/k} \W_{X/k}.
	\]
	Base changing to $K$, we obtain that
	\[
		(\Res_{K/k} T_{D/K})_{K}((\varphi(A))_{K}) \subset (\Res_{K/k} \W_{X/k})_{K}.
	\]
	Recall that, for any $K$-scheme $Z$, there exists a morphism $q_{Z} : (\Res_{K/k} Z)_{K} \to Z$. Moreover, this morphism is functorial in $Z$; that is to say, for any morphism $f : Z \to W$ of $K$-schemes, we have a commutative diagram
	\[\begin{tikzcd}
			(\Res_{K/k} Z)_{K}  \arrow[r, "q_{Z}"] \arrow[d, "(\Res_{K/k} f)_{K}"'] & Z \arrow[d, "f"] \\
			(\Res_{K/k} W)_{K}  \arrow[r, "q_{W}"] & W.
		\end{tikzcd}\]
	Applying $q_{\PPic_{X/K}}$ and the functoriality of $q$, we find that
	\[
		T_{D/K}(q_{\PPic_{X/K}}((\varphi(A))_{K})) \subset \W_{X/k}.
	\]
	Since both $q_{\PPic_{X/K}}$ and $\varphi$ are morphisms of group schemes, by \cite[Expos\'e VI\textsubscript{B}, Proposition 1.2]{sga3_2a}, the image $q_{\PPic_{X/K}}((\varphi(A))_{K})$ is an abelian variety. Since $\varphi$ is an isomorphism, we know that $\varphi(A)(k) \cong A(k)$ and so the former is infinite. Moreover, by definition of $q$, any two distinct points of $\varphi(A)(k)$ give rise to distinct points of $q_{\PPic_{X/K}}((\varphi(A))_{K})(K)$. Therefore, the set $q_{\PPic_{X/K}}((\varphi(A))_{K})(K)$ is also infinite, and so $q_{\PPic_{X/K}}((\varphi(A))_{K})$ has positive rank. As previously, the variety $q_{\PPic_{X/K}}((\varphi(A))_{K})$ is connected, and so $q_{\PPic_{X/K}}((\varphi(A))_{K}) \subset \PPic_{X/K}^{0}$. Thus, the divisor $D \in \DDiv_{X/K}(K)$ on $X/K$ is AV-parametrized.
\end{proof}

\subsection{Maps between isolated divisors} \label{sec:isolated_divisors:morphisms}

While isolated divisors are defined solely in terms of a single variety $X$, the bulk of the results in this paper stem from considering morphisms of varieties $f : X \to Y$. Indeed, as we will now show, isolated divisors are well-behaved under pushforwards and pullbacks of certain morphisms.

Throughout the remainder of this section, we let $f : X \to Y$ be a finite locally free morphism of degree $d \geq 1$ between smooth, projective varieties over the number field $k$. We employ the same notation for the divisor and Picard schemes of $Y$ as was introduced for $X$ earlier.

The morphism $f : X \to Y$ induces a group homomorphism $f^{\ast} : \Pic(Y_{T}) \to \Pic(X_{T})$ for all $k$-schemes $T$, which gives rise to a homomorphism $f^{\ast} : \Pic_{Y/k}(T) \to \Pic_{X/k}(T)$. The latter homomorphism is functorial in $T$, and thus defines a natural transformation $f^{\ast} : \Pic_{Y/k} \to \Pic_{X/k}$ between relative Picard functors. By the universality of sheafification, and the Yoneda lemma, we obtain a morphism of $k$-group schemes $f^{\ast} : \PPic_{Y/k} \to \PPic_{X/k}$.

Since the morphism $f : X \to Y$ is flat, by \cite[Proposition 21.4.5]{egaiv4}, we obtain a pullback map $f^{\ast} : \DDiv_{Y/k}(T) \to \DDiv_{X/k}(T)$ for all $k$-schemes $T$. As for the Picard scheme, this map is functorial in $T$, and so defines a morphism of $k$-schemes $f^{\ast} : \DDiv_{Y/k} \to \DDiv_{X/k}$. Moreover, by \cite[\nopp 21.4.2.1]{egaiv4}, there is a commutative diagram
\begin{equation}\begin{tikzcd}
		\DDiv_{Y/k} \arrow[r, "\AAbel_{Y/k}"] \arrow[d, "f^{\ast}"] & \PPic_{Y/k} \arrow[d, "f^{\ast}"] \\
		\DDiv_{X/k} \arrow[r, "\AAbel_{X/k}"] & \PPic_{X/k}.
	\end{tikzcd} \label{eq:isolated_divisors:divisor_picard_pullback}\end{equation}

We define the norm and pushforward morphisms similarly. Namely, as $f : X \to Y$ is finite locally free, by \cite[Proposition 6.5.6]{egaii}, there exists a group homomorphism $\Norm_{X_{T}/Y_{T}} : \Pic(X_{T}) \to \Pic(Y_{T})$ for all $k$-schemes $T$. By \cite[Proposition 6.5.8]{egaii}, this gives rise to a homomorphism $\Norm_{X/Y} : \Pic_{X/k}(T) \to \Pic_{Y/k}(T)$, which is functorial in $T$. Therefore, as before, we obtain a morphism of $k$-group schemes $\Norm_{X/Y} : \PPic_{X/k} \to \PPic_{Y/k}$.

Similarly, since the morphism $f : X \to Y$ is finite locally free, by \cite[\nopp 21.5.5]{egaiv4}, there exists a map $f_{\ast} : \DDiv_{X/k}(T) \to \DDiv_{Y/k}(T)$ for all $k$-schemes $T$. This map is functorial in $T$, and so defines a morphism of $k$-schemes $f_{\ast} : \DDiv_{X/k} \to \DDiv_{Y/k}$. By construction, we also obtain the following commutative diagram:
\begin{equation}\begin{tikzcd}
		\DDiv_{X/k} \arrow[r, "\AAbel_{X/k}"] \arrow[d, "f_{\ast}"] & \PPic_{X/k} \arrow[d, "\Norm_{X/Y}"] \\
		\DDiv_{Y/k} \arrow[r, "\AAbel_{Y/k}"] & \PPic_{Y/k}.
	\end{tikzcd} \label{eq:isolated_divisors:divisor_picard_pushforward}\end{equation}

The notion of isolated divisors behaves nicely with respect to the pushforward and pullback maps defined above. For instance, non-isolated divisors are preserved under pullback.

\begin{theorem} \label{thm:isolated_divisors:pullback_isolated}
	Let $f : X \to Y$ be as above, and let $D \in \DDiv_{Y/k}(k)$ be a relative effective Cartier divisor on $Y$. Then, the following hold:
	\begin{enumerate}
		\item Suppose that $D$ is \Pone-parametrized. Then $f^{\ast}(D)$ is \Pone-parametrized.
		\item Suppose that $D$ is AV-parametrized. Then $f^{\ast}(D)$ is AV-parametrized.
	\end{enumerate}
	In particular, if $f^{\ast}(D)$ is isolated, then $D$ is also isolated.
\end{theorem}

\begin{proof}
	Suppose first that $D$ is \Pone-parametrized. Therefore, there exists another relative effective Cartier divisor $D' \neq D \in \DDiv_{Y/k}(k)$ such that $\AAbel_{Y/k}(D) = \AAbel_{Y/k}(D')$. Since $f$ is faithfully flat, the morphism $f^{\ast} : \DDiv_{Y/k} \to \DDiv_{X/k}$ is injective on $k$-points, by \cite[Proposition 21.4.8]{egaiv4}. In particular, we have $f^{\ast}(D') \neq f^{\ast}(D)$. Moreover, by (\ref{eq:isolated_divisors:divisor_picard_pullback}), we have $\AAbel_{X/k}(f^{\ast}(D)) = \AAbel_{X/k}(f^{\ast}(D'))$. Therefore, $f^{\ast}(D)$ is \Pone-parametrized.

	Suppose now that $D$ is AV-parametrized. Then, there exists a positive rank abelian subvariety $A \subset \PPic_{Y/k}^{0}$ such that $\AAbel_{Y/k}(D) + A \subset \W_{Y/k}$. Since $f^{\ast}$ is a morphism of group schemes, we obtain that
	\[
		f^{\ast}(\AAbel_{Y/k}(D)) + f^{\ast} A \subset f^{\ast}(\W_{Y/k}).
	\]
	By (\ref{eq:isolated_divisors:divisor_picard_pullback}), it follows that
	\[
		\AAbel_{X/k}(f^{\ast}(D)) + f^{\ast} A \subset \W_{X/k}.
	\]
	By Theorem \ref{thm:divisor_picard_schemes:abel_rational_image}, there exists a subgroup $G \leq \PPic_{Y/k}(k)$ such that
	\[
		\AAbel_{Y/k}(\DDiv_{Y/k}(k)) = \W_{Y/k}(k) \cap G,
	\]
	and the quotient group $\PPic_{Y/k}(k) / G$ is torsion. As the group $A(k)$ is finitely generated by the Mordell-Weil theorem, it follows that $A(k) / (A(k) \cap G) \leq \PPic_{Y/k}(k) / G$ is a finite group. Since $A$ has positive rank, the group $A(k) \cap G$ must therefore be infinite.

	Consider a point $x \in A(k) \cap G$. Note that, since $\AAbel_{Y/k}(D) \in G$, we have $\AAbel_{Y/k}(D) + x \in G$. Moreover, since $\AAbel_{Y/k}(D) + A \subset \W_{Y/k}$, it follows that $\AAbel_{Y/k}(D) + x \in \W_{Y/k}(k) \cap G$. Thus, there exists a relative effective Cartier divisor $E \in \DDiv_{Y/k}(k)$ on $Y$ such that $\AAbel_{Y/k}(E) = \AAbel_{Y/k}(D) + x$. By \cite[Proposition 21.5.6]{egaiv4}, since $f$ is finite locally free of degree $d$, we know that $f_{\ast}(f^{\ast}(D)) = dD$ and $f_{\ast}(f^{\ast}(E)) = dE$. Since $\AAbel_{Y/k} : \DDiv_{Y/k}(k) \to \PPic_{Y/k}(k)$ is a monoid homomorphism, it follows that
	\[
		\AAbel_{Y/k}(f_{\ast}(f^{\ast}(E))) - \AAbel_{Y/k}(f_{\ast}(f^{\ast}(D))) = d x.
	\]
	By (\ref{eq:isolated_divisors:divisor_picard_pullback}) and (\ref{eq:isolated_divisors:divisor_picard_pushforward}), and the fact that both $\Norm_{X/Y}$ and $f^{\ast}$ are morphisms of group schemes, it follows that $\Norm_{X/Y}(f^{\ast}(x)) = dx$. In particular, the kernel of the homomorphism $f^{\ast} : A(k) \cap G \to \PPic_{X/k}(k)$ is contained in the set of $d$-torsion points $A(k)[d]$. The latter is finite as $A(k)$ is finitely generated. Therefore, since the group $A(k) \cap G$ is infinite, so is the image $f^{\ast}(A(k) \cap G)$. Thus, the abelian variety $f^{\ast}(A)$ has positive rank, and so $f^{\ast}(D)$ is AV-parametrized.
\end{proof}

In order to prove a similar result for the pushforward map, we require the following result on finiteness of preimages of the pushforward map.

\begin{lemma} \label{thm:isolated_divisors:pushforward_preimage_finite}
	Let $f : X \to Y$ be as above, and let $D$ be a relative effective Cartier divisor on $Y$. Then there are finitely many relative effective Cartier divisors $E$ on $X$ such that $f_{\ast}(E) = D$.
\end{lemma}

\begin{proof}
	Let $E$ be a relative effective Cartier divisor on $X$ such that $f_{\ast}(E) = D$. By Remark \ref{rmk:isolated_divisors:divisor_cycle_equivalence}, $D$ can be identified with a finite formal sum $\sum_{i = 1}^{n} n_{i} \cdot Z_{i}$ of codimension-one irreducible closed subschemes $Z_{i}$ of $Y$. Similarly, we can identify $E$ with a finite formal sum $\sum_{j = 1}^{m} m_{j} \cdot W_{j}$ of codimension-one irreducible closed subschemes $W_{j}$ of $X$.

	By \cite[Proposition 21.10.17]{egaiv4}, and the fact that the monoid of effective codimension-one cycles on $Y$ is a free commutative monoid, it follows that for all $j$, there exists $i$ such that $W_{j} \subseteq f^{-1}(Z_{i})$ and $m_{j} \leq n_{i}$. As $f$ is finite locally free, the fiber $f^{-1}(Z_{i})$ contains finitely many codimension-one irreducible closed subschemes of $X$, for all $i$. Therefore, there are only finitely many formal sums $\sum_{j = 1}^{m} m_{j} \cdot W_{j}$ satisfying the above property, and so finitely many such relative effective Cartier divisors $E$.
\end{proof}

\begin{remark}
	Surprisingly, the above theorem does not hold if $X$ is not normal. As an example, consider the morphism $f : \Spec(R) \to \mathbb{A}^{1}_{\Q}$ given by the ring map $\Q[x] \hookrightarrow \Q[x, y] / (x^2 + y^2) =: R$. The ideals $(ax + y)$ of $R$, for $a \in \Q$, give infinitely many distinct effective Cartier divisors on $\Spec(R)$, but the pushforward of each under $f$ is the ideal $(x^2)$ of $\Q[x]$.
\end{remark}

As is the case for the pullback morphism, non-isolated divisors are also preserved under the pushforward morphism.

\begin{theorem} \label{thm:isolated_divisors:pushforward_isolated}
	Let $f : X \to Y$ be as above, and let $D \in \DDiv_{X/k}(k)$ be a relative effective Cartier divisor on $X$. Then, the following hold:
	\begin{enumerate}
		\item Suppose that $D$ is \Pone-parametrized. Then $f_{\ast}(D)$ is \Pone-parametrized.
		\item Suppose that $D$ is AV-parametrized. Then $f_{\ast}(D)$ is not isolated.
	\end{enumerate}
	In particular, if $f_{\ast}(D)$ is isolated, then $D$ is also isolated.
\end{theorem}

\begin{remark}
	Unlike in the case of the pullback morphism, the pushforward of a relative effective Cartier divisor $D$ on $X$ which is AV-parametrized can itself be AV-isolated. For instance, if $X$ is a positive rank elliptic curve over a number field $k$, any $k$-rational point $x \in X(k)$ is AV-parametrized (when viewed as a Cartier divisor). However, the connected component $\PPic_{\mathbb{P}^{1}_{k}/k}^{0}$ of the Picard scheme of $\mathbb{P}^{1}_{k}$ consists of a single point, and so any divisor on $\mathbb{P}^{1}$ is AV-isolated. Therefore, the pushforward of any $k$-rational point $x \in X(k)$ along a morphism $f : X \to \mathbb{P}^{1}$ is AV-isolated, despite $x$ itself being AV-parametrized.
\end{remark}

\begin{proof}[Proof of Theorem \ref{thm:isolated_divisors:pushforward_isolated}]
	Suppose first that $D$ is \Pone-parametrized. By Theorem \ref{thm:divisor_picard_schemes:abel_fiber}, we know that the fiber $\AAbel_{X/k}^{-1}(\AAbel_{X/k}(D))$ is isomorphic to $\Res_{K/k}(\mathbb{P}^{n-1}_{K})$, for some $n \geq 0$. Since $D$ is \Pone-parametrized, the fiber $\AAbel_{X/k}^{-1}(\AAbel_{X/k}(D))$ contains at least two $k$-rational points. Therefore, we must have $n \geq 2$, and so the set $\AAbel_{X/k}^{-1}(\AAbel_{X/k}(D))(k)$ is infinite. On the other hand, by Lemma \ref{thm:isolated_divisors:pushforward_preimage_finite}, the set $\Sigma = \{E \in \DDiv_{X/k}(k) : f_{\ast}(E) = f_{\ast}(D)\}$ is finite. Thus, the set $\AAbel_{X/k}^{-1}(\AAbel_{X/k}(D))(k) \setminus \Sigma$ is non-empty, and there exists a relative effective Cartier divisor $D' \in \DDiv_{X/k}(k)$ on $X$ such that $\AAbel_{X/k}(D) = \AAbel_{X/k}(D')$ and $f_{\ast}(D) \neq f_{\ast}(D')$. By (\ref{eq:isolated_divisors:divisor_picard_pushforward}), it follows that $\AAbel_{Y/k}(f_{\ast}(D)) = \AAbel_{Y/k}(f_{\ast}(D'))$, and so $f_{\ast}(D)$ is \Pone-parametrized.

	Suppose now that $D$ is AV-parametrized. Therefore, there exists a positive rank abelian subvariety $A \subset \PPic_{X/k}^{0}$ such that $\AAbel_{X/k}(D) + A \subset \W_{X/k}$. Since $\Norm_{X/Y} : \PPic_{X/k} \to \PPic_{Y/k}$ is a morphism of group schemes, we have
	\[
		\Norm_{X/Y}(\AAbel_{X/k}(D)) + \Norm_{X/Y} A \subset \Norm_{X/Y}(\W_{X/k}).
	\]
	Using (\ref{eq:isolated_divisors:divisor_picard_pushforward}), it follows that $\AAbel_{Y/k}(f_{\ast}(D)) + \Norm_{X/Y} A \subset \W_{Y/k}$. Since $\Norm_{X/Y} : \PPic_{X/k} \to \PPic_{Y/k}$ is a morphism of group schemes, it follows that $\Norm_{X/Y} A$ is an abelian variety. Thus, if $\Norm_{X/Y} A$ has positive rank, we obtain that $f_{\ast}(D)$ is AV-parametrized, and hence is not isolated.

	On the other hand, if $\Norm_{X/Y} A$ has rank zero, then the kernel $B$ of the morphism of abelian varieties $\Norm_{X/Y} : A \to \Norm_{X/Y} A$ is a group scheme with infinitely many $k$-rational points. By Theorem \ref{thm:divisor_picard_schemes:abel_rational_image}, there exists a subgroup $G \leq \PPic_{X/k}(k)$ such that
	\[
		\AAbel_{X/k}(\DDiv_{X/k}(k)) = \W_{X/k}(k) \cap G,
	\]
	and moreover, the quotient group $\PPic_{X/k}(k) / G$ is torsion. As $B(k)$ is a subgroup of the finitely generated abelian group $A(k)$, it is itself finitely generated. Therefore, the quotient group $B(k) / (B(k) \cap G) \leq \PPic_{X/k}(k) / G$ is a finitely generated torsion group, and hence is finite. As the group $B(k)$ is infinite, it follows that the group $B(k) \cap G$ is infinite as well.

	By Lemma \ref{thm:isolated_divisors:pushforward_preimage_finite}, the set $\Sigma = \{E \in \DDiv_{X/k}(k) : f_{\ast}(E) = f_{\ast}(D)\}$ is finite. Therefore, the set $(\AAbel_{X/k}(D) + B(k) \cap G) \setminus \AAbel_{X/k}(\Sigma)$ is non-empty and contains a point $x \in \PPic_{X/k}(k)$. Since $\AAbel_{X/k}(D) \in G$ by definition, it follows that $x \in G$. Moreover, since $B \subset A$, it follows from the definition of $A$ that $x \in \W_{X/k}(k)$. Therefore, we obtain that $x \in \AAbel_{X/k}(\DDiv_{X/k}(k))$, and so, there exists $D' \in \DDiv_{X/k}(k)$ such that $\AAbel_{X/k}(D') = x$. Since $x \notin \AAbel_{X/k}(\Sigma)$, we have $f_{\ast}(D) \neq f_{\ast}(D')$. Moreover, by (\ref{eq:isolated_divisors:divisor_picard_pushforward}), we have
	\[
		\AAbel_{Y/k}(f_{\ast}(D')) = \Norm_{X/Y}(x) = \Norm_{X/Y}(\AAbel_{X/k}(D)) = \AAbel_{Y/k}(f_{\ast}(D)),
	\]
	where the second equality stems from the fact that $B$ is the kernel of the homomorphism $\Norm_{X/Y}$. Therefore, $f_{\ast}(D)$ is \Pone-parametrized, and hence is not isolated.
\end{proof}

\subsection{Isolated points on curves} \label{sec:isolated_divisors:isolated_points}

Suppose now that $X$ is a curve, that is to say, $X$ is one-dimensional. Recall that, by Remark \ref{rmk:isolated_divisors:divisor_cycle_equivalence}, the monoid $\DDiv_{X/k}(k)$ is isomorphic to the free commutative monoid generated by the irreducible closed subschemes of $X$ of codimension one. As $X$ has dimension one, the latter are precisely the closed points of $X$. Therefore, the monoid $\DDiv_{X/k}(k)$ is isomorphic to the monoid of effective formal sums of closed points of $X$.

This identification between relative effective Cartier divisors of $X/k$ and sums of closed points of $X$ allows us to extend the notion of isolated divisors to closed points. If $X$ is geometrically integral, this recovers the notion of isolated points defined in \cite[Definition 4.1]{bourdon2019}.

\begin{definition}
	Let $x \in X$ be a closed point. We say that $x$ is \textbf{\Pone-parametrized} if the corresponding relative effective Cartier divisor is \Pone-parametrized, and \textbf{\Pone-isolated} otherwise. We define AV-parametrized, AV-isolated and isolated points in a similar manner.
\end{definition}

Recall that, for any closed point $x \in X$, the \textbf{degree} $\deg_{k}(x)$ of $x$ over $k$ is defined by
\[
	\deg_{k}(x) = [k(x) : k],
\]
where $k(x)$ denotes the residue field of $x$. The degree extends linearly to all formal sums of closed points of $X$, and therefore, by the correspondence detailed above, to all relative effective Cartier divisors on $X/k$.

Alternatively, recall that closed points on $X$ are in bijection with the orbits of geometric points $y \in X(\bar{k})$ under the action of the absolute Galois group $G_{k}$. The degree of a closed point of $X$ over $k$ is then equal to the size of the corresponding orbit of geometric points.

By virtue of the Stein factorization $X \to \Spec K \to \Spec k$, one can similarly define the degree $\deg_{K}(x)$ of a closed point $x \in X$ over $K$ by
\[
	\deg_{K}(x) = [k(x) : K].
\]
As before, this extends naturally to all relative effective Cartier divisors on $X/K$. Denote by $r$ the degree of the field extension $K/k$. Then, we have $\deg_{k}(D) = r \deg_{K}(D)$ for all relative effective Cartier divisors $D$ on $X$, as, for any closed point $x \in X$, we have
\[
	[k(x) : k] = [k(x) : K] [K : k].
\]
When the base field is unambiguous, we will often omit the subscript and speak of the degree $\deg(x)$ of $x$.

The degree gives a powerful tool for reinterpreting the results of the previous sections in the particular case of curves. For instance, as $X$ is integral, two relative effective Cartier divisors on $X$ are algebraically equivalent if and only if their degrees are equal. Therefore, Theorem \ref{thm:isolated_divisors:non_isolated_infinite} shows that there are infinitely many relative effective Cartier divisors on $X$ of a given degree if and only if there exists a non-isolated divisor of said degree. A stronger statement in the case of geometrically connected curves can be found in \cite[Theorem 4.2]{bourdon2019}.

From this, it is tempting to use the degree to describe the structure of the divisor scheme $\DDiv_{X/k}$. However, in the case of geometrically disconnected curves, there is a slight subtlety.

\begin{remark}
	Let $D$ be a relative effective Cartier divisor on $X$ of degree $d$. By \cite[Exercise 3.8]{kleiman2005}, the divisor scheme $\DDiv_{X/k}$ can be written as the disjoint union of schemes $\DDiv_{X/k}^{e}$ parametrizing relative effective Cartier divisors of degree $e$, for $e \in \Z_{\geq 0}$. By \cite[Remark 3.9]{kleiman2005}, the scheme $\DDiv_{X/k}^{e}$ can also be described explicitly as the symmetric product $X^{(e)}$ of the curve $X$. It is important to note, however, that the subscheme $\DDiv_{X/k}^{D}$ is not equal to the scheme $\DDiv_{X/k}^{d}$, unless $X$ is geometrically integral. In general, the scheme $\DDiv_{X/k}^{d}$ can be disconnected, and $\DDiv_{X/k}^{D}$ is the union of the connected components of $\DDiv_{X/k}^{d}$ containing a $k$-rational point.
\end{remark}

As in Section \ref{sec:isolated_divisors:morphisms}, let $f : X \to Y$ be a finite locally free morphism of degree $d \geq 1$ between smooth, projective curves over the number field $k$. Under the identification between relative effective Cartier divisors on $X/k$ and sums of closed points of $X$, we can give an explicit description of the pullback and pushforward maps. Let $x \in X$ be a closed point. Then, by \cite[Proposition 21.10.17]{egaiv4}, we have
\[
	f_{\ast}(x) = [k(x) : k(f(x))] \cdot f(x).
\]
This extends linearly to all positive sums of closed points of $X$. In particular, we note that
\[
	\deg_{k}(f_{\ast}(x)) = [k(x) : k(f(x))] [k(f(x)) : k] = \deg_{k}(x).
\]
Similarly, let $y \in Y$ be a closed point. Then, by \cite[Proposition 21.10.4]{egaiv4}, we have
\[
	f^{\ast}(y) = \sum_{x \in f^{-1}(y)} e_{x} x,
\]
where $e_{x} = \operatorname{length}_{\O_{X, x}}(\O_{X, x} / \mathfrak{m}_{\O_{Y, y}} \O_{X, x})$ is the ramification index of $f$ at $x$. Again, this extends linearly to all positive sums of closed points of $Y$. By \cite[Proposition 21.10.18]{egaiv4}, as $f$ is finite locally free of degree $d$, we have $f_{\ast}(f^{\ast}(y)) = d y$. Therefore,
\[
	\deg_{k}(f^{\ast}(y)) = \deg_{k}(f_{\ast}(f^{\ast}(y))) = d \deg_{k}(y).
\]

Using these explicit descriptions, we can carry the results of Section \ref{sec:isolated_divisors:morphisms} over to the setting of isolated points. In the case of pullbacks, we obtain the following generalization of \cite[Theorem 4.3]{bourdon2019} to the case of geometrically disconnected curves.

\begin{theorem} \label{thm:isolated_divisors:pullback_isolated_point}
	Let $f: X \to Y$ be as above. Let $x \in X$ and $y \in Y$ be closed points such that $f(x) = y$. Suppose that $\deg_{k}(x) = d \cdot \deg_{k}(y)$. Then, the following hold:
	\begin{enumerate}
		\item Suppose that $y$ is \Pone-parametrized. Then $x$ is \Pone-parametrized.
		\item Suppose that $y$ is AV-parametrized. Then $x$ is AV-parametrized.
	\end{enumerate}
	In particular, if $x$ is isolated, then $y$ is also isolated.
\end{theorem}

\begin{proof}
	Consider the pullback $f^{\ast}(y)$ of the closed point $y \in Y$. Note that, by assumption, we have $x \in f^{-1}(y)$ and $\deg_{k}(x) = \deg_{k}(f^{\ast}(y))$. As the degree and ramification degree of every closed point of $X$ is positive, it follows from the definition that we must have $x = f^{\ast}(y)$. The result then follows from Theorem \ref{thm:isolated_divisors:pullback_isolated}.
\end{proof}

The case of pushforwards gives rise to the following theorem, which nicely complements the one above.

\begin{theorem} \label{thm:isolated_divisors:pushforward_isolated_point}
	Let $f: X \to Y$ be as above. Let $x \in X$ and $y \in Y$ be closed points such that $f(x) = y$. Suppose that $\deg_{k}(x) = \deg_{k}(y)$. Then, the following hold:
	\begin{enumerate}
		\item Suppose that $x$ is \Pone-parametrized. Then $y$ is \Pone-parametrized.
		\item Suppose that $x$ is AV-parametrized. Then $y$ is not isolated.
	\end{enumerate}
	In particular, if $y$ is isolated, then $x$ is also isolated.
\end{theorem}

\begin{proof}
	Consider the pushforward $f_{\ast}(x)$ of the closed point $x \in X$. Since $y = f(x)$, we have $f_{\ast}(x) = [k(x) : k(y)] \cdot y$. Moreover, by assumption, we have $\deg_{k}(y) = \deg_{k}(f_{\ast}(x))$. As this degree is positive, it follows that $[k(x) : k(y)] = 1$, and so $y = f_{\ast}(x)$. Thus, the result follows from Theorem \ref{thm:isolated_divisors:pushforward_isolated}.
\end{proof}

To conclude this section, we give a bound for the degree of \Pone-isolated points on the curve $X$, in terms of the genus of the geometric components of $X$. This criterion derives from the Riemann-Roch theorem, and generalizes \cite[Lemma 2.3]{ejder2022} to the case of geometrically disconnected curves.

\begin{theorem} \label{thm:isolated_divisors:riemann_roch_criterion}
	Let $X, k$ and $K$ be as above. Let $r = [K : k]$ be the number of geometric components of $X$, and let $g$ be the genus of the geometrically integral curve $X/K$. Let $x \in X$ be a closed point on $X/k$, and suppose that $\deg_{k}(x) > r g$. Then $x$ is \Pone-parametrized.
\end{theorem}

\begin{proof}
	By Theorem \ref{thm:divisor_picard_schemes:abel_fiber}, the fiber $\AAbel_{X/k}^{-1}(\AAbel_{X/k}(x))$ is isomorphic to the Weil restriction $\Res_{K/k}(\mathbb{P}^{n-1}_{K})$, where $n = \dim_{K} H^{0}(X, \O_{X}(x))$. As the curve $X/K$ is smooth and geometrically integral, the Riemann-Roch theorem \cite[Theorem 9.1.1]{bosch2012} shows that
	\[
		\dim_{K} H^{0}(X, \O_{X}(x)) \geq \deg_{K}(x) + 1 - g.
	\]
	By assumption, we know that $\deg_{k}(x) > r g$, and so $\deg_{K}(x) > g$. Therefore, we have $n > 1$, and so the fiber $\AAbel_{X/k}^{-1}(\AAbel_{X/k}(x))$ has infinitely many $k$-rational points. In particular, $x$ is \Pone-parametrized.
\end{proof}
	\section{Group-theoretic preliminaries} \label{sec:group_theory}

In the following sections, we will aim to apply the results of Section \ref{sec:isolated_divisors:isolated_points} to the setting of modular curves. Before doing so, we collect a number of group-theoretic results which will be used throughout. The reader is invited to skip this section on first reading, and return to each result as and when they are used.

\subsection{Products of group subsets}

We begin with some elementary, yet somewhat unconventional results about products of group subsets. Recall that the product of two subsets $S$ and $T$ of a group $G$, denoted by $ST$, is defined by
\[
	ST = \{st : s \in S, t \in T\}.
\]
This is not necessarily a subgroup of $G$, even when both $S$ and $T$ are subgroups themselves. However, in this latter case, if $ST = TS$, then $ST$ forms a group, and we say that the subgroups $S$ and $T$ \textbf{permute}. This occurs notably if $S$ normalizes $T$; that is to say, if $S$ is contained in the normalizer of $T$.

As products of subgroups are not necessarily groups, manipulating these requires some care. However, these products are still well-behaved, as the results of this section aim to show. Firstly, we prove a generalization of the so-called modular law for groups to group subsets.

\begin{lemma} \label{thm:group_theory:modular_law_subsets}
	Let $S, T$ and $U$ be subsets of a group $G$. Let $\langle U \rangle$ be the subgroup of $G$ generated by the elements of $U$. Suppose that $\langle U \rangle S \subseteq S$, or equivalently, $S$ is the union of a number of right cosets of $\langle U \rangle$. Then,
	\[
		U(S \cap T) = S \cap UT.
	\]
\end{lemma}

\begin{proof}
	It is clear that $U(S \cap T) \subseteq UT$. Moreover, we have $U(S \cap T) \subseteq US \subseteq S$. Therefore, $U(S \cap T) \subseteq S \cap UT$.

	On the other hand, let $ut \in S \cap UT$, for some $u \in U$, $t \in T$. Then, $t = u^{-1} (ut) \in \langle U \rangle S \subseteq S$. Therefore, $S \cap UT \subseteq U(S \cap T)$, and the result follows.
\end{proof}

While the product of two arbitrary subgroups is not necessarily a subgroup itself, taking the intersection with a suitable third group rectifies this, as the following shows.

\begin{lemma} \label{thm:group_theory:product_meet_normalizer}
	Let $G, H$ and $K$ be subgroups of a common overgroup $M$. Suppose that $G$ normalizes $H$. Then $G \cap HK$ is a subgroup of $G$.
\end{lemma}

\begin{proof}
	It suffices to show that the non-empty set $G \cap HK$ is closed under products and inverses. Let $g = hk$ and $g' = h' k'$ be elements of $G \cap HK$, for some $g, g' \in G$, $h, h' \in H$ and $k, k' \in K$. Then, we have
	\[
		g^{-1} = (g^{-1} h^{-1} g) (g^{-1} h g) g^{-1} = (g^{-1} h^{-1} g) g^{-1} h = (g^{-1} h^{-1} g) k^{-1} \in G \cap HK,
	\]
	and
	\[
		g g' = g h' k' = (g h' g^{-1}) g k' = (g h' g^{-1}) h k k' \in G \cap HK,
	\]
	as required.
\end{proof}

Taking products with a suitable third group also preserves the index of subgroups.

\begin{lemma} \label{thm:group_theory:product_preserves_index}
	Let $G$ and $H$ be subgroups of a common overgroup $M$ with $H \leq G$. Let $N$ be a subgroup of $M$ which permutes with both $G$ and $H$, and such that $G \cap N = H \cap N$. Then
	\[
		[G : H] = [GN : HN].
	\]
\end{lemma}

\begin{proof}
	Let $G / H$ denote the set of left cosets of $H$ in $G$, and $GN / HN$ denote the set of left cosets of $HN$ in $GN$. Consider the map of sets
	\begin{align*}
		f : G / H & \to GN / HN   \\
		g H       & \mapsto g HN.
	\end{align*}
	This map is clearly well-defined and surjective. Let $g, g' \in G$ be such that $f(gH) = f(g'H)$. Then, we have $g g'^{-1} \in HN$, and so there exist $h \in H$ and $n \in N$ such that $g g'^{-1} = hn$. Since $H \leq G$, we have $n = h^{-1} g g'^{-1} \in G$, so $n \in G \cap N$. By assumption, we have $G \cap N = H \cap N$, so $n \in H$. Hence $g g'^{-1} = hn \in H$, and so $g H = g H'$. Thus, $f$ is a bijection, and the result follows.
\end{proof}

We also have the following elementary statement on indices, which does not fit nicely elsewhere.

\begin{lemma} \label{thm:group_theory:surjectivity_preserves_index}
	Let $f : G \to H$ be a surjective group homomorphism, and let $K$ be a subgroup of $H$. Then $[G : f^{-1}(K)] = [H : K]$.
\end{lemma}

\begin{proof}
	The map $g f^{-1}(K) \mapsto f(g) K$ is a well-defined bijection between the set of (left) cosets of $f^{-1}(K)$ in $G$ and the set of cosets of $K$ in $H$.
\end{proof}

\subsection{The profinite group \texorpdfstring{$\GL{\Zhat}$}{GL\_2(Zhat)}}

Another important object of study is the general linear group $\GL{\Zhat}$ of degree 2 over the profinite integers. We recall a few fundamental properties of this group here.

The group $\GL{\Zhat}$ is a profinite group, with $\GL{\Zhat} = \varprojlim_{m} \GL{\Z[m]}$. Therefore, for any positive integer $n$, we obtain a surjective projection morphism $\pi_{n} : \GL{\Zhat} \to \GL{\Z[n]}$. Moreover, as is customary, $\GL{\Zhat}$ is equipped with the profinite topology. The open subgroups of $\GL{\Zhat}$ are then in correspondence with subgroups of $\GL{\Z[n]}$, as $n$ ranges through all positive integers. Namely, any open subgroup $H$ of $\GL{\Zhat}$ contains the kernel of $\pi_{n}$, for some positive integer $n$. In particular, for such an $n$, we have that $H = \pi_{n}^{-1}(\pi_{n}(H))$. The least such $n$ is called the \textbf{level} of $H$. Conversely, for any subgroup $H$ of $\GL{\Z[n]}$, the preimage $\pi_{n}^{-1}(H)$ is an open subgroup of $\GL{\Z}$. As a result, we will freely alternate between open subgroups of $\GL{\Zhat}$ and subgroups of $\GL{\Z[n]}$ throughout.

By the Chinese remainder theorem, we may express $\GL{\Zhat}$ as the direct product
\[
	\GL{\Zhat} = \prod_{p} \GL{\Z_{p}},
\]
where $p$ ranges through all primes. The groups $\GL{\Z_{p}}$ are also profinite, and the profinite topology on $\GL{\Zhat}$ coincides with the restricted product topology on their product.

The projection morphisms $\pi_{n} : \GL{\Zhat} \to \GL{\Z[n]}$ are compatible in a natural way, as the following result shows.

\begin{lemma} \label{thm:group_theory:kernel_product_gl2_zhat}
	Let $n$ and $m$ be two positive integers, and let $\pi_{n} : \GL{\Zhat} \to \GL{\Z[n]}$ and $\pi_{m} : \GL{\Zhat} \to \GL{\Z[m]}$ be the respective projection maps. Then, we have
	\[
		(\ker \pi_{n}) (\ker \pi_{m}) = \ker \pi_{(m, n)}.
	\]
\end{lemma}

\begin{proof}
	It is clear that the kernels of both $\pi_{n}$ and $\pi_{m}$ are contained in that of $\pi_{(m, n)}$, and so it remains to show that
	\[
		\ker \pi_{(m, n)} \subseteq (\ker \pi_{n}) (\ker \pi_{m}).
	\]
	Let $1 + (m, n)A$ be an element of the kernel of $\pi_{(m, n)}$, for some $2 \times 2$ matrix $A \in M_{2}(\Zhat)$. By B\'ezout's lemma, there exist integers $a, b \in \Z$ such that $am + bn = (m, n)$. We have
	\[
		(1 + maA)(1 + nbA) = 1 + (am + bn)A + mnabA^2 = 1 + (m, n)A + mnab A^2,
	\]
	and so
	\[
		(1 + maA)(1 + nbA)(1 + (m, n)A)^{-1} \in \ker \pi_{mn}.
	\]
	Thus, we have $1 + (m, n)A \in (\ker \pi_{mn})(\ker \pi_{n})(\ker \pi_{m})$. Since the kernel of $\pi_{mn}$ is contained in the kernel of $\pi_{n}$, the result follows.
\end{proof}

Finally, we define a few standard subgroups of $\GL{\Zhat}$ which will reoccur later.

\begin{notation}
	Let $n$ be a positive integer. We let $B_{0}(n)$ be the subgroup of $\GL{\Z[n]}$ consisting of upper triangular matrices, and let $B_{1}(n)$ be the subgroup
	\[
		B_{1}(n) = \{\left(\begin{smallmatrix} 1 & b \\ 0 & c \end{smallmatrix}\right) : b \in \Z[n], c \in \Z*[n]\} \leq \GL{\Z[n]}.
	\]
	As is customary, we also use $B_{0}(n)$ and $B_{1}(n)$ to denote the corresponding open subgroups of $\GL{\Zhat}$. Note that, for any $n$ dividing $m$, we have $B_{0}(m) \equiv B_{0}(n) \pmod n$, and so $B_{0}(n) = (\ker \pi_{n}) B_{0}(m) \leq \GL{\Zhat}$.
\end{notation}
	\section{Modular curves} \label{sec:modular_curves}

Before tackling the subject of isolated points on modular curves, we first recall the definition of the modular curves $X_{H}$, as well as some of their properties. We pay particular attention to the moduli interpretation of $X_{H}$, as much of our approach to isolated points will rely on it. Much of this follows the exposition given in \cite{rouse2022}, though there are some notational differences.

\subsection{Level structures and Galois representations of elliptic curves} \label{sec:modular_curves:finite_level_structures}

To begin, we define the notion of level-$n$ and $H$-level structures on an elliptic curve $E$, which forms the basis of the moduli interpretation of $X_{H}$.

\begin{definition} \label{def:modular_curves:level_structure}
	Let $k$ be a number field, $E$ an elliptic curve over $k$, and $n \geq 1$ an integer. A \textbf{level-$n$ structure on $E$} is an isomorphism of $\Z[n]$-modules
	\[
		\alpha : E[n] \to (\Z[n])^{2}.
	\]

	Let $H$ be a subgroup of $\GL{\Z[n]}$. We say that two level-$n$ structures $\alpha, \beta : E[n] \to (\Z[n])^{2}$ are \textbf{$H$-equivalent} if there exists an element $h \in H$ such that $\beta = h \circ \alpha$. An \textbf{$H$-level structure on $E$} is an $H$-equivalence class of level-$n$ structures on $E$. Given a level-$n$ structure $\alpha$ on $E$, we denote by $[\alpha]_{H}$ the corresponding $H$-level structure.
\end{definition}

We also recall the definition of the mod-$n$ Galois representation attached to an elliptic curve.

\begin{definition} \label{def:modular_curves:galois_representation}
	Let $k$ be a number field and $E$ an elliptic curve over $k$. Fix $n \geq 1$, and let $\alpha: E[n] \to (\Z[n])^{2}$ be a level-$n$ structure on $E$. Let $\sigma \in G_{k}$ be an element of the absolute Galois group of $k$. As $E$ is defined over $k$, $\sigma$ defines an isomorphism $\sigma : E[n] \to E[n]$ of $\Z[n]$-modules. The \textbf{mod-$n$ Galois representation $\bar{\rho}_{E, \alpha}$ associated to $E$ and $\alpha$} is the homomorphism
	\begin{align*}
		\bar{\rho}_{E, \alpha} : G_{k} & \to \GL{\Z[n]}                                 \\
		\sigma                         & \mapsto \alpha \circ \sigma \circ \alpha^{-1}.
	\end{align*}
\end{definition}

\begin{remark}
	The mod-$n$ Galois representation associated to $E$ is often defined up to conjugation, which allows one to remove the dependence on the level-$n$ structure used to define it. However, the choice of level-$n$ structure will be quite important here, and so we choose to make this dependence explicit.
\end{remark}

The image of the mod-$n$ Galois representation is not invariant under $\bar{k}$-isomorphisms of $E$. In order to remedy this, we define a slightly larger group.

\begin{definition} \label{def:modular_curves:extended_galois_image}
	Let $k$ be a number field and $E$ an elliptic curve over $k$. Fix $n \geq 1$, and let $\alpha: E[n] \to (\Z[n])^{2}$ be a level-$n$ structure on $E$. We define $A_{E, \alpha}$ to be the subgroup of $\GL{\Z[n]}$ given by
	\[
		A_{E, \alpha} = \{\alpha \circ \varphi \circ \alpha^{-1} : \varphi \in \Aut(E_{\bar{k}})\} \subset \GL{\Z[n]}.
	\]
	The subgroup $A_{E, \alpha}$ is normalized by the image $\bar{\rho}_{E, \alpha}(G_{k})$ of the mod-$n$ Galois representation associated to $E$ and $\alpha$. Indeed, for any $\sigma \in G_{k}$ and $\varphi \in \Aut(E_{\bar{k}})$, we have
	\[
		\bar{\rho}_{E, \alpha}(\sigma) \circ (\alpha \circ \varphi \circ \alpha^{-1}) \circ \bar{\rho}_{E, \alpha}(\sigma)^{-1} = \alpha \circ \sigma \circ \varphi \circ \sigma^{-1} \circ \alpha^{-1}.
	\]
	The composition $\sigma \circ \varphi \circ \sigma^{-1}$ is an automorphism of $E_{\bar{k}}$, and so $g A_{E, \alpha} g^{-1} = A_{E, \alpha}$ for all $g \in \bar{\rho}_{E, \alpha}(G_{k})$. In particular, the product $\bar{\rho}_{E, \alpha}(G_{k}) A_{E, \alpha} \subset \GL{\Z[n]}$ is a group, which we call the \textbf{extended mod-$n$ Galois image associated to $E$ and $\alpha$}, and denote $G_{E, \alpha}$.
\end{definition}

As alluded to above, the extended mod-$n$ Galois image $G_{E, \alpha}$ is invariant under $\bar{k}$-isomorphisms of $E$.

\begin{lemma} \label{thm:modular_curves:extended_galois_image_invariance}
	Let $E$ and $E'$ be two elliptic curves defined over a number field $k$. Fix $n \geq 1$, and let $\alpha$ be a level-$n$ structure on $E$. Let $\varphi : E' \to E$ be a $\bar{k}$-isomorphism of elliptic curves. Then
	\[
		G_{E, \alpha} = G_{E', \alpha \circ \varphi}.
	\]
\end{lemma}

\begin{proof}
	Let $\sigma \in G_{k}$ and $\psi \in \Aut(E_{\bar{k}})$. Then, we have
	\[
		\alpha \circ \sigma \circ \psi \circ \alpha^{-1} = (\alpha \circ \varphi) \circ \sigma \circ (\sigma^{-1} \circ \varphi^{-1} \circ \sigma \circ \psi \circ \varphi) \circ (\varphi^{-1} \circ \alpha^{-1}).
	\]
	The composition $\sigma^{-1} \circ \varphi^{-1} \circ \sigma$ is a $\bar{k}$-isomorphism from $E$ to $E'$. Therefore, the composition $\sigma^{-1} \circ \varphi^{-1} \circ \sigma \circ \psi \circ \varphi$ is a $\bar{k}$-automorphism of $E'$. Thus, we obtain that $\alpha \circ \varphi \circ \psi \circ \alpha^{-1} \in G_{E', \alpha \circ \varphi}$, and so $G_{E, \alpha} \subseteq G_{E', \alpha \circ \varphi}$. The reverse inclusion is analogous.
\end{proof}

\begin{remark}
	If $E$ is an elliptic curve defined over a number field $k$ such that $j(E) \notin \{0, 1728\}$, then $\Aut(E_{\bar{k}}) = \{\id, \iota\}$, where $\iota$ denotes the involution mapping $P$ to $-P$. Therefore, for any level-$n$ structure $\alpha$ on $E$, we have
	\[
		A_{E, \alpha} = \{\pm I\} \leq \GL{\Z[n]}.
	\]
	In particular, we have $G_{E, \alpha} = \pm \bar{\rho}_{E, \alpha}(G_{k})$ for all such elliptic curves $E$.
\end{remark}

\subsection{The modular curves \texorpdfstring{$X_{H}$ and $Y_{H}$}{X\_H and Y\_H}} \label{sec:modular_curves:modular_curves}

Equipped with the notion of $H$-level structures on elliptic curves, we can now define the modular curves $X_{H}$. This follows the definition given by Deligne and Rapoport in \cite{deligne1973}.

\begin{definition} \label{def:modular_curves:modular_curve}
	Let $n \geq 1$ be an integer, and let $H \leq \GL{\Z[n]}$ be a subgroup. We define the \textbf{modular curve $Y_{H}$} to be the generic fiber of the coarse moduli space of the algebraic stack $\mathcal{M}_{H}^{0}$ parametrizing elliptic curves with $H$-level structure. Similarly, we define the \textbf{modular curve $X_{H}$} to be the generic fiber of the coarse moduli space of the stack $\mathcal{M}_{H}$ parametrizing generalized elliptic curves with $H$-level structure. For more information on the definitions of the stacks $\mathcal{M}_{H}$ and $\mathcal{M}_{H}^{0}$, we refer the reader to \cite{deligne1973}.

	The modular curve $X_{H}$ is a smooth projective curve over $\Q$, with $Y_{H}$ an affine subscheme of $X_{H}$. The closed points of $X_{H} \setminus Y_{H}$ are the \textbf{cusps} of $X_{H}$, while the closed points of $Y_{H}$ are called \textbf{non-cuspidal}.
\end{definition}

The geometric points of $Y_{H}$ can be described precisely. Let $E, E'$ be two elliptic curves defined over $\Qbar$. Let $[\alpha]_{H}$ be an $H$-level structure on $E$ and $[\alpha']_{H}$ be an $H$-level structure on $E'$. We say that the pairs $(E, [\alpha]_{H})$ and $(E', [\alpha']_{H})$ are equivalent if there exists an isomorphism $\varphi : E \to E'$ over $\Qbar$ such that $[\alpha]_{H} = [\alpha' \circ \varphi]_{H}$. Then, the set $Y_{H}(\Qbar)$ consists of equivalence classes $[(E, [\alpha]_{H})]$ of pairs as above.

The action of the absolute Galois group $G_{\Q}$ on $Y_{H}(\Qbar)$ can also be described explicitly. Let $\sigma \in G_{\Q}$, and $[(E, [\alpha]_{H})] \in Y_{H}(\Qbar)$. The map $\sigma$ induces an isomorphism $\sigma^{-1} : (\sigma E)[n] \to E[n]$ defined by $P \mapsto \sigma^{-1}(P)$. Then, the (left) action of $G_{\Q}$ on $Y_{H}(\Qbar)$ is given by
\[
	\sigma \cdot [(E, [\alpha]_{H})] = [(\sigma E, [\alpha \circ \sigma^{-1}]_{H})].
\]
Note that this action is well-defined, as, if the pairs $(E, [\alpha]_{H})$ and $(E', [\alpha']_{H})$ are equivalent, we have
\begin{align*}
	\exists \, \varphi : E \overset{\sim}{\to} E', h \in H \text{ s.t. } & \alpha = h \circ \alpha' \circ \varphi                                                                         \\
	\implies                                                             & \alpha \circ \sigma^{-1} = h \circ \alpha' \circ \varphi \circ \sigma^{-1}                                     \\
	\implies                                                             & \alpha \circ \sigma^{-1} = h \circ (\alpha' \circ \sigma^{-1}) \circ (\sigma \circ \varphi \circ \sigma^{-1}).
\end{align*}
The composition $\sigma \circ \varphi \circ \sigma^{-1}$ is an isomorphism from $\sigma E$ to $\sigma E'$, and so the above shows that the pairs $\sigma \cdot (E, [\alpha]_{H})$ and $\sigma \cdot (E', [\alpha']_{H})$ are also equivalent, for all $\sigma \in G_{\Q}$.

\begin{definition} \label{def:modular_curves:minimal_representative}
	Let $H$ be a subgroup of $\GL{\Z[n]}$, and $x \in X_{H}$ be a non-cuspidal closed point. We say that a pair $(E, [\alpha]_{H})$ is a \textbf{minimal representative for $x$} if $E$ is an elliptic curve defined over $\Q(j(E))$, $[\alpha]_{H}$ is an $H$-level structure on $E$, and the point $x$ corresponds to the $G_{\Q}$-orbit of $[(E, [\alpha]_{H})] \in Y_{H}(\Qbar)$.
\end{definition}

We note that such minimal representatives are not unique. For instance, any twist of $E$ gives rise to another minimal representative. Minimal representatives do however exist for all non-cuspidal closed points of $X_{H}$, by the following lemma.

\begin{lemma} \label{thm:modular_curves:minimal_representative_existence}
	Let $[(E, [\alpha]_{H})] \in Y_{H}(\Qbar)$ be a geometric point of $Y_{H}$, and let $E'$ be an elliptic curve defined over $\Qbar$ such that $j(E) = j(E')$. Then, there exists an $H$-level structure $[\alpha']_{H}$ on $E'$ such that
	\[
		[(E, [\alpha]_{H})] = [(E', [\alpha']_{H})].
	\]
	In particular, there exists a minimal representative for all non-cuspidal closed points of $X_{H}$.
\end{lemma}

\begin{proof}
	Since $E$ and $E'$ have the same $j$-invariant, there exists an isomorphism $\varphi : E \to E'$ over $\Qbar$. Consider the level-$n$ structure $\alpha' = \alpha \circ \varphi^{-1}$ on $E'$. By construction, we have $[\alpha]_{H} = [\alpha' \circ \varphi]_{H}$, and so
	\[
		[(E, [\alpha]_{H})] = [(E', [\alpha']_{H})].
	\]
	The second part follows from the fact that there exists an elliptic curve $E'$ defined over $\Q(j(E))$ such that $j(E') = j(E)$.
\end{proof}

Using these minimal representatives, as well as the above description of the Galois action on $Y_{H}(\Qbar)$, we can give a concrete formula for the degrees of the non-cuspidal closed points of $X_{H}$.

\begin{theorem} \label{thm:modular_curves:point_degree}
	Let $x \in X_{H}$ be a non-cuspidal closed point with minimal representative $(E, [\alpha]_{H})$. Then,
	\[
		\deg(x) = [\Q(j(E)) : \Q] [G_{E, \alpha} : G_{E, \alpha} \cap A_{E, \alpha} H].
	\]
\end{theorem}

\begin{proof}
	The degree of $x \in Y_{H}$ is equal to the size of the $G_{\Q}$-orbit of the point $[(E, [\alpha]_{H})] \in Y_{H}(\Qbar)$. By the orbit-stabilizer theorem, it follows that
	\[
		\deg(x) = [G_{\Q} : \Stab_{G_{\Q}}([(E, [\alpha]_{H})])].
	\]

	Let $\sigma \in G_{\Q}$. By definition, we know that $\sigma \in \Stab_{G_{\Q}}([(E, [\alpha]_{H})])$ if and only if there exists an isomorphism $\varphi : E \to \sigma E$ over $\Qbar$ such that $[\alpha]_{H} = [\alpha \circ \sigma^{-1} \circ \varphi]_{H}$. An isomorphism $\varphi : E \to \sigma E$ can only exist if $j(E) = j(\sigma E)$, or in other words, if $\sigma \in G_{\Q(j(E))}$. Note that as $E$ is defined over $\Q(j(E))$, we have $\sigma E = E$ for all $\sigma \in G_{\Q(j(E))}$. Therefore, we have that $\sigma \in \Stab_{G_{\Q}}([(E, [\alpha]_{H})])$ if and only if $\sigma \in G_{\Q(j(E))}$ and there exists an automorphism $\varphi \in \Aut(E_{\Qbar})$ such that $[\alpha]_{H} = [\alpha \circ \sigma^{-1} \circ \varphi]_{H}$. We have
	\begin{align*}
		[\alpha]_{H} = [\alpha \circ \sigma^{-1} \circ \varphi]_{H} & \iff \alpha \circ \varphi^{-1} \circ \sigma \circ \alpha^{-1} \in H                         \\
		                                                            & \iff \alpha \circ \varphi^{-1} \circ \alpha^{-1} \circ \bar{\rho}_{E, \alpha}(\sigma) \in H \\
		                                                            & \iff \bar{\rho}_{E, \alpha}(\sigma) \in (\alpha \circ \varphi \circ \alpha^{-1}) H.
	\end{align*}
	Therefore, $\sigma \in \Stab_{G_{\Q}}([(E, [\alpha]_{H})])$ if and only if $\sigma \in G_{\Q(j(E))}$ and $\bar{\rho}_{E, \alpha}(\sigma) \in A_{E, \alpha} H$. Thus, we obtain
	\begin{align*}
		\deg(x) & = [G_{\Q} : G_{\Q(j(E))}] [G_{\Q(j(E))} : \bar{\rho}_{E, \alpha}^{-1}(\bar{\rho}_{E, \alpha}(G_{\Q(j(E))}) \cap A_{E, \alpha} H)] \\
		        & = [\Q(j(E)) : \Q] [G_{\Q(j(E))} : \bar{\rho}_{E, \alpha}^{-1}(\bar{\rho}_{E, \alpha}(G_{\Q(j(E))}) \cap A_{E, \alpha} H)].
	\end{align*}
	Note that, as $\bar{\rho}_{E, \alpha}(G_{\Q(j(E))})$ normalizes $A_{E, \alpha}$, the intersection $\bar{\rho}_{E, \alpha}(G_{\Q(j(E))}) \cap A_{E, \alpha} H$ is a group, by Lemma \ref{thm:group_theory:product_meet_normalizer}. Therefore, by Lemma \ref{thm:group_theory:surjectivity_preserves_index}, we have
	\[
		\deg(x) = [\Q(j(E)) : \Q] [\bar{\rho}_{E, \alpha}(G_{\Q(j(E))}) : \bar{\rho}_{E, \alpha}(G_{\Q(j(E))}) \cap A_{E, \alpha} H].
	\]
	By Lemma \ref{thm:group_theory:modular_law_subsets}, we have $A_{E, \alpha}(\bar{\rho}_{E, \alpha}(G_{\Q(j(E))}) \cap A_{E, \alpha} H) = G_{E, \alpha} \cap A_{E, \alpha} H$. Therefore, by Lemma \ref{thm:group_theory:product_preserves_index}, we have
	\[
		[\bar{\rho}_{E, \alpha}(G_{\Q(j(E))}) : \bar{\rho}_{E, \alpha}(G_{\Q(j(E))}) \cap A_{E, \alpha} H] = [G_{E, \alpha} : G_{E, \alpha} \cap A_{E, \alpha} H],
	\]
	and the result follows.
\end{proof}

\begin{remark} \label{rmk:modular_curves:geometric_components}
	The geometric components of the curve $X_{H}$ can also be described purely in terms of the group $H$. Recall that the Galois group $\Gal(\Q(\zeta_{n}) / \Q)$ is isomorphic to $\Z*[n]$, with $i \in \Z*[n]$ corresponding to the automorphism $\zeta_{n} \mapsto \zeta_{n}^{i}$. Let $H$ be a subgroup of $\GL{\Z[n]}$. Then, by \cite[Corollaire 5.6]{deligne1973}, the Stein factorization of the modular curve $X_{H}$ is
	\[
		X_{H} \to \Spec \! \left(\Q(\zeta_{n})^{\det H}\right) \to \Spec \Q,
	\]
	where $\Q(\zeta_{n})^{\det H}$ denotes the fixed field of the subgroup $\det H$. In particular, $X_{H}$ has $[\Z*[n] : \det H]$ geometric components.
	
	The morphism $X_{H} \to \Spec \! \left(\Q(\zeta_{n})^{\det H}\right)$ can be explicitly described on non-cuspidal geometric points. Recall that, since the Galois group $\Gal(\Q(\zeta_{n})^{\det H} / \Q)$ is isomorphic to the quotient group $\Z*[n] / \det H$, the geometric points of $\Spec(\Q(\zeta_{n})^{\det H})$ are in bijection with the set of cosets of $\det H$ in $\Z*[n]$. Let $E / \Qbar$ be an elliptic curve, and $\alpha$ a level-$n$ structure on $E$. Let $e_{n} : E[n]^{2} \to \mu_{n}$ denote the Weil pairing on $E[n]$. The image
	\[
		e_{n} \! \left(\alpha^{-1}(\left(\begin{smallmatrix} 1 \\ 0 \end{smallmatrix}\right)), \alpha^{-1}(\left(\begin{smallmatrix} 0 \\ 1 \end{smallmatrix}\right)) \right)
	\]
	is a primitive $n$-th root of unity, and so equals $\zeta_{n}^{d(\alpha)}$, for some $d(\alpha) \in \Z*[n]$. The morphism $X_{H} \to \Spec \! \left(\Q(\zeta_{n})^{\det H}\right)$ is then given, on non-cuspidal geometric points, by the map
	\begin{align*}
		Y_{H}(\Qbar) & \to \Z*[n] / \det H \\
		[(E, [\alpha]_{H})] & \mapsto d(\alpha) \det H.
	\end{align*}
	Note that, as the Weil pairing is alternating and bilinear, we have that $d(h \circ \alpha) = d(\alpha) \det(h)$, for all $h \in H$. Therefore, the above map is well-defined.
	
	The structure of the geometric components of $X_{H}$ is also well-known. Let $\Gamma_{H} \leq \SL{\Z}$ be the preimage of the intersection $H \cap \SL{\Z[n]}$ under the reduction mod-$n$ map $\SL{\Z} \to \SL{\Z[n]}$. Then, by \cite[Remark 2.3.3]{rouse2022}, the base change of the geometrically connected curve $X_{H} / \Q(\zeta_{n})^{\det H}$ to $\mathbb{C}$ is isomorphic to the classical modular curve $\Gamma_{H} \backslash \mathbb{H}^{\ast}$, where $\mathbb{H}^{\ast}$ is the extended upper half plane. In particular, the genus of $X_{H} / \Q(\zeta_{n})^{\det H}$ is equal to the genus of $\Gamma_{H} \backslash \mathbb{H}^{\ast}$, which can be computed using \cite[Theorem 3.1.1]{diamond2005}.
\end{remark}

Finally, we define the four standard modular curves $X(1)$, $X(n)$, $X_{0}(n)$ and $X_{1}(n)$.

\begin{notation}
	Throughout, we denote by $X(1)$ the modular curve $X_{\GL{\Z[n]}}$. As we shall see in the following section, the resulting curve does not depend on the choice of $n$. The modular curve $X(1)$ is naturally isomorphic to $\mathbb{P}^{1}$ over $\Q$, with the isomorphism given, on non-cuspidal geometric points, by the map
	\begin{align*}
		Y(1)(\Qbar) & \to \mathbb{A}^{1}_{\Q}(\Qbar) \\
		[(E, [\alpha]_{\GL{\Z[n]}})] & \mapsto j(E).
	\end{align*}
	
	Similarly, we denote by $X(n)$ the modular curve $X_{H}$, where $H = \{\pm I\} \leq \GL{\Z[n]}$. By the previous remark, this is a geometrically disconnected curve, whose geometric components are isomorphic, over $\mathbb{C}$, to the Riemann surface $\Gamma_{n} \backslash \mathbb{H}^{\ast}$. We also denote by $X_{0}(n)$ the modular curve $X_{B_{0}(n)}$, and by $X_{1}(n)$ the modular curve $X_{B_{1}(n)}$.
\end{notation}

\subsection{Morphisms between modular curves} \label{sec:modular_curves:morphisms}

There are a number of morphisms between modular curves which are induced by the underlying moduli interpretation. These will play an important role in understanding isolated points on these modular curves, alongside the results of Section \ref{sec:isolated_divisors}. We define these morphisms here.

\begin{definition}
	Let $H$ and $H'$ be two subgroups of $\GL{\Z[n]}$ such that $H \leq H'$. The \textbf{inclusion morphism} $X_{H} \to X_{H'}$ is the morphism of modular curves given, on non-cuspidal geometric points, by the map
	\begin{align*}
		Y_{H}(\Qbar)      & \to Y_{H'}(\Qbar)           \\
		[(E, [\alpha]_{H})] & \mapsto [(E, [\alpha]_{H'})].
	\end{align*}
	The inclusion morphism is a finite locally free morphism of curves over $\Q$, with degree equal to $[\pm H' : \pm H]$.
\end{definition}

\begin{remark}
	The inclusion morphism is canonical in the sense that, for any three subgroups $H \leq H' \leq H'' \leq \GL{\Z[n]}$, the composition of inclusion morphisms
	\[
		X_{H} \to X_{H'} \to X_{H''}
	\]
	is equal to the inclusion morphism $X_{H} \to X_{H''}$.
\end{remark}

For any subgroup $H$ of $\GL{\Z[n]}$, the inclusion morphism gives rise to a morphism of modular curves $j : X_{H} \to X(1)$, called the $j$-map. Under the identification of $X(1)$ with $\mathbb{P}^{1}$, the $j$-map takes the geometric point $[(E, [\alpha]_{H})] \in Y_{H}(\Qbar)$ to the $j$-invariant $j(E) \in \mathbb{A}^{1}_{\Q}(\Qbar)$. The degree of the $j$-map is equal to the index $[\GL{\Z[n]} : \pm H]$.

The inclusion morphism also gives rise to isomorphisms between different modular curves. Firstly, for all subgroups $H \leq \GL{\Z[n]}$, the inclusion morphism gives an isomorphism $X_{H} \cong X_{\pm H}$. Therefore, when considering properties of the curves $X_{H}$, such as their set of isolated points, we can restrict ourselves to considering subgroups such that $-I \in H$.

Secondly, the inclusion morphism gives an isomorphism between the Stein factorizations of the modular curves corresponding to subgroups with the same intersection in $\SL{\Z[n]}$.
	
\begin{theorem} \label{thm:modular_curves:sl_intersection_defines_geometric_components}
	Let $H \leq H'$ be two subgroups of $\GL{\Z[n]}$ such that $H \cap \SL{\Z[n]} = H' \cap \SL{\Z[n]}$. Let
	\begin{align*}
		X_{H} \overset{\pi}{\to} \Spec \Q(\zeta_{n})^{\det H} \to \Spec \Q & \text{ and} \\
		\qquad X_{H'} \overset{\pi'}{\to} \Spec \Q(\zeta_{n})^{\det H'} \to \Spec \Q &
	\end{align*}
	be the Stein factorizations of $X_{H}$ and $X_{H'}$ respectively. Then, there is a Cartesian square
	\[\begin{tikzcd}
		X_{H} \arrow[d, "f"] \arrow[r, "\pi"] & \Spec \Q(\zeta_{n})^{\det H} \arrow[d, "g"] \\
		X_{H'} \arrow[r, "\pi'"] & \Spec \Q(\zeta_{n})^{\det H'},
	\end{tikzcd}\]
	where $f$ is the inclusion morphism $X_{H} \to X_{H'}$, and $g$ is the morphism of spectra induced by the ring inclusion $\Q(\zeta_{n})^{\det H'} \hookrightarrow \Q(\zeta_{n})^{\det H'}$. In particular, the geometrically integral curve $X_{H} / \Q(\zeta_{n})^{\det H}$ is isomorphic to the base change of the curve $X_{H'} / \Q(\zeta_{n})^{\det H'}$ to the field $\Q(\zeta_{n})^{\det H}$.
\end{theorem}

\begin{proof}
	We first show that the above diagram is commutative. To do so, it suffices to check that the compositions $g \circ \pi$ and $\pi' \circ f$ agree for all geometric points of $Y_{H}$, since $Y_{H}$ is a geometrically reduced, dense open subset of $X_{H}$. Let $[(E, [\alpha]_{H})] \in Y_{H}(\Qbar)$ be a geometric point of $Y_{H}$. By Remark \ref{rmk:modular_curves:geometric_components}, we have that $\pi([(E, [\alpha]_{H})]) = d(\alpha) \det H$, where we identify the geometric points of $\Spec \Q(\zeta_{n})^{\det H}$ with the cosets of $\det H$ in $\Z*[n]$. Therefore, we have
	\[
		g(\pi([(E, [\alpha]_{H})])) = g(d(\alpha) \det H) = d(\alpha) \det H'.
	\]
	On the other hand, we know that $f([(E, [\alpha]_{H})]) = [(E, [\alpha]_{H'})]$, and so
	\[
		\pi'(f([(E, [\alpha]_{H})])) = \pi'([(E, [\alpha]_{H'})]) = d(\alpha) \det H'.
	\]
	Therefore, the diagram is commutative, and it remains to show that it forms a Cartesian square.
	
	By the universality of the fiber product, we obtain a commutative diagram
	\[\begin{tikzcd}
		X_{H} \arrow[ddr, bend right, "f"] \arrow[rrd, bend left=15, "\pi"] \arrow[rd, "\varphi"] \\
		& (X_{H'})_{\Q(\zeta_{n})^{\det H}} \arrow[d, "\psi"] \arrow[r] & \Spec \Q(\zeta_{n})^{\det H} \arrow[d, "g"] \\
		& X_{H'} \arrow[r, "\pi'"] & \Spec \Q(\zeta_{n})^{\det H'},
	\end{tikzcd}\]
	Note that, since $X_{H'} / \Q(\zeta_{n})^{\det H'}$ is a geometrically integral curve, the base change $(X_{H'})_{\Q(\zeta_{n})^{\det H}}$ is integral. The degree of $g$ is given by
	\[
		\deg(g) = [\Q(\zeta_{n})^{\det H} : \Q(\zeta_{n})^{\det H'}] = [\det H' : \det H].
	\]
	Therefore, the degree of $\psi$ is also equal to $[\det H' : \det H]$. On the other hand, the degree of $f$ is equal to $[\pm H' : \pm H]$. Since $H \cap \SL{\Z[n]} = H' \cap \SL{\Z[n]}$, we have
	\[
		\deg(f) = [\pm H' : \pm H] = [\det(\pm H') : \det(\pm H)] = [\det H' : \det H] = \deg(\psi).
	\]
	Thus, the map $\varphi$ is an isomorphism of curves, and the claim follows.
\end{proof}

\begin{remark}
	This result, alongside Theorem \ref{thm:isolated_divisors:stein_factorization_isolated}, shows that studying the isolated points on the modular curves $X_{H}$, for all $H \leq \GL{\Z[n]}$, is equivalent to studying the isolated points on the base changes of the modular curves $X_{H'}$, where $H'$ are the minimal subgroups of $\GL{\Z[n]}$ with given intersection in $\SL{\Z[n]}$. This provides an alternative way of understanding isolated points on geometrically disconnected modular curves, which deserves further attention. We leave this for future work.
\end{remark}

Another important class of morphisms between modular curves are the conjugation isomorphisms, which are defined as follows.

\begin{definition}
	Let $H$ be a subgroup of $\GL{\Z[n]}$, and let $g \in \GL{\Z[n]}$. The \textbf{conjugation isomorphism} $X_{H} \to X_{g H g^{-1}}$ is the morphism of modular curves given, on non-cuspidal geometric points, by the map
	\begin{align*}
		Y_{H}(\Qbar)      & \overset{\sim}{\to} Y_{g H g^{-1}}(\Qbar)        \\
		[(E, [\alpha]_{H})] & \mapsto [(E, [g \circ \alpha]_{g H g^{-1}})].
	\end{align*}
	The conjugation isomorphism is an isomorphism of curves over $\Q$. More generally, one can also conjugate by elements of $\GL{\mathbb{A}_{\Q}^{f}}$, see \cite[Proposition 3.19]{deligne1973}.
\end{definition}

Finally, let $n$ and $m$ be integers such that $n$ divides $m$. Let $H \leq \GL{\Z[n]}$ be a subgroup, and $H' = \pi_{n}^{-1}(H) \leq \GL{\Z[m]}$ the preimage of $n$ under the reduction mod-$n$ map $\pi_{n} : \GL{\Z[m]} \to \GL{\Z[n]}$. Then, by \cite[\nopp 3.6]{deligne1973}, there is an isomorphism between the modular curves $X_{H}$ and $X_{H'}$ defined over $\Q$. In particular, the dependence of the definitions on the integer $n$ is somewhat arbitrary, and one can omit its use by working over $\GL{\Zhat}$. We give these definitions in Section \ref{sec:modular_curves:working_over_zhat}. However, working with subgroups of $\GL{\Z[n]}$ is somewhat more amenable to computations, which is why we have chosen to include both perspectives here.

\subsection{Closed points over a fixed \texorpdfstring{$j$}{j}-invariant} \label{sec:modular_curves:closed_points_as_double_cosets}

In future sections, we will rely heavily on understanding the structure of closed points on modular curves lying above a fixed $j$-invariant. In this section, we present a purely group-theoretic description of these closed points, which extends that given in \cite[Section 2.3]{rouse2022}. This yields a remarkably practical way to compute the properties of such points, assuming knowledge of the Galois representation of the elliptic curve. This will be used to great effect in Section \ref{sec:level_7}.

For ease of referencing, we collect these properties in the following, rather large, proposition.

\begin{proposition} \label{thm:modular_curves:closed_points_as_double_cosets}
	Let $E$ be an elliptic curve defined over $\Q(j(E))$, and fix a level-$n$ structure $\alpha : E[n] \to \GL{\Z[n]}$. Let $H$ and $H'$ be two subgroups of $\GL{\Z[n]}$, with $H \leq H'$, and let $h$ be an element of $\GL{\Z[n]}$. Then, the following hold:
	\begin{enumerate}
		\item There is a bijection between the set of elements of $\GL{\Z[n]}$ and the set of level-$n$ structures $\beta : E[n] \to \GL{\Z[n]}$ given by
		      \[
			      g \mapsto g \circ \alpha.
		      \]
		\item For any $g \in \GL{\Z[n]}$, we have
		      \begin{align*}
			      A_{E, g \circ \alpha} = g A_{E, \alpha} g^{-1}, &  & G_{E, g \circ \alpha} = g G_{E, \alpha} g^{-1}.
		      \end{align*}
		\item The set of geometric points in $X_{H}(\Qbar)$ lying over $j(E) \in X(1)(\Qbar)$ is in bijection with the set of double cosets $H g A_{E, \alpha}$, for $g \in \GL{\Z[n]}$.
		\item The set of closed points in $X_{H}$ lying above the closed point of $X(1)$ corresponding to the $G_{\Q}$-orbit of $j(E) \in X(1)(\Qbar)$ is in bijection with the set of double cosets $H g G_{E, \alpha}$, for $g \in \GL{\Z[n]}$.
		\item The degree of the closed point $x \in X_{H}$ corresponding to the double coset $H g G_{E, \alpha}$ is given by
		      \[
			      \deg(x) = [\Q(j(E)) : \Q] [g G_{E, \alpha} g^{-1} : g G_{E, \alpha} g^{-1} \cap (g A_{E, \alpha} g^{-1}) H].
		      \]
		\item The inclusion morphism $X_{H} \to X_{H'}$ maps the closed point of $X_{H}$ corresponding to the double coset $H g G_{E, \alpha}$ to the closed point of $X_{H'}$ corresponding to the double coset $H' g G_{E, \alpha}$.
		\item The conjugation isomorphism $X_{H} \to X_{h H h^{-1}}$ maps the closed point of $X_{H}$ corresponding to the double coset $H g G_{E, \alpha}$ to the closed point of $X_{h H h^{-1}}$ corresponding to the double coset $(h H h^{-1}) h g G_{E, \alpha}$.
	\end{enumerate}
\end{proposition}

\begin{proof}
	The first two statements follow naturally from Definitions \ref{def:modular_curves:level_structure} and \ref{def:modular_curves:extended_galois_image}.

	We know that the set of geometric points in $X_{H}(\Qbar)$ lying over $j(E) \in X(1)(\Qbar)$ is given by the set of geometric points $[(E', [\beta]_{H})] \in X_{H}(\Qbar)$, where $j(E') = j(E)$. By Lemma \ref{thm:modular_curves:minimal_representative_existence}, one can assume that $E' = E$. In other words, the set of geometric points in $X_{H}(\Qbar)$ lying over $j(E) \in X(1)(\Qbar)$ is given by
	\[
		\left\{[(E, [\beta]_{H})] \in X_{H}(\Qbar) : \beta \text{ a level-$n$ structure on $E$}\right\}.
	\]
	By (1), this is equal to the set
	\[
		\Sigma = \left\{[(E, [g \circ \alpha]_{H})] \in X_{H}(\Qbar) : g \in \GL{\Z[n]} \right\}.
	\]
	Consider the map $\psi : \Sigma \to H \backslash {\GL{\Z[n]}} / A_{E, \alpha}$ given by
	\[
		\psi([(E, [g \circ \alpha]_{H})]) = H g A_{E, \alpha}.
	\]
	This map is clearly surjective. Let $g$ and $g'$ be two elements of $\GL{\Z[n]}$. We show that the map $\psi$ is both well-defined and injective in one fell swoop, by showing that $[(E, [g \circ \alpha]_{H})] = [(E, [g' \circ \alpha]_{H})]$ if and only if $H g A_{E, \alpha} = H g' A_{E, \alpha}$. We have
	\begin{align*}
		[(E, {} & [g \circ \alpha]_{H})] = [(E, [g' \circ \alpha]_{H})]                                                        \\
		        & \iff \exists \, \varphi \in \Aut(E) \text{ s.t. } [g \circ \alpha]_{H} = [g' \circ \alpha \circ \varphi]_{H} \\
		        & \iff \exists \, \varphi \in \Aut(E) \text{ s.t. } g \in H (g' \circ \alpha \circ \varphi \circ \alpha^{-1})  \\
		        & \iff g \in H g' A_{E, \alpha}                                                                                \\
		        & \iff H g A_{E, \alpha} = H g' A_{E, \alpha}.
	\end{align*}
	Therefore, the map $\psi$ is a bijection, and (3) follows.

	Let $x \in X_{H}$ be a closed point lying above the closed point corresponding to the $G_{\Q}$-orbit of $j(E) \in X(1)(\Qbar)$. The closed point $x$ corresponds to the $G_{\Q}$-orbit of some geometric point $[(E', [\beta]_{H})] \in X_{H}(\Qbar)$. By construction, $j(E')$ must be in the $G_{\Q}$-orbit of $j(E)$, and so there exists $\sigma \in G_{\Q}$ such that $j(\sigma E') = \sigma(j(E')) = j(E)$. Since $\sigma \cdot [(E', [\beta]_{H})] = [(\sigma E', [\beta \circ \sigma^{-1}]_{H})]$, we can therefore assume that $j(E') = j(E)$. By Lemma \ref{thm:modular_curves:minimal_representative_existence}, we can moreover assume that $E' = E$. Thus, the set of closed points lying above the closed point corresponding to the $G_{\Q}$-orbit of $j(E) \in X(1)(\Qbar)$ is in bijection with the set of orbits
	\[
		\left\{G_{\Q} \cdot [(E, [\beta]_{H})]: \beta \text{ a level-$n$ structure on $E$}\right\}.
	\]
	Let $\beta$ and $\beta'$ be two level-$n$ structures on $E$, and let $\sigma \in G_{\Q}$ be such that
	\[
		\sigma \cdot [(E, [\beta]_{H})] = [(E, [\beta']_{H})].
	\]
	By definition, it follows that there exists a $\Qbar$-isomorphism $\sigma E \cong E$, and so $j(\sigma E) = j(E)$. Therefore, we obtain that $\sigma \in G_{\Q(j(E))}$. Thus, there is a bijection of sets
	\begin{align*}
		 & \left\{G_{\Q} \cdot [(E, [\beta]_{H})]: \beta \text{ a level-$n$ structure on $E$}\right\}                              \\
		 & \qquad \cong \left\{G_{\Q(j(E))} \cdot [(E, [\beta]_{H})]: \beta \text{ a level-$n$ structure on $E$}\right\} =: \Sigma
	\end{align*}
	Consider the map
	\begin{align*}
		H \backslash {\GL{\Z[n]}} / G_{E, \alpha} & \to \Sigma                                              \\
		H g G_{E, \alpha}                         & \mapsto G_{\Q(j(E))} \cdot [(E, [g \circ \alpha]_{H})].
	\end{align*}
	This is a surjection by (1), and is well-defined and injective by a similar argument as was done for (3). Statement (4) follows accordingly.

	Statement (5) follows from Theorem \ref{thm:modular_curves:point_degree}, and the explicit formulation of the bijection given in the proof of (4). Similarly, statements (6) and (7) follow from the proof of (4), and the definitions of the respective maps of modular curves.
\end{proof}

\subsection{Working with \texorpdfstring{$\GL{\Zhat}$}{GL\_2(Zhat)}} \label{sec:modular_curves:working_over_zhat}

As mentioned at the end of Section \ref{sec:modular_curves:morphisms}, while the definition of the modular curves $X_{H}$ requires an integer $n$ and a subgroup $H$ of $\GL{\Z[n]}$, the modular curve $X_{H}$ constructed is independent of $n$, as long as $H$ defines the same open subgroup of $\GL{\Zhat}$. This allows one to define the modular curves $X_{H}$, for $H$ an open subgroup of $\GL{\Zhat}$, as is commonly done in the literature.

This construction is useful when working with modular curves of varying level, as it allows one to work with multiple modular curves without having to keep track of their respective levels. However, in order to streamline the use of this definition throughout the rest of this paper, we will require a description of the moduli interpretation of $X_{H}$ which depends solely on the open subgroup $H$, rather than its image in $\GL{\Z[n]}$. To this end, we present a formulation of the moduli interpretation of $X_{H}$ which, while equivalent to the one presented in Section \ref{sec:modular_curves:modular_curves}, satisfies the aforementioned properties.

To begin, we define the notions of profinite level structures and $H$-level structures, for $H$ an open subgroup of $\GL{\Zhat}$. These generalize the notions of level-$n$ structures and $H$-level structures, for $H$ a subgroup of $\GL{\Z[n]}$, respectively.

\begin{definition}
	Let $k$ be a number field, and $E$ an elliptic curve over $k$. A \textbf{profinite level structure on $E$} is an isomorphism of $\Zhat$-modules
	\[
		\alpha : \varprojlim_{m} E[m] \overset{\sim}{\to} \varprojlim_{m} (\Z[m])^{2} = \Zhat^{2}.
	\]

	Let $H$ be an open subgroup of $\GL{\Zhat}$. We say that two profinite level structures $\alpha, \beta : \varprojlim_{m} E[m] \to \Zhat^{2}$ are \textbf{$H$-equivalent} if there exists an element $h \in H$ such that $\beta = h \circ \alpha$. An \textbf{$H$-level structure on $E$} is an $H$-equivalence class of profinite level structures on $E$. Given a profinite level structure $\alpha$ on $E$, we denote by $[\alpha]_{H}$ the corresponding $H$-level structure.
\end{definition}

Let $\alpha$ be a profinite level structure on an elliptic curve $E$. Then, for all $n \geq 1$, there is a level-$n$ structure $\alpha_{n} : E[n] \to (\Z[n])^{2}$ such that the following diagram commutes:
\[\begin{tikzcd}
	\varprojlim_{m} E[m] \arrow[r, "\alpha"] \arrow[d] & \varprojlim_{m} (\Z[m])^{2} \arrow[d] \\
	E[n] \arrow[r, "\alpha_{n}"] & (\Z[n])^{2}.
\end{tikzcd}\]
This gives a correspondence between the notion of $H$-level structures defined above and the one defined in Section \ref{sec:modular_curves:finite_level_structures}, as the following theorem shows.

\begin{theorem} \label{thm:modular_curves:mod_n_profinite_equivalence}
	Let $k$ be a number field, and $E$ an elliptic curve over $k$. Let $H$ be an open subgroup of $\GL{\Zhat}$, and let $n$ be a multiple of the level of $H$. Then, the map
	\[
		[\alpha]_{H} \mapsto [\alpha_{n}]_{\pi_{n}(H)}
	\]
	is a well-defined bijection between the set of $H$-level structures on $E$ and the set of $\pi_{n}(H)$-level structures on $E$.
\end{theorem}

\begin{proof}
	The map is surjective, as for any level-$n$ structure $\alpha'$ on $E$, there exists a profinite level structure $\alpha$ on $E$ such that $\alpha_{n} = \alpha'$. Therefore, it suffices to prove that the map is both well-defined and surjective. To do so, let $\alpha$ and $\beta$ be two profinite level structures on $E$. It suffices to prove that $[\alpha]_{H} = [\beta]_{H}$ if and only if $[\alpha_{n}]_{\pi_{n}(H)} = [\beta_{n}]_{\pi_{n}(H)}$. Note that, since $n$ is a multiple of the level of $H$, we know that $H = \pi_{n}^{-1}(\pi_{n}(H))$. Therefore, we have
	\begin{align*}
		[\alpha]_{H} = [\beta]_{H} & \iff \alpha \circ \beta^{-1} \in H \\
		& \iff \pi_{n}(\alpha \circ \beta^{-1}) \in \pi_{n}(H) \\
		& \iff \alpha_{n} \circ \beta_{n}^{-1} \in \pi_{n}(H) \\
		& \iff [\alpha_{n}]_{\pi_{n}(H)} = [\beta_{n}]_{\pi_{n}(H)}. \qedhere
	\end{align*}
\end{proof}

The notion of profinite level structures allows us to define the adelic Galois representation of an elliptic curve $E$.

\begin{definition}
	Let $E$ be an elliptic curve defined over a number field $k$, and let $\alpha$ be a profinite level structure on $E$. Let $\sigma \in G_{k}$ be an element of the absolute Galois group of $k$. Since $E$ is defined over $k$, one obtains isomorphisms $\sigma : E[n] \to E[n]$ for all $n$, which give rise to an isomorphism $\sigma : \varprojlim_{m} E[m] \to \varprojlim_{m} E[m]$ of $\Zhat$-modules. The \textbf{adelic Galois representation associated to $E$ and $\alpha$} is the map
	\begin{align*}
		\rho_{E, \alpha} : G_{k} & \to \GL{\Zhat}                                 \\
		\sigma                   & \mapsto \alpha \circ \sigma \circ \alpha^{-1}.
	\end{align*}
	We let $A_{E, \alpha}$ denote the subgroup
	\[
		A_{E, \alpha} = \{\alpha \circ \varphi \circ \alpha^{-1} : \varphi \in \Aut(E_{\bar{k}})\} \subset \GL{\Zhat}.
	\]
	As in the case of finite level, the subgroup $A_{E, \alpha}$ is normalized by the image $\rho_{E, \alpha}(G_{k})$ of the adelic Galois representation, and we let the \textbf{extended adelic Galois image associated to $E$ and $\alpha$}, denoted by $G_{E, \alpha}$, be the product $\rho_{E, \alpha}(G_{k}) A_{E, \alpha} \subset \GL{\Zhat}$.
\end{definition}

\begin{remark} \label{rmk:modular_curves:serre_open_image}
	Let $E$ be a non-CM elliptic curve over a number field $k$, and $\alpha$ a profinite level structure on $E$. Then, by Serre's open image theorem, the image $\rho_{E, \alpha}(G_{k})$ of the adelic Galois representation of $E$ is an open subgroup of $\GL{\Zhat}$. It follows that, in this case, the extended adelic Galois image $G_{E, \alpha}$ is also an open subgroup of $\GL{\Zhat}$.
\end{remark}

The correspondence between profinite and mod-$n$ level structures gives rise to a correspondence between extended mod-$n$ and adelic Galois images. Namely, given an elliptic curve $E$ defined over a number field $k$, $\alpha$ a profinite level structure on $E$, and $n$ a positive integer, we have
\[
	G_{E, \alpha_{n}} = \pi_{n}(G_{E, \alpha}).
\]

We can now give the definition of the modular curve $X_{H}$ associated to an open subgroup $H$ of $\GL{\Zhat}$. For the moment, this is defined solely using the existing modular curves from Section \ref{sec:modular_curves:modular_curves}.

\begin{definition}
	Let $H$ be an open subgroup of $\GL{\Zhat}$ of level $n$. Let $m$ be an integer such that $n \mid m$. We define the \textbf{modular curve $X_{H}$} to be the modular curve $X_{\pi_{m}(H)}$, and $Y_{H}$ the affine subscheme $Y_{\pi_{m}(H)}$. By the remarks in the previous section, the curve $X_{H}$ is independent of the choice of $m$.
\end{definition}

The equivalence of level structures given in Theorem \ref{thm:modular_curves:mod_n_profinite_equivalence} gives a natural reinterpretation of the moduli description of the modular curve $X_{H}$ in terms of profinite level structures. We summarize this as follows.

Let $H$ be an open subgroup of $\GL{\Zhat}$. The set $Y_{H}(\Qbar)$ of non-cuspidal geometric points of the modular curve $X_{H}$ is in bijection with the set of equivalence classes of pairs $(E, [\alpha]_{H})$, where $E$ is an elliptic curve defined over $\Qbar$, and $[\alpha]_{H}$ an $H$-level structure on $E$. For each non-cuspidal closed point $x \in X_{H}$, we define a minimal representative of $x$ to be such a pair $(E, [\alpha]_{H})$, where $E$ is defined over $\Q(j(E))$ and the closed point $x$ corresponds to the $G_{\Q}$-orbit of the geometric point $[(E, [\alpha]_{H})] \in Y_{H}(\Qbar)$. Finally, for any automorphism $\sigma \in G_{\Q}$ and any geometric point $[(E, [\alpha]_{H})] \in Y_{H}(\Qbar)$, we have
\[
	\sigma \cdot [(E, [\alpha]_{H})] = [(\sigma E, [\alpha \circ \sigma^{-1}]_{H})].
\]

The degree formula of Theorem \ref{thm:modular_curves:point_degree} can also be re-expressed using adelic Galois representations. This is more delicate than the above, and so we give an explicit proof here. 

\begin{theorem} \label{thm:modular_curves:profinite_point_degree}
	Let $H$ be an open subgroup of $\GL{\Zhat}$, and let $x \in X_{H}$ be a non-cuspidal closed point with minimal representative $(E, [\alpha]_{H})$. Then,
	\[
		\deg(x) = [\Q(j(E)) : \Q] [G_{E, \alpha} : G_{E, \alpha} \cap A_{E, \alpha} H].
	\]
\end{theorem}

\begin{proof}
	Let $n$ be the level of $H$. By Theorem \ref{thm:modular_curves:point_degree}, the degree of $x$ is given by
	\[
		\deg(x) = [\Q(j(E)) : \Q] [G_{E, \alpha_{n}} : G_{E, \alpha_{n}} \cap A_{E, \alpha_{n}} \pi_{n}(H)].
	\]
	As the map $\pi_{n} : \GL{\Zhat} \to \GL{\Z[n]}$ is surjective, by Lemma \ref{thm:group_theory:surjectivity_preserves_index}, we have
	\begin{align*}
		[G_{E, \alpha_{n}} : G_{E, \alpha_{n}} \cap A_{E, \alpha_{n}} \pi_{n}(H)] & = [\pi_{n}^{-1}(G_{E, \alpha_{n}}) : \pi_{n}^{-1}(G_{E, \alpha_{n}} \cap A_{E, \alpha_{n}} \pi_{n}(H))] \\
		& = [\pi_{n}^{-1}(G_{E, \alpha_{n}}) : \pi_{n}^{-1}(G_{E, \alpha_{n}}) \cap \pi_{n}^{-1}(A_{E, \alpha_{n}}) H].
	\end{align*}
	By definition, we have $G_{E, \alpha_{n}} = \pi_{n}(G_{E, \alpha})$ and $A_{E, \alpha_{n}} = \pi_{n}(A_{E, \alpha})$, and so, the above index becomes
	\[
		[(\ker \pi_{n}) G_{E, \alpha} : (\ker \pi_{n}) G_{E, \alpha} \cap (\ker \pi_{n}) A_{E, \alpha} H].
	\]
	Thus, by Lemma \ref{thm:group_theory:modular_law_subsets}, we have
	\[
		(\ker \pi_{n}) G_{E, \alpha} \cap (\ker \pi_{n}) A_{E, \alpha} H = (\ker \pi_{n}) (G_{E, \alpha} \cap (\ker \pi_{n}) A_{E, \alpha} H).
	\]
	Note that $\ker \pi_{n} \leq H$ by definition, and so $(\ker \pi_{n}) A_{E, \alpha} H = A_{E, \alpha} H$. Therefore, by Lemma \ref{thm:group_theory:product_preserves_index}, we have
	\[
		[(\ker \pi_{n}) G_{E, \alpha} : (\ker \pi_{n}) G_{E, \alpha} \cap (\ker \pi_{n}) A_{E, \alpha} H] = [G_{E, \alpha} : G_{E, \alpha} \cap A_{E, \alpha} H],
	\]
	and the result follows.
\end{proof}

\begin{remark}
	More generally, the proof of the above theorem shows that, for any normal subgroup $N \lhd \GL{\Zhat}$ such that $N \leq H$, we have the equality
	\[
		[G_{E, \alpha} : G_{E, \alpha} \cap A_{E, \alpha} H] = [G_{E, \alpha} N : G_{E, \alpha} N \cap A_{E, \alpha} H].
	\]
	In particular, we can calculate the index $[G_{E, \alpha} : G_{E, \alpha} \cap A_{E, \alpha} H]$ knowing only the group $G_{E, \alpha} N$.
\end{remark}

Many of the properties of the modular curves $X_{H}$ can also be described purely in terms of the subgroup $H \leq \GL{\Zhat}$, rather than its image in $\GL{\Z[n]}$. For instance, the geometric components of $X_{H}$ are described as follows.

Let $\Qab = \cup_{n \geq 1} \Q(\zeta_{n})$ be the maximal abelian extension of $\Q$. The Galois group $\Gal(\Qab / \Q)$ is isomorphic to $\Zhat$, and so we obtain an action of $\Zhat$ on $\Qab$. Let $H$ be an open subgroup of $\GL{\Zhat}$. Then, as in the case of finite level, the Stein factorization of $X_{H}$ is given by
\[
	X_{H} \to \Spec \! \left((\Qab)^{\det H}\right) \to \Spec \Q.
\]
In particular, the number of geometric components of $X_{H}$ is equal to the index $[\Zhat* : \det H]$.

The inclusion and conjugation morphisms also have similar moduli interpretations using Theorem \ref{thm:modular_curves:mod_n_profinite_equivalence}. For instance, given two open subgroups $H$ and $H'$ of $\GL{\Zhat}$, with $H \leq H'$, the inclusion morphism $X_{H} \to X_{H'}$ is given, on non-cuspidal geometric points, by the map
\begin{align*}
	Y_{H}(\Qbar)      & \to Y_{H'}(\Qbar)           \\
	[(E, [\alpha]_{H})] & \mapsto [(E, [\alpha]_{H'})].
\end{align*}
Moreover, the morphism $X_{H} \to X_{H'}$ is finite locally free of degree $[\pm H' : \pm H]$.
	\section{Isolated points on modular curves} \label{sec:isolated_points_modular_curves}

Equipped with the moduli interpretation of points on modular curves, we can now apply the results of Section \ref{sec:isolated_divisors} to the specific case of modular curves. This will then allow us to prove the single-sink theorem for isolated points on modular curves, as well as the other results which form the backbone of our method for finding isolated points on modular curves.

Firstly, we reformulate the degree conditions of the mapping theorems, Theorems \ref{thm:isolated_divisors:pullback_isolated_point} and \ref{thm:isolated_divisors:pushforward_isolated_point}, solely in terms of the images of Galois representations of elliptic curves. This will streamline their application throughout the remainder of this section. For Theorem \ref{thm:isolated_divisors:pullback_isolated_point}, we obtain the following.

\begin{theorem} \label{thm:isolated_points_modular_curves:pullback_isolated}
	Let $H \leq H' \leq \GL{\Zhat}$ be open subgroups, and let $f : X_{H} \to X_{H'}$ be the inclusion morphism of modular curves. Let $x \in X_{H}$ be a non-cuspidal closed point with minimal representative $(E, [\alpha]_{H})$. Suppose that $[G_{E, \alpha} \cap A_{E, \alpha} H' : G_{E, \alpha} \cap A_{E, \alpha} H] = [\pm H' : \pm H]$. Then, if $x$ is isolated, so is $f(x)$.
\end{theorem}

\begin{proof}
	The degree of the inclusion morphism $f : X_{H} \to X_{H'}$ is equal to $[\pm H' : \pm H]$. Moreover, by Theorem \ref{thm:modular_curves:point_degree}, the quotient $\deg(x) / \deg(f(x))$ is equal to
	\[
		\frac{\deg(x)}{\deg(f(x))} = [G_{E, \alpha} \cap A_{E, \alpha} H' : G_{E, \alpha} \cap A_{E, \alpha} H].
	\]
	Therefore, the result follows from Theorem \ref{thm:isolated_divisors:pullback_isolated_point}.
\end{proof}

Similarly, we can reformulate Theorem \ref{thm:isolated_divisors:pushforward_isolated_point} to obtain the following.

\begin{theorem} \label{thm:isolated_points_modular_curves:pushforward_isolated}
	Let $H \leq H' \leq \GL{\Zhat}$ be open subgroups, and let $f : X_{H} \to X_{H'}$ be the inclusion morphism of modular curves. Let $x \in X_{H}$ be a non-cuspidal closed point with minimal representative $(E, [\alpha]_{H})$. Suppose that $[G_{E, \alpha} \cap A_{E, \alpha} H' : G_{E, \alpha} \cap A_{E, \alpha} H] = 1$. Then, if $f(x)$ is isolated, so is $x$.
\end{theorem}

\begin{proof}
	By Theorem \ref{thm:modular_curves:point_degree}, the quotient $\deg(x) / \deg(f(x))$ is equal to
	\[
		\frac{\deg(x)}{\deg(f(x))} = [G_{E, \alpha} \cap A_{E, \alpha} H' : G_{E, \alpha} \cap A_{E, \alpha} H].
	\]
	Therefore, the result follows from Theorem \ref{thm:isolated_divisors:pushforward_isolated_point}.
\end{proof}

By leveraging both of these results, we can now prove the single-sink theorem for isolated points on modular curves. For a more thorough discussion of this theorem and its importance, see Section \ref{sec:introduction:single_sink}.

\begin{theorem}[Single-sink theorem] \label{thm:isolated_points_modular_curves:galois_image_isolated}
	Let $H \leq \GL{\Zhat}$ be an open subgroup. Let $x \in X_{H}$ be a non-cuspidal, non-CM closed point with minimal representative $(E, [\alpha]_{H})$. Suppose that $x$ is isolated. Then the closed point $y \in X_{G_{E, \alpha}}$ corresponding to the $G_{\Q}$-orbit of the point $[(E, [\alpha]_{G_{E, \alpha}})] \in Y_{G_{E, \alpha}}(\Qbar)$ is also isolated.
\end{theorem}

\begin{proof}
	By Remark \ref{rmk:modular_curves:serre_open_image}, the extended adelic Galois image $G_{E, \alpha}$ is an open subgroup of $\GL{\Zhat}$. Therefore, we can consider the following maps of modular curves:
	\[\begin{tikzcd}
		X_{H} & X_{H \cap G_{E, \alpha}} \arrow[l, "f"'] \arrow[r, "g"] & X_{G_{E, \alpha}}.
	\end{tikzcd}\]
	Let $z \in X_{H \cap G_{E, \alpha}}$ be the closed point corresponding to the $\Qbar$-orbit of the point $(E, [\alpha]_{G_{E, \alpha}})$. Note that, by definition, we have $f(z) = x$ and $g(z) = y$.
	
	Since $A_{E, \alpha} \leq G_{E, \alpha}$, by Lemma \ref{thm:group_theory:modular_law_subsets}, we have
	\begin{align*}
		[G_{E, \alpha} \cap A_{E, \alpha} H : G_{E, \alpha} \cap A_{E, \alpha} (H \cap G_{E, \alpha})] &= [G_{E, \alpha} \cap A_{E, \alpha} H : G_{E, \alpha} \cap A_{E, \alpha} H] \\
		&= 1.
	\end{align*}
	Thus, by Theorem \ref{thm:isolated_points_modular_curves:pushforward_isolated}, since $x$ is isolated, so is $z$. By assumption, the elliptic curve $E$ does not have complex multiplication. In particular, $j(E) \neq 0$ and $j(E) \neq 1728$, and so $A_{E, \alpha} = \{\pm I\}$. Therefore, we have
	\begin{align*}
		[G_{E, \alpha} \cap A_{E, \alpha} G_{E, \alpha} : G_{E, \alpha} \cap A_{E, \alpha} (H \cap G_{E, \alpha})] &= [G_{E, \alpha} : G_{E, \alpha} \cap \pm H] \\
		&= [\pm G_{E, \alpha} : \pm(G_{E, \alpha} \cap H)].
	\end{align*}
	Therefore, by Theorem \ref{thm:isolated_points_modular_curves:pullback_isolated}, since $z$ is isolated, so is $y$.
\end{proof}

As was discussed in Section \ref{sec:introduction:finding_isolated_points}, our method for finding isolated points on modular curves relies on generalizations of the above single-sink theorem. These generalizations exploit extra information about the subgroup $H$ to show the existence of isolated points on modular curves corresponding to overgroups of the extended adelic Galois image of $E$, as follows. This also provides a generalization to non-branch CM-points.

\begin{theorem} \label{thm:isolated_points_modular_curves:normal_galois_image_isolated}
	Let $H \leq \GL{\Zhat}$ be an open subgroup, and let $N \lhd \GL{\Zhat}$ be a normal open subgroup such that $N \leq H$. Let $x \in X_{H}$ be a non-cuspidal closed point with minimal representative $(E, [\alpha]_{H}) \in Y_{H}(\Qbar)$ and $j(E) \notin \{0, 1728\}$. Suppose that $x$ is isolated. Then the closed point $y \in X_{N G_{E, \alpha}}$ corresponding to the $G_{\Q}$-orbit of the point $[(E, [\alpha]_{N G_{E, \alpha}})] \in Y_{N G_{E, \alpha}}(\Qbar)$ is also isolated.
\end{theorem}

\begin{proof}
	As $N$ is an open normal subgroup of $\GL{\Zhat}$, the product $N G_{E, \alpha}$ is an open subgroup of $\GL{\Zhat}$. Therefore, we may consider the following maps of modular curves:
	\[\begin{tikzcd}
		X_{H} & X_{H \cap N G_{E, \alpha}} \arrow[l, "f"'] \arrow[r, "g"] & X_{N G_{E, \alpha}}.
	\end{tikzcd}\]
	Let $z \in X_{H \cap N G_{E, \alpha}}$ be the closed point corresponding to the $\Qbar$-orbit of the point $(E, [\alpha]_{H \cap N G_{E, \alpha}}) \in Y_{H \cap N G_{E, \alpha}}(\Qbar)$. By definition, we have $f(z) = x$ and $g(z) = y$.
	
	We have
	\begin{align*}
		[G_{E, \alpha} & \cap A_{E, \alpha} H : G_{E, \alpha} \cap A_{E, \alpha} (H \cap N G_{E, \alpha})] \\
		&= [G_{E, \alpha} \cap A_{E, \alpha} H : G_{E, \alpha} \cap A_{E, \alpha} H] \\
		&= 1.
	\end{align*}
	Thus, by Theorem \ref{thm:isolated_points_modular_curves:pushforward_isolated}, since $x$ is isolated, so is $z$.
	
	By assumption, we have $j(E) \notin \{0, 1728\}$, and so $A_{E, \alpha} = \{\pm I\}$. Therefore, we have
	\begin{align*}
		[G_{E, \alpha} \cap A_{E, \alpha} (N G_{E, \alpha}) : G_{E, \alpha} \cap A_{E, \alpha} (H \cap N G_{E, \alpha})] &= [G_{E, \alpha} : G_{E, \alpha} \cap \pm H].
	\end{align*}
	As $N \leq H$, by Lemmas \ref{thm:group_theory:modular_law_subsets} and \ref{thm:group_theory:product_preserves_index}, we have
	\[
		[G_{E, \alpha} : G_{E, \alpha} \cap \pm H] = [N G_{E, \alpha} : N G_{E, \alpha} \cap \pm H] = [\pm N G_{E, \alpha} : \pm(N G_{E, \alpha} \cap H)].
	\]
	Therefore, by Theorem \ref{thm:isolated_points_modular_curves:pullback_isolated}, since $z$ is isolated, so is $x$.
\end{proof}

In the case that $N = \ker \pi_{n}$, and noting the correspondence between profinite and mod-$n$ level structures, we obtain the following corollary, which was given as Theorem \ref{thm:introduction:mod_n_image_isolated}.

\begin{corollary} \label{thm:isolated_points_modular_curves:mod_n_galois_image_isolated}
	Let $H \leq \GL{\Zhat}$ be an open subgroup of level $n$. Let $x \in X_{H}$ be a non-cuspidal closed point with minimal representative $(E, [\alpha]_{H})$ and $j(E) \notin \{0, 1728\}$. Suppose that $x$ is isolated. Then the closed point $y \in X_{G_{E, \alpha_{n}}}$ corresponding to the $G_{\Q}$-orbit of the point $[(E, [\alpha_{n}]_{G_{E, \alpha_{n}}})] \in Y_{G_{E, \alpha_{n}}}(\Qbar)$ is also isolated.
\end{corollary}

As explained in Section \ref{sec:introduction:finding_isolated_points}, we also require an analogous result with the roles of $H$ and $G_{E, \alpha}$ interchanged.

\begin{theorem} \label{thm:isolated_points_modular_curves:normal_H_isolated}
	Let $H \leq \GL{\Zhat}$ be an open subgroup. Let $x \in X_{H}$ be a non-cuspidal closed point with minimal representative $(E, [\alpha]_{H})$ and $j(E) \notin \{0, 1728\}$. Let $N \lhd \GL{\Zhat}$ be a normal subgroup such that $N \leq G_{E, \alpha}$. Suppose that $x$ is isolated. Then the closed point $y \in X_{H N}$ corresponding to the $G_{\Q}$-orbit of the point $[(E, [\alpha]_{H N})] \in Y_{H N}(\Qbar)$ is also isolated.
\end{theorem}

\begin{proof}
	Since $N$ is a normal subgroup of $\GL{\Zhat}$, and $H$ an open subgroup of $\GL{\Zhat}$, the product $H N$ is also an open subgroup of $\GL{\Zhat}$. Therefore, consider the inclusion morphism of modular curves $f : X_{H} \to X_{H N}$. Note that, by definition, we have $f(x) = y$. By assumption, we have $j(E) \notin \{0, 1728\}$, and so $A_{E, \alpha} = \{\pm I\}$. Thus, we have
	\[
		[G_{E, \alpha} \cap A_{E, \alpha} H N : G_{E, \alpha} \cap A_{E, \alpha} H] = [G_{E, \alpha} \cap \pm H N : G_{E, \alpha} \cap \pm H].
	\]
	Since $N \leq G_{E, \alpha}$, by Lemma \ref{thm:group_theory:modular_law_subsets}, we have
	\[
		[G_{E, \alpha} \cap \pm H N : G_{E, \alpha} \cap \pm H] = [N(G_{E, \alpha} \cap \pm H) : G_{E, \alpha} \cap \pm H].
	\]
	By Lemma \ref{thm:group_theory:product_preserves_index}, we obtain that
	\begin{align*}
		[N(G_{E, \alpha} \cap \pm H) : G_{E, \alpha} \cap \pm H] & = [N(G_{E, \alpha} \cap \pm H) : (N \cap \pm H) (G_{E, \alpha} \cap \pm H)] \\
		& = [N : N \cap \pm H],
	\end{align*}
	and similarly,
	\begin{align*}
		[N : N \cap \pm H] = [\pm H N : \pm H (N \cap \pm H)] = [\pm H N : \pm H].
	\end{align*}
	Therefore, we have $[G_{E, \alpha} \cap A_{E, \alpha} H N : G_{E, \alpha} \cap A_{E, \alpha} H] = [\pm H N : \pm H]$. Thus, by Theorem \ref{thm:isolated_points_modular_curves:pullback_isolated}, since $x$ is isolated, so is $y$.
\end{proof}

As before, when $N = \ker \pi_{n}$, we obtain the following corollary, which is an analogue of \cite[Theorem 5.1]{bourdon2019}.

\begin{corollary} \label{thm:isolated_points_modular_curves:mod_n_H_isolated}
	Let $H \leq \GL{\Zhat}$ be an open subgroup. Let $x \in X_{H}$ be a non-cuspidal closed point with minimal representative $(E, [\alpha]_{H})$ and $j(E) \notin \{0, 1728\}$. Let $n$ be the level of the extended adelic Galois image $G_{E, \alpha}$. Suppose that $x$ is isolated. Then the closed point $y \in X_{\pi_{n}(H)}$ corresponding to the $G_{\Q}$-orbit of the point $[(E, [\alpha_{n}]_{\pi_{n}(H)})] \in Y_{\pi_{n}(H)}(\Qbar)$ is also isolated.
\end{corollary}
	\section{Isolated points on curves of level 7} \label{sec:level_7}

We now turn our attention to applications of the method presented in Section \ref{sec:introduction:finding_isolated_points} to specific classes of modular curves. In this section, we aim to compute all of the non-cuspidal, non-CM isolated points with rational $j$-invariant on all modular curves of level 7. The reasons behind the choice of this example are two-fold. Firstly, while there has been much study on the set of isolated points on the modular curves $X_{1}(n)$, much less has been made of other modular curves. This example serves to demonstrate the generality of the techniques developed above, as well as some of the interesting phenomena which can occur when studying more general modular curves.

Secondly, while the calculations and computations required for this example are quite short and straightforward, the example illustrates a general method for finding isolated points which should be applicable to modular curves of higher level, and $j$-invariants of larger degree. In this spirit, we shall try to provide some more refined approaches to computations than might be strictly necessary for this example.

We begin by stating the target result: the classification of all non-cuspidal, non-CM isolated points with rational $j$-invariant on all modular curves of level 7.

\begin{theorem} \label{thm:level_7:level_7_isolated_points}
	Let $H \leq \GL{\Zhat}$ be an open subgroup of level 7 such that $-I \in H$. Let $x \in X_{H}$ be a non-cuspidal, non-CM isolated closed point with minimal representative $(E, [\alpha]_{H})$, such that $j(E) \in \Q$. Then $j(E) = \frac{3^{3} \cdot 5 \cdot 7^{5}}{2^{7}}$, and $H$ and $x$ are as described in Table \ref{tbl:level_7:isolated_points_level_7}.
\end{theorem}

\begin{table}
	\centering
	\label{tbl:level_7:isolated_points_level_7}
	\caption{The non-cuspidal, non-CM isolated points with rational $j$-invariant on modular curves of level 7. Each row corresponds to a set of isolated points, of same degree, on a conjugacy class of modular curves $X_{H}$. A representative of the conjugacy class of $H$, as a subgroup of $\GL{\Z[7]}$, is given, as well as the genus and number of geometric components of $X_{H}$. The degree of the isolated point, as well as the number of isolated points with the same degree on the curve are also given. All points have $j$-invariant equal to $\frac{3^{3} \cdot 5 \cdot 7^{5}}{2^{7}}$.}
	\begin{tabular}{lrrrr}
		$H$ & $g(X_{H})$ & $[\Z*[7] : \det H]$ & $\deg(x)$ & Number of $x$ \\ \toprule
		$\{\pm I\}$ & 3 & 6 & 18 & 56 \\
		$\langle \pm \left( \begin{smallmatrix} 1 & 0 \\ 2 & 6 \end{smallmatrix} \right) \rangle$ & 3 & 3 & 9 & 6 \\
		$\langle \pm \left( \begin{smallmatrix} 2 & 0 \\ 0 & 2 \end{smallmatrix} \right) \rangle$ & 3 & 2 & 6 & 56 \\
		$\langle \pm \left( \begin{smallmatrix} 3 & 6 \\ 6 & 1 \end{smallmatrix} \right) \rangle$ & 3 & 2 & 6 & 2 \\
		$\langle \pm \left( \begin{smallmatrix} 1 & 2 \\ 2 & 5 \end{smallmatrix} \right) \rangle$ & 1 & 6 & 6 & 2 \\
		$\langle \pm \left( \begin{smallmatrix} 1 & 6 \\ 6 & 6 \end{smallmatrix} \right) \rangle$ & 3 & 1 & 3 & 6 \\
		$\langle \pm \left( \begin{smallmatrix} 2 & 2 \\ 2 & 6 \end{smallmatrix} \right), \left( \begin{smallmatrix} 1 & 0 \\ 2 & 6 \end{smallmatrix} \right) \rangle$ & 1 & 3 & 3 & 2 \\
		$\langle \pm \left( \begin{smallmatrix} 2 & 2 \\ 2 & 6 \end{smallmatrix} \right), \left( \begin{smallmatrix} 2 & 0 \\ 0 & 2 \end{smallmatrix} \right) \rangle$ & 1 & 2 & 2 & 2 \\
		$\langle \pm \left( \begin{smallmatrix} 2 & 0 \\ 0 & 1 \end{smallmatrix} \right), \left( \begin{smallmatrix} 0 & 5 \\ 6 & 0 \end{smallmatrix} \right) \rangle$ & 1 & 1 & 1 & 2 \\ \bottomrule
	\end{tabular}
\end{table}

The proof of this result follows the method outlined in Section \ref{sec:introduction:finding_isolated_points}. Firstly, in Section \ref{sec:level_7:determining_j}, we employ suitable generalizations of the single-sink theorem, as well as the classification of images of mod-$\ell$ Galois representations given in \cite{zywina2015}, to show that $j(E)$ must equal $\frac{3^{3} \cdot 5 \cdot 7^{5}}{2^{7}}$.

Secondly, we find a finite subset $S'$ of the set of modular curves of level 7 such that the proof of Theorem \ref{thm:level_7:level_7_isolated_points} reduces to finding the isolated points with the above $j$-invariant on the modular curves of $S'$. This step differs somewhat from the exposition presented in Section \ref{sec:introduction:finding_isolated_points}. Indeed, the method presented in the aforementioned section suggests simply taking $S'$ to be the set of all modular curves of level 7, which is finite. However, in order to reduce the size of $S'$ further, we make direct use of the mapping theorems given in Section \ref{sec:isolated_divisors:isolated_points}. In conjunction with the results of Section \ref{sec:modular_curves:closed_points_as_double_cosets}, this yields a computationally  effective approach to reducing the size of the set $S'$, which is detailed in Section \ref{sec:level_7:restricting_H_and_x}. In this case, we show that the set $S'$ can be taken to consist of a single modular curve, namely the modular curve $X(7)$.

Finally, in Section \ref{sec:level_7:points_are_isolated}, we utilize an explicit model for $X(7)$ given in \cite{halberstadt2003} to determine the isolated points on this modular curve.

\subsection{Determining \texorpdfstring{$j(E)$}{j(E)}} \label{sec:level_7:determining_j}

The first step in the proof of Theorem \ref{thm:level_7:level_7_isolated_points} is to restrict the possible $j$-invariants of isolated points on the modular curves of level 7. As explained in Section \ref{sec:introduction:finding_isolated_points}, this is enabled by generalizations of the single-sink theorem, which give a link between the aforementioned problem and the problem of classifying images of mod-7 Galois representations of elliptic curves. In the case of elliptic curves over $\Q$, this is determined in \cite{zywina2015}. We summarize the relevant information in the following theorem.

\begin{theorem}[{\cite[Theorem 1.5]{zywina2015}}] \label{thm:level_7:mod_7_images}
	Let $E/\Q$ be a non-CM elliptic curve, and $\alpha : E[7] \to (\Z[7])^{2}$ be a level 7 structure on $E$. Then, the extended mod-7 Galois image $G_{E, \alpha}$ is conjugate to one of the following subgroups of $\GL{\Z[7]}$:
	\begin{align*}
		G_1 & := \left\langle \begin{pmatrix} 2 & 0 \\ 0 & 4 \end{pmatrix}, \begin{pmatrix} 0 & 2 \\ 1 & 0 \end{pmatrix}, \begin{pmatrix} -1 & 0 \\ 0 & -1 \end{pmatrix} \right\rangle, \\
		G_2 & := \left\{\begin{pmatrix} a & 0 \\ 0 & b \end{pmatrix}, \begin{pmatrix} 0 & a \\ b & 0 \end{pmatrix} : a \in \Z*[7], b \in \Z*[7]\right\}, \\
		G_3 & := \left\{\pm \begin{pmatrix} 1 & a \\ 0 & b \end{pmatrix} : a \in \Z[7], b \in \Z*[7]\right\}, \\
		G_4 & := \left\{\pm \begin{pmatrix} a & b \\ 0 & 1 \end{pmatrix} : a \in \Z*[7], b \in \Z[7]\right\}, \\
		G_5 & := \left\{\begin{pmatrix} a & b \\ 0 & \pm a \end{pmatrix} : a \in \Z*[7], b \in \Z[7]\right\}, \\
		G_6 & := \left\{\begin{pmatrix} a & -b \\ b & a \end{pmatrix}, \begin{pmatrix} a & b \\ b & -a \end{pmatrix} : (a, b) \in (\Z[7])^{2} \setminus \{(0, 0)\}\right\}, \\
		G_7 & := \left\{\begin{pmatrix} a & b \\ 0 & c \end{pmatrix} : a \in \Z*[7], b \in \Z[7], c \in \Z*[7]\right\}, \\
		G_8 & := \GL{\Z[7]}.
	\end{align*}
	Moreover, if $G_{E, \alpha}$ is conjugate to $G_{1}$, then $j(E) = \frac{3^{3} \cdot 5 \cdot 7^{5}}{2^{7}}$.
\end{theorem}

This classification allows us to prove the first part of Theorem \ref{thm:level_7:level_7_isolated_points}, as follows.

\begin{theorem} \label{thm:level_7:j_invariant_restriction}
	Let $H \leq \GL{\Zhat}$ be an open subgroup of level 7 such that $-I \in H$. Let $x \in X_{H}$ be a non-cuspidal, non-CM isolated closed point with minimal representative $(E, [\alpha]_{H})$, such that $j(E) \in \Q$. Then $j(E) = \frac{3^{3} \cdot 5 \cdot 7^{5}}{2^{7}}$.
\end{theorem}

\begin{proof}
	Since $H$ has level 7, by Corollary \ref{thm:isolated_points_modular_curves:mod_n_galois_image_isolated}, we know that the closed point $y \in X_{G_{E, \alpha_{7}}}$ corresponding to the $G_{\Q}$-orbit of the point $[(E, [\alpha_{7}]_{G_{E, \alpha_{7}}})] \in X_{G_{E, \alpha_{7}}}$ is isolated.
	
	By Theorem \ref{thm:level_7:mod_7_images}, we know that the extended mod-7 Galois image $G_{E, \alpha_{7}}$ must be conjugate to one of the groups $G_{1}, \dots G_{8}$. Since $\det(G_i) = \Z*[7]$ for all $i$, the modular curves $X_{G_{i}}$ are all geometrically connected. Moreover, using Remark \ref{rmk:modular_curves:geometric_components}, it is straightforward to compute that the modular curves $X_{G_{i}}$, for $2 \leq i \leq 8$, are curves of genus 0. Therefore, by Theorem \ref{thm:isolated_divisors:riemann_roch_criterion}, there are no \Pone-isolated points on these modular curves. Since $y \in X_{G_{E, \alpha_{7}}}$ is isolated, it follows that $G_{E, \alpha_{7}}$ is conjugate to $G_{1}$. Thus, by Theorem \ref{thm:level_7:mod_7_images}, we obtain that $j(E) = \frac{3^{3} \cdot 5 \cdot 7^{5}}{2^{7}}$.
\end{proof}

\subsection{Restricting \texorpdfstring{$H$ and $x$}{H and x}} \label{sec:level_7:restricting_H_and_x}

We have now shown that the $j$-invariant of any non-cuspidal, non-CM isolated point with rational $j$-invariant on a modular curve of level 7 must be equal to $\frac{3^{3} \cdot 5 \cdot 7^{5}}{2^{7}}$. The second step is to reduce the proof of Theorem \ref{thm:level_7:level_7_isolated_points} to the computation of the isolated point with the given $j$-invariant on a finite set of modular curves. While the set of all modular curves of level 7 is finite, we aim to restrict the size of this set as much as possible, to simplify later computations. To do so, we rely on the results from Section \ref{sec:modular_curves:closed_points_as_double_cosets}, which allow us to compute all of the closed points on modular curves of level 7 lying above the given $j$-invariant, along with some of their properties. These properties will moreover allow us to show that many of these closed points are not isolated. In particular, we shall show the following.

\begin{theorem} \label{thm:level_7:restricting_H_and_x}
	Let $H \leq \GL{\Zhat}$ be an open subgroup of level 7 such that $-I \in H$. Let $x \in X_{H}$ be a non-cuspidal, non-CM isolated closed point with minimal representative $(E, [\alpha]_{H})$, with $j(E) = \frac{3^{3} \cdot 5 \cdot 7^{5}}{2^{7}}$. Then $H$ and $x$ must be as in Table \ref{tbl:level_7:isolated_points_level_7}.
\end{theorem}

We begin by enumerating the possible subgroups $H$ and closed points $x$. Throughout, we fix an elliptic curve $E/\Q$ with $j$-invariant $\frac{3^{3} \cdot 5 \cdot 7^{5}}{2^{7}}$, and a level-7 structure $\alpha$ on $E$ such that the extended Galois image $G_{E, \alpha}$ is equal to the subgroup $G_{1}$ given in Theorem \ref{thm:level_7:mod_7_images}. A precise description of $E$ and $\alpha$ is unnecessary, as, by Theorem \ref{thm:modular_curves:closed_points_as_double_cosets}, the structure of the closed points lying above the given $j$-invariant is determined solely by $G_{E, \alpha}$.

Using \texttt{Magma}, we find that there are 998 subgroups of $\GL{\Z[7]}$ containing $-I$, distributed amongst 53 conjugacy classes. For each such subgroup $H$, the closed points of $X_{H}$ lying above the $j$-invariant $\frac{3^{3} \cdot 5 \cdot 7^{5}}{2^{7}}$ are in bijection with the double cosets $HgG_{E, \alpha}$, by Theorem \ref{thm:modular_curves:closed_points_as_double_cosets}. We enumerate these using \texttt{Magma} as well, and find that there are 12690 closed points lying above the modular curves corresponding to these 998 subgroups of $\GL{\Z[7]}$.

This number is computationally tractable, and it would be possible to iterate through all 12690 closed points and compute whether each is possibly isolated. However, in order to make the computations in the next section lighter, we use a more enlightening approach. This relies on the notion of an isolation graph, which we define now.

\begin{definition} \label{def:level_7:isolation_graph}
	Let $\Omega$ be a set of curves over a number field $k$. Let $\Sigma$ be a set of closed points $x \in C$, where $C$ belongs to $\Omega$, and $\Lambda$ a set of finite locally free maps $f : C \to C'$ between elements of $\Omega$. The \textbf{isolation graph associated to $\Omega$, $\Sigma$ and $\Lambda$} is the directed graph whose vertices are the closed points of $\Sigma$, and whose edges are given as follows: for two closed points $x$ and $y \in \Sigma$, lying on the curves $C$ and $C' \in \Omega$ respectively, there is an edge $x \to y$ if either
	\begin{itemize}
		\item there exists a map $f : C \to C'$ in $\Lambda$ such that $f(x) = y$ and $\deg_{k}(x) = \deg(f) \cdot \deg_{k}(y)$, or
		\item there exists a map $f : C' \to C$ in $\Lambda$ such that $f(y) = x$ and $\deg_{k}(x) = \deg_{k}(y)$.
	\end{itemize}
\end{definition}

An isolation graph encodes the relationships between isolated points which were proved in Theorems \ref{thm:isolated_divisors:pullback_isolated_point} and \ref{thm:isolated_divisors:pushforward_isolated_point}. Namely, if an isolation graph contains an edge $x \to y$, then if the closed point $x$ is isolated, so is the closed point $y$.

In this case, we let $\Omega$ be the set of modular curves $X_{H}$, where $H$ is a subgroup of $\GL{\Z[7]}$ containing $-I$. The set $\Sigma$ is the set of closed points on these modular curves with $j$-invariant equal to $\frac{3^{3} \cdot 5 \cdot 7^{5}}{2^{7}}$. Finally, the set $\Lambda$ is the set of inclusion morphisms $X_{H} \to X_{H'}$, where $H$ is a maximal subgroup of $H'$. We can restrict ourselves to considering maximal subgroups, since the composition of two inclusion morphisms is itself an inclusion morphism.

The isolation graph associated to $\Omega$, $\Sigma$ and $\Lambda$ is then straightforward to compute. The vertices of the graph can be labeled by pairs $(H, H g G_{E, \alpha})$, where $H$ is the subgroup of $\GL{\Z[7]}$ defining the modular curve, and $H g G_{E, \alpha}$ the double coset representing the closed point on $X_{H}$. To compute the edges of the isolation graph, we iterate through all of the pairs of subgroups $H$ and $H'$, where $H$ is a maximal subgroup of $H'$, as well as all the closed points $(H, H g G_{E, \alpha})$ on $X_{H}$. By Theorem \ref{thm:modular_curves:closed_points_as_double_cosets}, the inclusion morphism $f : X_{H} \to X_{H'}$ maps the point $(H, H g G_{E, \alpha})$ to the point $(H', H' g G_{E, \alpha})$. The degrees of the two closed points can also be readily computed using Theorem \ref{thm:modular_curves:closed_points_as_double_cosets}, while the degree of $f$ is given simply by the index $[\pm H' : \pm H]$. Thus, we can quickly check whether either of the two required conditions hold, and if so, add the appropriate edge to the isolation graph.

Implementing this procedure, we obtain a directed graph with 12690 vertices and 71235 edges. Once again, this is certainly computationally tractable. However, with an eye on future applications, we provide a simplification which drastically reduces the size of the isolation graph.

One may note that we have only used the inclusion morphisms at present, while there also exist conjugation isomorphisms between modular curves. Given a subgroup $H$ of $\GL{\Z[7]}$ and an element $h \in \GL{\Z[7]}$, the conjugation isomorphism $f : X_{H} \to X_{h H h^{-1}}$ maps the point $(H, H g G_{E, \alpha})$ to the point $(h H h^{-1}, (h H h^{-1}) h g G_{E, \alpha})$, by Theorem \ref{thm:modular_curves:closed_points_as_double_cosets}. The morphism $f$ is an isomorphism, and so, if one of these two points is isolated, so is the other. Thus, the conjugation morphisms allow us to identify closed points with the same isolated-ness.

It is possible to keep track of these morphisms by extending the set $\Lambda$ to contain all conjugation morphisms between modular curves of level 7. The sets of closed points with the same isolated-ness are then represented by the strongly connected components of this new isolation graph. One can now compute the condensation of the isolation graph, that is to say, the graph obtained by contracting each strongly connected component to a single vertex. Each vertex in this new graph then represents a set of closed points with the same isolated-ness, while the edges of the graph still encode the relations of Theorems \ref{thm:isolated_divisors:pullback_isolated_point} and \ref{thm:isolated_divisors:pushforward_isolated_point}.

While this condensation is a much smaller graph, computing it as above still requires creating the full isolation graph with 12690 vertices. To sidestep this, we provide an alternative understanding of the conjugation morphisms, which allows us to directly compute this condensation.

\begin{lemma} \label{thm:level_7:conjugation_graph_automorphism}
	Let $\Omega$, $\Sigma$ and $\Lambda$ be as defined above, and let $h$ be an element of $\GL{\Z[7]}$. Then the conjugation map defined by
	\[
		(H, H g G_{E, \alpha}) \mapsto (h H h^{-1}, (h H h^{-1}) h g G_{E, \alpha}),
	\]
	for all subgroups $H$ of $\GL{\Z[7]}$ containing $-I$ and all $g \in \GL{\Z[7]}$, is an automorphism of the isolation graph associated to $\Omega$, $\Sigma$ and $\Lambda$.
\end{lemma}

\begin{proof}
	The conjugation map is a bijection on the set of vertices of the isolation graph, as it has inverse given by conjugation by $h^{-1}$. Therefore, it suffices to show that, for any edge $x \to y$ in the isolation graph, there is an edge $h x h^{-1} \to h y h^{-1}$. Let $(H, H g G_{E, \alpha}) \to (H', H' g' G_{E, \alpha})$ be an edge in the isolation graph. By definition of the set $\Lambda$, one of $H$ and $H'$ must be a subgroup of the other, and we can take $g = g'$. Conjugation by $h$ preserves these properties, and it is straightforward to check that it also preserves the degree of the closed points and the degree of the morphism between the two modular curves. Thus, we must also have an edge $(h H h^{-1}, (h H h^{-1}) h g G_{E, \alpha}) \to (h H h^{-1}, (h H' h^{-1}) h g' G_{E, \alpha})$ in the isolation graph, as required.
\end{proof}

This result gives an action of $\GL{\Z[7]}$ on the isolation graph by automorphisms. As in the previous discussion, the existence of conjugation isomorphisms between the modular curves shows that two vertices in the same orbit under these automorphisms must have the same isolated-ness. Therefore, rather than computing the condensation of a larger isolation graph, one can instead compute the quotient of this isolation graph under these automorphisms. As before, each vertex in the quotient graph represents a set of closed points which share the same isolated-ness, while the edges still represent the relations of Theorems \ref{thm:isolated_divisors:pullback_isolated_point} and \ref{thm:isolated_divisors:pushforward_isolated_point}.

In order to compute this quotient graph directly, we first give an explicit description of the vertices of the graph.

\begin{lemma} \label{thm:level_7:vertices_quotient_graph}
	Let $\Omega$, $\Sigma$ and $\Lambda$ be as defined above, and consider the quotient of the isolation graph associated to $\Omega$, $\Sigma$ and $\Lambda$ described above. Then, there is a bijection between the vertices of this quotient graph and the pairs
	\[
		(H, N_{H} g G_{E, \alpha}),
	\]
	where $H$ ranges through a set of representatives for the conjugacy classes of subgroups of $\GL{\Z[7]}$ containing $-I$, $N_{H}$ is the normalizer of $H$ in $\GL{\Z[7]}$, and $g \in \GL{\Z[7]}$.
\end{lemma}

\begin{proof}
	Let $H$ and $H'$ be two conjugate subgroups of $\GL{\Z[7]}$, and consider a vertex $(H, H g G_{E, \alpha})$ of the inclusion graph. Since $H$ and $H'$ are conjugate, there exists $h \in \GL{\Z[7]}$ such that $h H h^{-1} = H'$. The conjugation by $h$ automorphism then sends the vertex $(H, H g G_{E, \alpha})$ to the vertex $(H', H' h g G_{E, \alpha})$. Thus, the orbits of the vertices of the inclusion graph are in bijection with the orbits of the vertices $(H, H g G_{E, \alpha})$, where $H$ ranges through a set of representatives for the conjugacy classes of subgroups of $\GL{\Z[7]}$ containing $-I$.
	
	Fix a subgroup $H$ of $\GL{\Z[7]}$ containing $-I$, and let $g$ and $g'$ be two elements of $\GL{\Z[7]}$. The two points $(H, H g G_{E, \alpha})$ and $(H, H g' G_{E, \alpha})$ are in the same orbit if and only if there exists $h \in \GL{\Z[7]}$ such that $h H h^{-1} = H$ and
	\[
		H g G_{E, \alpha} = H h g' G_{E, \alpha}.
	\]
	The second condition can be rewritten as $g \in H h g' G_{E, \alpha}$, and so the two aforementioned points are in the same orbit if and only if there exists $h \in N_{H}$ such that $g \in H h g' G_{E, \alpha}$. Since the normalizer $N_{H}$ contains $H$, this is equivalent to requiring that $g \in N_{H} g' G_{E, \alpha}$, or equivalently,
	\[
		N_{H} g G_{E, \alpha} = N_{H} g' G_{E, \alpha}.
	\]
	This gives the desired bijection.
\end{proof}

This result gives a simple method for computing the vertices of the quotient graph. The edges of the quotient graph can also be computed directly. Since we are taking quotients by graph automorphisms, the edges of the quotient graph starting at a given vertex are in bijection with the edges of the original isolation graph starting at a fixed lift of this vertex. Thus, for any vertex $(H, N_{H} g G_{E, \alpha})$ of the quotient graph, one can consider the lift $(H, H g G_{E, \alpha})$ in the isolation graph, and determine the edges starting from this vertex using the procedure described earlier. Then, for each edge $(H, H g G_{E, \alpha}) \to (H', H' g G_{E, \alpha})$ found, it suffices to find the image of $(H', H' g G_{E, \alpha})$ in the quotient graph, which can be done using the description in the proof of Lemma \ref{thm:level_7:vertices_quotient_graph}.

This procedure is implemented in the file \texttt{level\char`_7\char`_quotient\char`_graph.m}. In this case, we obtain a graph with 252 vertices and 779 edges, a marked reduction from the original isolation graph. This reduction comes primarily from the need to only consider 53 conjugacy classes of subgroups, rather than all 998 subgroups of $\GL{\Z[7]}$ containing $-I$.

\begin{figure}
	\centering
	\includegraphics[width=330pt]{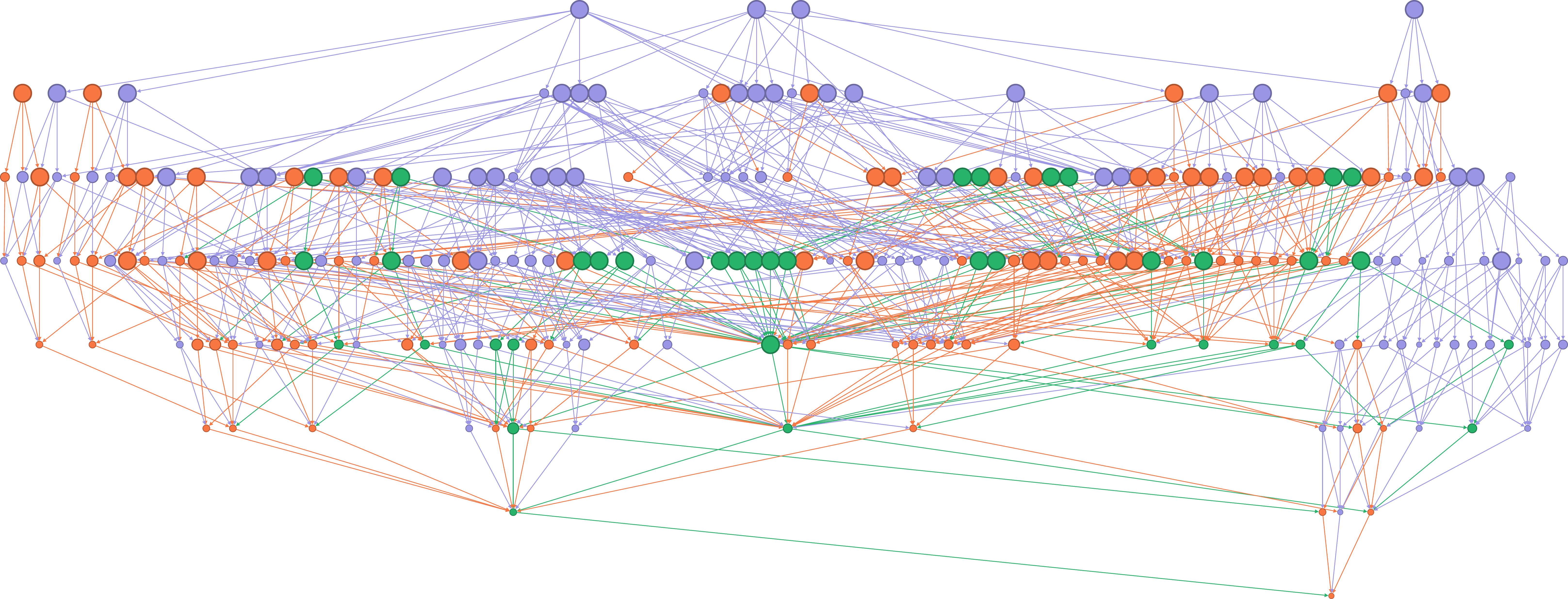}
	\caption{The quotient of the isolation graph for modular curves of level 7 and closed points with $j$-invariant $\frac{3^{3} \cdot 5 \cdot 7^{5}}{2^{7}}$. Each vertex represents a set of closed points on such modular curves which have the same isolated-ness, while edges represent the relations of Theorems \ref{thm:isolated_divisors:pullback_isolated_point} and \ref{thm:isolated_divisors:pushforward_isolated_point}. The size of the nodes represents the degree of the closed points, while the color of the nodes indicates the genus of the underlying modular curve. The edges are colored according to the color of their source vertex. The graph is topologically sorted, so all edges go from top to bottom.}
	\label{fig:level_7:classes_isolation_graph}
\end{figure}

The resulting graph is given in Figure \ref{fig:level_7:classes_isolation_graph}. The graph is topologically sorted, such that all the edges go from top to bottom. This presentation immediately draws attention to a number of facts. Firstly, we note that the graph contains a unique terminal vertex. This vertex corresponds to the rational points of $X_{G_{E, \alpha}}$ with $j$-invariant equal to $\frac{3^{3} \cdot 5 \cdot 7^{5}}{2^{7}}$, and its existence is a direct consequence Corollary \ref{thm:isolated_points_modular_curves:mod_n_galois_image_isolated}. Secondly, the genus of the underlying curves does not monotonically decrease as the graph is traversed. This stems from the interplay between Theorem \ref{thm:isolated_divisors:pullback_isolated_point}, which maps isolated points on higher genus curves to points on lower genus curves, and Theorem \ref{thm:isolated_divisors:pushforward_isolated_point}, which works in the other direction. However, the degree of the closed points does decrease monotonically from top to bottom, which is a consequence of the degree conditions in these two theorems.

Equipped with this quotient graph, we now aim to show that many of the vertices in this graph correspond to closed points which are not isolated. Consider a vertex $(H, N_{H} g G_{E, \alpha})$ of the quotient graph. This vertex corresponds to a set of closed points with the same isolated-ness, containing, by the proof of Theorem \ref{thm:level_7:vertices_quotient_graph}, the closed point of $x \in X_{H}$ corresponding to the double coset $H g G_{E, \alpha}$. The degree of this closed point can be computed using Theorem \ref{thm:modular_curves:closed_points_as_double_cosets}, while the number and genus of the connected components of $X_{H}$ can be computed using Remark \ref{rmk:modular_curves:geometric_components}. If the degree of the closed point is strictly greater than the product of the number and genus of the connected components of $X_{H}$, then the closed point $x \in X_{H}$ is not \Pone-isolated, by Theorem \ref{thm:isolated_divisors:riemann_roch_criterion}.

\begin{figure}
	\centering
	\includegraphics[width=330pt]{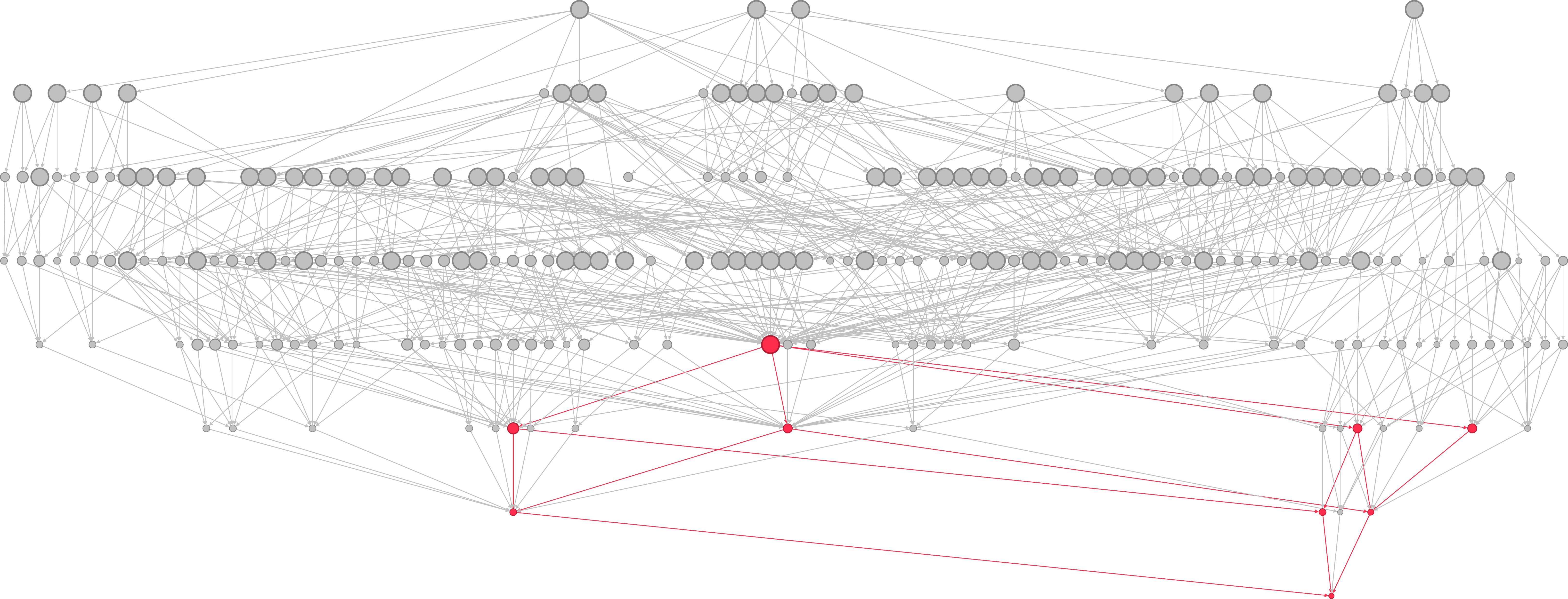}
	\caption{The quotient of the isolation graph for modular curves of level 7 and closed points with $j$-invariant $\frac{3^{3} \cdot 5 \cdot 7^{5}}{2^{7}}$, as in Figure \ref{fig:level_7:classes_isolation_graph}. The subgraph induced by the possibly isolated vertices is highlighted in red, while all other vertices are not isolated by Theorem \ref{thm:isolated_divisors:riemann_roch_criterion}.}
	\label{fig:level_7:classes_isolation_graph_possibly_isolated}
\end{figure}

Iterating through all the vertices in the quotient graph and applying the above, we obtain Figure \ref{fig:level_7:classes_isolation_graph_possibly_isolated}. The vertices for which Theorem \ref{thm:isolated_divisors:riemann_roch_criterion} applies are in gray, while the subgraph induced by the remaining vertices is highlighted in red. We note that, out of the 252 vertices in the quotient graph, 243 can be shown to correspond to non-isolated closed points using the aforementioned theorem. The nine remaining vertices correspond to modular curves $X_{H}$ and closed points $x \in X_{H}$ precisely as described in Table \ref{tbl:level_7:isolated_points_level_7}, thus proving Theorem \ref{thm:level_7:restricting_H_and_x}.

In addition to proving that most of the modular curves of level 7 do not contain an isolated point with $j$-invariant equal to $\frac{3^{3} \cdot 5 \cdot 7^{5}}{2^{7}}$, the isolation graph can also be used to deduce a minimal set of points which whose isolated-ness must be determined. Indeed, the remaining nine vertices described above form a single connected subgraph of the quotient graph. Moreover, this subgraph contains a unique initial vertex. This vertex corresponds to the degree 18 points on the modular curve $X_{\{\pm I\}} = X(7)$ with $j$-invariant $\frac{3^{3} \cdot 5 \cdot 7^{5}}{2^{7}}$. By the construction of the isolation graph, if any of these degree 18 points are isolated, the closed points corresponding to the other eight vertices of the subgraph are also isolated. Therefore, we obtain the following.

\begin{lemma} \label{thm:level_7:unique_initial_closed_point}
	Suppose that one of the degree 18 closed points with $j$-invariant $\frac{3^{3} \cdot 5 \cdot 7^{5}}{2^{7}}$ on the modular curve $X(7)$ is isolated. Then, all closed points described in Table \ref{tbl:level_7:isolated_points_level_7} are isolated.
\end{lemma}

Therefore, the proof of Theorem \ref{thm:level_7:level_7_isolated_points} has been reduced to determining whether a single closed point on a single modular curve is isolated. We tackle the problem of confirming that these degree 18 points are indeed isolated in the next section. In the meantime, we conclude this section with one final remark.

\begin{figure}
	\centering
	\includegraphics[width=330pt]{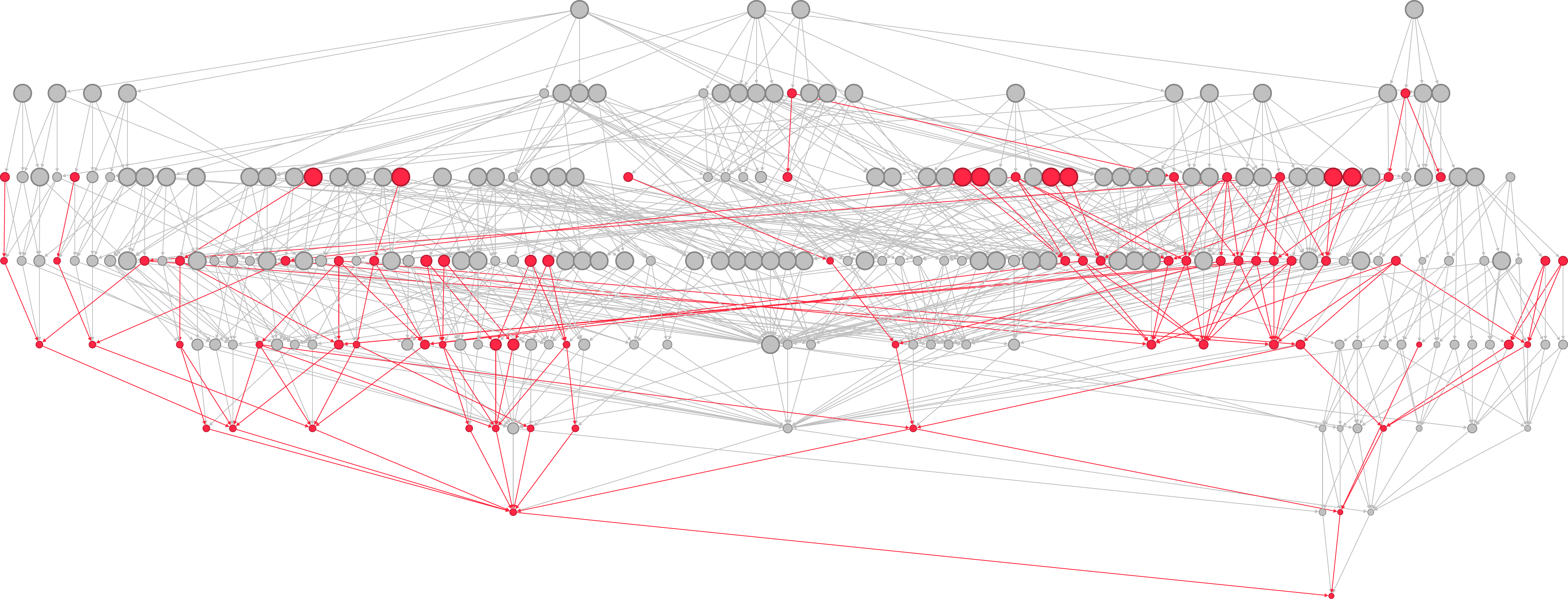}
	\caption{The quotient of the isolation graph for modular curves of level 7 and closed points with $j$-invariant $\frac{3^{3} \cdot 5 \cdot 7^{5}}{2^{7}}$, as in Figure \ref{fig:level_7:classes_isolation_graph}. The subgraph induced by the closed points on geometrically connected modular curves is highlighted in red.}
	\label{fig:level_7:classes_isolation_graph_det_index_1}
\end{figure}

\begin{remark}
Figure \ref{fig:level_7:classes_isolation_graph_det_index_1} again shows the quotient of the isolation graph discussed previously. The subgraph induced by the closed points lying on geometrically connected modular curves is highlighted in red. We note that this subgraph does not contain a unique terminal vertex; rather, there are three such vertices. In particular, for the two new terminal vertices, Corollary \ref{thm:isolated_points_modular_curves:mod_n_galois_image_isolated} cannot be proven if we restrict ourselves to maps between geometrically connected curves. This shows the importance of considering geometrically disconnected curves, even if one is solely interested in understanding isolated points on geometrically connected curves.
\end{remark}

\subsection{Isolated points on \texorpdfstring{$X(7)$}{X(7)}} \label{sec:level_7:points_are_isolated}

Using Theorems \ref{thm:level_7:j_invariant_restriction} and \ref{thm:level_7:restricting_H_and_x} and Lemma \ref{thm:level_7:unique_initial_closed_point}, we have now reduced the proof of Theorem \ref{thm:level_7:level_7_isolated_points} to showing that one of the degree 18 points on the modular curve $X(7)$ with $j$-invariant $\frac{3^{3} \cdot 5 \cdot 7^{5}}{2^{7}}$ is isolated. We devote this section to the proof of this final statement.

\begin{theorem} \label{thm:level_7:x7_is_isolated}
	Let $x \in X(7)$ be a degree 18 closed point with $j(x) = \frac{3^{3} \cdot 5 \cdot 7^{5}}{2^{7}}$. Then $x$ is isolated.
\end{theorem}

The proof of this result is quite ad-hoc but mostly straightforward: we construct the desired degree 18 points on an explicit model of $X(7)$, which we then use to check that the points are both \Pone-isolated and AV-isolated. One small difficulty stems from the fact that the modular curve $X(7)$ is geometrically disconnected. To circumvent this, we use Theorem \ref{thm:isolated_divisors:stein_factorization_isolated}, which allows us to work over the Stein factorization of $X(7)$. The latter is a geometrically integral curve, but is defined over the larger number field $\Q(\zeta_{7})$ instead.

\begin{proof}[Proof of Theorem \ref{thm:level_7:x7_is_isolated}]
	Since $\det(\pm I) = 1$, the Stein factorization of the modular curve $X(7)$ is
	\[
		X(7) \to \Spec \Q(\zeta_{7}) \to \Spec \Q.
	\]
	By Theorem \ref{thm:isolated_divisors:stein_factorization_isolated}, it suffices to show that the degree 3 closed point on $X(7) / \Q(\zeta_{7})$ corresponding to the degree 18 closed point $x \in X(7) / \Q$ is isolated. Let $X_{\text{arith}}(7)$ be the modular curve $X_{K}$, where
	\[
		K = \{\pm \left(\begin{smallmatrix}1 & 0 \\ 0 & a\end{smallmatrix}\right) : a \in \Z*[7]\} \leq \GL{\Z[7]}.
	\]
	Since $K \cap \SL{\Z[7]} = \pm I$, we obtain, by Theorem \ref{thm:modular_curves:sl_intersection_defines_geometric_components}, a Cartesian square
	\[\begin{tikzcd}
		X(7) \arrow[d, "f"] \arrow[r] & \Spec \Q(\zeta_{7}) \arrow[d] \\
		X_{\text{arith}}(7) \arrow[r] & \Spec \Q,
	\end{tikzcd}\]
	where $f$ is the inclusion morphism $X(7) \to X_{\text{arith}}(7)$. The morphism $X_{\text{arith}}(7) \to \Spec \Q$ factors through the $j$-map $j : X_{\text{arith}}(7) \to X(1)$. Noting that $X(1) \cong \mathbb{P}^{1}_{\Q}$, we can therefore rewrite the above commutative diagram as
	\[\begin{tikzcd}
		X(7) \arrow[d, "f"] \arrow[r] & \mathbb{P}^{1}_{\Q(\zeta_{7})} \arrow[d] \arrow[r] & \Spec \Q(\zeta_{7}) \arrow[d] \\
		X_{\text{arith}}(7) \arrow[r, "j"] & \mathbb{P}^{1}_{\Q} \arrow[r] & \Spec \Q,
	\end{tikzcd}\]
	where both squares are Cartesian. Since the composition $j \circ f$ is precisely the $j$-map $X(7) \to X(1)$, it follows that the closed points of $X(7)$ with $j$-invariant $\frac{3^{3} \cdot 5 \cdot 7^{5}}{2^{7}}$ are the closed points of $X(7)$ lying above $\frac{3^{3} \cdot 5 \cdot 7^{5}}{2^{7}} \in \mathbb{P}^{1}_{\Q(\zeta_{7})}(\Q(\zeta_{7}))$.
	
	We know by \cite[Equation 3-1]{halberstadt2003} that a model for the modular curve $X_{\text{arith}}(7)$ over $\Q$ is given by the Klein quartic $X^3 Y + Y^3 Z + Z^3 X = 0$. Moreover, Halberstadt and Kraus give an explicit equation for the degree-168 $j$-map $j : X_{\text{arith}}(7) \to X(1)$; see \cite[Equation 3-10]{halberstadt2003}. By base change to $\Q(\zeta_{7})$, this therefore gives an explicit model for $X(7) / \Q(\zeta_{7})$, as well as the map $X(7) \to \mathbb{P}^{1}_{\Q(\zeta_{7})}$. Using \texttt{Magma}, we can use this to explicitly compute the degree 3 points on $X(7) / \Q(\zeta_{7})$ lying above $\frac{3^{3} \cdot 5 \cdot 7^{5}}{2^{7}} \in \mathbb{P}^{1}_{\Q(\zeta_{7})}(\Q(\zeta_{7}))$. Moreover, we can compute the dimension of the Riemann-Roch spaces of these points. In all cases, we find that the dimension of the Riemann-Roch space is always 1. Thus, all of the desired degree 3 points on $X(7) / \Q(\zeta_{7})$ are \Pone-isolated.
	
	Finally, we compute the Mordell-Weil rank of the Jacobian of $X(7) / \Q(\zeta_{7})$. We know, by \cite[Section 8]{halberstadt2003}, that the Jacobian of $X_{\text{arith}}(7)$ is isogenous over $\Q(\zeta_7)$ to $A^3$, where $A$ is the elliptic curve
	\[
		A : y^2 + xy = x^3 - x^2 - 2x - 1.
	\]
	Using \texttt{Magma}, we compute that the Mordell-Weil rank of $A_{\Q(\zeta_7)}$ is 0. Therefore, the Mordell-Weil rank of the Jacobian of $X(7) / \Q(\zeta_{7})$ is also 0. It follows that all closed points of $X(7) / \Q(\zeta_{7})$ are AV-isolated, thus completing the proof.
\end{proof}
	\section{Isolated points on the modular curves \texorpdfstring{$X_{0}(n)$}{X\_0(n)}} \label{sec:isolated_points_x0}

As a second application of our method, we now aim to give a conjectural classification of the non-cuspidal, non-CM isolated points with rational $j$-invariant on the modular curves $X_{0}(n)$. This problem is closely related to the problem of determining the non-cuspidal, non-CM isolated points with rational $j$-invariant on the modular curves $X_{1}(n)$, which has attracted much attention in recent literature. While the methods for tackling the first problem should be applicable to this second problem, the computations necessary for the first problem can already be found in the literature, while those for the second have not been carried out previously. Therefore, we limit ourselves to studying $X_{0}(n)$ at present, and leave $X_{1}(n)$ for future work.

As was done for the modular curves of level 7, we employ the method outlined in Section \ref{sec:introduction:finding_isolated_points}. However, the first step differs significantly from the classification of the isolated points on modular curves of level 7. In the latter case, we were able to use Corollary \ref{thm:isolated_points_modular_curves:mod_n_galois_image_isolated} in order to restrict the possible $j$-invariants of the isolated points, as each modular curve had the same level. However, as the level of the modular curves $X_{0}(n)$ is unbounded, the same strategy cannot be applied here. Instead, we rely on the notion of agreeable closures, which was first defined in \cite{zywina2024}.

\begin{definition}
	Let $G$ be an open subgroup of $\GL{\Zhat}$, and let $n$ be the level of the commutator subgroup $[G, G] \leq \SL{\Zhat}$. Recall that $\GL{\Zhat} = \prod_{p} \GL{\Z_{p}}$. We define the \textbf{agreeable closure of $G$} to be the subgroup
	\[
		\mathcal{G} = \Zhat* \cdot G \cdot \left( \prod_{p \nmid n} \GL{\Z_{p}} \right),
	\]
	where $\Zhat*$ denotes the subgroup of scalar matrices of $\GL{\Zhat}$. The agreeable closure $\mathcal{G}$ is an open subgroup of $\GL{\Zhat}$ whose level is only divisible by primes dividing $n$.
\end{definition}

Our aim is now to prove a generalization of the single-sink theorem, Theorem \ref{thm:isolated_points_modular_curves:galois_image_isolated}, to agreeable closures. To do so, we first prove a stronger version of Corollary \ref{thm:isolated_points_modular_curves:mod_n_H_isolated} in the case of $X_{0}(n)$, which will also be helpful in the second step of the method.

\begin{lemma} \label{thm:isolated_points_x0:x0_sl_level_isolated}
	Let $n \geq 1$, and $x \in X_{0}(n)$ be a non-cuspidal, non-CM isolated closed point with minimal representative $(E, [\alpha]_{B_{0}(n)})$. Let $m$ be the level of $[G_{E, \alpha}, G_{E, \alpha}] \leq \SL{\Zhat}$. Then the closed point $y \in X_{0}((n, m))$ corresponding to the $G_{\Q}$-orbit of the geometric point $[(E, [\alpha]_{B_{0}((n, m))})] \in X_{0}((n, m))(\Qbar)$ is isolated.
\end{lemma}

\begin{proof}
	By definition, $m$ is the smallest integer such that the commutator subgroup $[G_{E, \alpha}, G_{E, \alpha}]$ contains the kernel of the map $\pi'_{m} : \SL{\Zhat} \to \SL{\Z[m]}$. Denote by $N \leq \SL{\Zhat}$ the kernel of this map, and note that $N$ is a normal subgroup of $\GL{\Zhat}$ contained in $G_{E, \alpha}$. Therefore, by Theorem \ref{thm:isolated_points_modular_curves:normal_H_isolated}, the closed point $z \in X_{N B_{0}(n)}$ corresponding to the geometric point $[(E, [\alpha]_{N B_{0}(n)})] \in X_{N B_{0}(n)}(\Qbar)$ is isolated. It remains to show that $N B_{0}(n) = B_{0}((n, m))$.
	
	The subgroup $B_{0}(n) \leq \GL{\Zhat}$ contains the subgroup
	\[
		\Sigma = \left\{\left(\begin{smallmatrix}
			1 & 0 \\ 0 & 1 + am
		\end{smallmatrix}\right) \in \GL{\Zhat} : a \in \Zhat, 1 + am \in \Zhat*\right\}.
	\]
	Since $\Sigma \subset \ker \pi_{m}$, it follows that $N \Sigma \subset \ker \pi_{m}$. Let $A \in \ker \pi_{m}$. Since $A \in \GL{\Zhat}$, we have that $\det(A) \in \Zhat*$. Moreover, by construction, $\det(A)$ is congruent to 1 modulo $m$, and so is $\det(A)^{-1}$ as well. Thus, $\left(\begin{smallmatrix}
		1 & 0 \\ 0 & \det(A)^{-1}
	\end{smallmatrix}\right) \in \Sigma$. Since
	\[
		A \left(\begin{smallmatrix}
			1 & 0 \\ 0 & \det(A)^{-1}
		\end{smallmatrix}\right) \in N,
	\]
	we obtain that $\ker \pi_{m} \subset N \Sigma$. Therefore, we have
	\[
		N B_{0}(n) = N \Sigma B_{0}(n) = (\ker \pi_{m}) B_{0}(n).
	\]
	Since $B_{0}(n)$ has level $n$, it contains the kernel $\ker \pi_{n}$. Hence, by Lemma \ref{thm:group_theory:kernel_product_gl2_zhat}, we obtain that
	\[
		(\ker \pi_{m}) B_{0}(n) = (\ker \pi_{m}) (\ker \pi_{n}) B_{0}(n) = (\ker \pi_{(n, m)}) B_{0}(n) = B_{0}((n, m)).
	\]
	Thus, we have $N B_{0}(n) = B_{0}((n, m))$, and the result follows.
\end{proof}

Using this result, we may now prove the aforementioned generalization of the single-sink theorem.

\begin{theorem} \label{thm:isolated_points_x0:agreeable_closure_isolated}
	Let $n \geq 1$, and $x \in X_{0}(n)$ be a non-cuspidal, non-CM isolated closed point with minimal representative $(E, [\alpha]_{B_{0}(n)})$. Let $\mathcal{G} \leq \GL{\Zhat}$ be the agreeable closure of $G_{E, \alpha}$. Then the closed point $y \in X_{\mathcal{G}}$ corresponding to the geometric point $[(E, [\alpha]_{\mathcal{G}})] \in X_{\mathcal{G}}(\Qbar)$ is isolated.
\end{theorem}

\begin{proof}
	Let $m$ be the level of $[G_{E, \alpha}, G_{E, \alpha}]$, and $k$ the level of $G_{E, \alpha}$. In particular, note that $m$ divides $k$. By Lemma \ref{thm:isolated_points_x0:x0_sl_level_isolated}, the closed point $z \in X_{0}((n, m))$ corresponding to the $G_{\Q}$-orbit of the $\Qbar$-point $[(E, [\alpha]_{B_{0}((n, m))})] \in X_{0}((n, m))(\Qbar)$ is isolated. The subgroup $B_{0}((n, m))$ has level $(n, m)$, and so contains both the product $\prod_{p \nmid m} \GL{\Z_{p}}$ and the kernel $\ker \pi_{k}$. Moreover, $B_{0}((n, m))$ contains the subgroup of scalar matrices, which we denote $\Zhat*$ by abuse of notation. All three of these subgroups are normal subgroups of $\GL{\Zhat}$, and hence so is their product $(\ker \pi_{k})\cdot \Zhat* \cdot \prod_{p \nmid m} \GL{\Z_{p}}$. This product is also an open subgroup of $\GL{\Zhat}$, since the kernel $\ker \pi_{k}$ is open by definition. Note that we have
	\[
		G_{E, \alpha} \Big((\ker \pi_{k}) \cdot \Zhat* \cdot \prod_{p \nmid m} \GL{\Z_{p}}\Big) = G_{E, \alpha} \cdot \Zhat* \cdot \prod_{p \nmid m} \GL{\Z_{p}} = \mathcal{G}.
	\]
	Therefore, since $z \in X_{0}((n, m))$ is isolated, so is $y \in X_{\mathcal{G}}$, by Theorem \ref{thm:isolated_points_modular_curves:normal_galois_image_isolated}.
\end{proof}

In order to carry through the first step of our method, in Section \ref{sec:level_7:determining_j}, we required a classification of the extended mod-7 Galois images of elliptic curves over $\Q$. In the same way, we now require a classification of the agreeable closures $\mathcal{G}$ of the extended adelic Galois images of elliptic curves over $\Q$. A conjectural classification of such agreeable closures $\mathcal{G}$ has been computed by Zywina in \cite{zywina2024}. We summarize these computations as follows.

\begin{conjecture}[{\cite[Theorem 1.9, Section 14.3]{zywina2024}}] \label{thm:isolated_points_x0:exceptional_j_invariants}
	Let $E$ be a non-CM elliptic curve over $\Q$ and $\alpha$ a profinite level structure on $E$. Let $\mathcal{G}$ be the agreeable closure of the extended adelic Galois image $G_{E, \alpha} \leq \GL{\Zhat}$. Then, one of the following hold:
	\begin{itemize}
		\item $X_{\mathcal{G}}$ is a geometrically connected curve of genus 0
		\item $X_{\mathcal{G}}$ is a geometrically connected curve of genus 1, and $X_{\mathcal{G}}(\Q)$ is infinite
		\item $X_{\mathcal{G}}(\Q)$ is finite, and the group $\mathcal{G}$ and the $j$-invariant of $E$ are given in Table \ref{tbl:isolated_points_x0:exceptional_j_invariants}
	\end{itemize}
\end{conjecture}

\bgroup
\def\arraystretch{1.2}
\begin{longtable}{lrrrr}
	\caption{The known non-CM elliptic curves $E/\Q$ such that the agreeable closure $\mathcal{G} \leq \GL{\Zhat}$ of $G_{E, \alpha}$ gives rise to a modular curve $X_{\mathcal{G}}$ with finitely many rational points. The first column lists the $j$-invariant of $E$. The level and index of $\mathcal{G}$, as well as the genus of the modular curve $X_{\mathcal{G}}$, are indicated in columns 2-4. The level of the commutator subgroup $[G_{E, \alpha}, G_{E, \alpha}]$ is also listed.} \\
	\label{tbl:isolated_points_x0:exceptional_j_invariants}
	$j$-invariant & Level of $\mathcal{G}$ & $[\GL{\Zhat} : \mathcal{G}]$ & $g(X_{\mathcal{G}})$ & Level of $[G_{E, \alpha}, G_{E, \alpha}]$\\ \toprule
	\endfirsthead
	$j$-invariant & Level of $\mathcal{G}$ & $[\GL{\Zhat} : \mathcal{G}]$ & $g(X_{\mathcal{G}})$ & Level of $[G_{E, \alpha}, G_{E, \alpha}]$\\ \toprule
	\endhead
	$-13824$ & $12$ & $12$ & $1$ & $12$ \\
	$-82944$ & $12$ & $24$ & $1$ & $12$ \\
	$-138240$ & $12$ & $24$ & $1$ & $12$ \\
	$-\frac{35937}{4}$ & $12$ & $32$ & $1$ & $12$ \\
	$\frac{109503}{64}$ & $12$ & $32$ & $1$ & $12$ \\
	$\frac{3375}{64}$ & $12$ & $48$ & $1$ & $12$ \\
	$\frac{130 \cdot 442^3}{3^{13}}$ & $13$ & $91$ & $3$ & $26$ \\
	$-\frac{143 \cdot 1040^3}{3^{13}}$ & $13$ & $91$ & $3$ & $26$ \\
	$\frac{12077 \cdot 1957713745728^3}{305^{13}}$ & $13$ & $91$ & $3$ & $26$ \\
	$550731776$ & $14$ & $42$ & $1$ & $14$ \\
	$-\frac{61084010175^3}{61^{15}}$ & $15$ & $60$ & $2$ & $30$ \\
	$2048$ & $16$ & $24$ & $1$ & $16$ \\
	$78608$ & $16$ & $24$ & $1$ & $16$ \\
	$\frac{16974593}{256}$ & $16$ & $24$ & $1$ & $16$ \\
	$\frac{4097^3}{16}$ & $16$ & $24$ & $1$ & $16$ \\
	$-\frac{3 \cdot 18249920^3}{17^{16}}$ & $16$ & $64$ & $2$ & $16$ \\
	$-\frac{7 \cdot 1723187806080^3}{79^{16}}$ & $16$ & $64$ & $2$ & $16$ \\
	$-\frac{297756989}{2}$ & $17$ & $18$ & $1$ & $34$ \\
	$-\frac{882216989}{131072}$ & $17$ & $18$ & $1$ & $34$ \\
	$110592$ & $18$ & $18$ & $1$ & $18$ \\
	$-44789760$ & $18$ & $54$ & $2$ & $18$ \\
	$\frac{1026895}{1024}$ & $20$ & $24$ & $1$ & $20$ \\
	$-\frac{1723025}{4}$ & $20$ & $24$ & $1$ & $20$ \\
	$-36$ & $20$ & $40$ & $2$ & $20$ \\
	$-\frac{64278657}{1024}$ & $20$ & $40$ & $2$ & $20$ \\
	$-\frac{5 \cdot 2805^3}{2^{10}}$ & $20$ & $60$ & $3$ & $20$ \\
	$\frac{30081024}{3125}$ & $20$ & $120$ & $6$ & $20$ \\
	$\frac{4543847424}{3125}$ & $20$ & $120$ & $6$ & $20$ \\
	$\frac{1906624}{729}$ & $24$ & $72$ & $2$ & $12$ \\
	$\frac{82881856}{27}$ & $24$ & $72$ & $2$ & $24$ \\
	$-\frac{50836^3}{75^6}$ & $24$ & $72$ & $2$ & $24$ \\
	$\frac{247084^3}{45^6}$ & $24$ & $72$ & $2$ & $24$ \\
	$\frac{697317169440^3}{23^{24}}$ & $24$ & $96$ & $3$ & $24$ \\
	$\frac{5 \cdot 34800^3}{7^5}$ & $25$ & $75$ & $2$ & $50$ \\
	$\frac{351}{4}$ & $28$ & $64$ & $3$ & $28$ \\
	$-\frac{13 \cdot 1437^3}{2^{14}}$ & $28$ & $64$ & $3$ & $28$ \\
	$4096$ & $30$ & $36$ & $1$ & $30$ \\
	$3376^3$ & $30$ & $36$ & $1$ & $30$ \\
	$-1273201875$ & $30$ & $120$ & $7$ & $30$ \\
	$\frac{919425^3}{496^4}$ & $32$ & $96$ & $3$ & $32$ \\
	$-\frac{857985^3}{62^8}$ & $32$ & $96$ & $3$ & $16$ \\
	$-12 \cdot 19755^3$ & $36$ & $108$ & $6$ & $36$ \\
	$-9317$ & $37$ & $38$ & $2$ & $74$ \\
	$-7 \cdot 285371^3$ & $37$ & $38$ & $2$ & $74$ \\
	$-\frac{1315 \cdot 1942940^3}{7^{20}}$ & $40$ & $40$ & $2$ & $40$ \\
	$-5000$ & $40$ & $60$ & $3$ & $10$ \\
	$-121$ & $44$ & $24$ & $2$ & $44$ \\
	$-24729001$ & $44$ & $24$ & $2$ & $44$ \\
	$-\frac{13 \cdot 1542120^3}{43^{11}}$ & $44$ & $110$ & $4$ & $22$ \\
	$4913$ & $48$ & $72$ & $3$ & $48$ \\
	$16974593$ & $48$ & $72$ & $3$ & $48$ \\
	$\frac{52113^3}{112^6}$ & $48$ & $72$ & $3$ & $48$ \\
	$\frac{930927^3}{14^{12}}$ & $48$ & $72$ & $3$ & $24$ \\
	$\frac{2268945}{128}$ & $56$ & $112$ & $5$ & $14$ \\
	$432$ & $60$ & $40$ & $2$ & $60$ \\
	$-316368$ & $60$ & $40$ & $2$ & $60$ \\
	$\frac{59971704^3}{11^{15}}$ & $60$ & $90$ & $4$ & $30$ \\
	$-216$ & $72$ & $36$ & $3$ & $18$ \\
	$4374$ & $72$ & $54$ & $2$ & $18$ \\
	$419904$ & $72$ & $54$ & $2$ & $18$ \\
	$9 \cdot 1206^3$ & $72$ & $54$ & $2$ & $18$ \\
	$51 \cdot 7884^3$ & $72$ & $54$ & $2$ & $18$ \\
	$-21024576$ & $72$ & $144$ & $11$ & $36$ \\
	$189812888^3$ & $84$ & $126$ & $4$ & $42$ \\
	$90 \cdot 1150^3$ & $100$ & $100$ & $4$ & $50$ \\
	$-\frac{25}{2}$ & $120$ & $48$ & $3$ & $30$ \\
	$-\frac{121945}{32}$ & $120$ & $48$ & $3$ & $30$ \\
	$\frac{46969655}{32768}$ & $120$ & $48$ & $3$ & $30$ \\
	$-\frac{349938025}{8}$ & $120$ & $48$ & $3$ & $30$ \\
	$\frac{1720060^3}{11^{15}}$ & $120$ & $60$ & $2$ & $30$ \\
	$\frac{1331}{8}$ & $120$ & $72$ & $3$ & $30$ \\
	$-\frac{1680914269}{32768}$ & $120$ & $72$ & $3$ & $30$ \\
	$\frac{1839651^3}{2^{15}}$ & $120$ & $90$ & $4$ & $30$ \\
	$-\frac{1852491^3}{2^{15}}$ & $120$ & $90$ & $4$ & $30$ \\
	$\frac{3375}{2}$ & $168$ & $64$ & $3$ & $42$ \\
	$-\frac{140625}{8}$ & $168$ & $64$ & $3$ & $42$ \\
	$-\frac{9 \cdot 505^3}{2^{21}}$ & $168$ & $64$ & $3$ & $42$ \\
	$-\frac{5745^3}{2^7}$ & $168$ & $64$ & $3$ & $42$ \\
	$40561972^3$ & $168$ & $126$ & $4$ & $42$ \\
	$64$ & $360$ & $108$ & $7$ & $90$ \\
	$-2876^3$ & $360$ & $108$ & $7$ & $90$ \\ \bottomrule
\end{longtable}
\egroup

In tandem with Theorem \ref{thm:isolated_points_x0:agreeable_closure_isolated}, these computations strongly restrict the possible $j$-invariants of non-cuspidal, non-CM isolated points with rational $j$-invariant on the modular curves $X_{0}(n)$.

\begin{theorem} \label{thm:isolated_points_x0:j_invariant_is_exceptional}
	Assume that Conjecture \ref{thm:isolated_points_x0:exceptional_j_invariants} holds. Let $n \geq 1$, and $x \in X_{0}(n)$ be a non-cuspidal, non-CM isolated closed point with minimal representative $(E, [\alpha]_{B_{0}(n)})$, with $j(E) \in \Q$. Then the $j$-invariant $j(E)$ is given in Table \ref{tbl:isolated_points_x0:exceptional_j_invariants}.
\end{theorem}

\begin{proof}
	Let $\mathcal{G} \leq \GL{\Zhat}$ be the agreeable closure of the extended adelic Galois image $G_{E, \alpha}$. By Theorem \ref{thm:isolated_points_x0:agreeable_closure_isolated}, the closed point $y \in X_{\mathcal{G}}$ corresponding to the $G_{\Q}$-orbit of the $\Qbar$-point $[(E, [\alpha]_{\mathcal{G}})] \in X_{\mathcal{G}}(\Qbar)$ is isolated. Note that the agreeable closure $\mathcal{G}$ contains the extended adelic Galois image $G_{E, \alpha}$, and that $j(E) \in \Q$. Therefore, by Theorem \ref{thm:modular_curves:profinite_point_degree}, the closed point $y$ has degree 1.
	
	Since the point $y \in X_{\mathcal{G}}$ is isolated, by Theorem \ref{thm:isolated_divisors:riemann_roch_criterion}, the geometrically connected modular curve $X_{\mathcal{G}}$ cannot have genus 0. Moreover, if the modular curve $X_{\mathcal{G}}$ has genus 1, the image $\W_{X_{\mathcal{G}}/k}^{y}$ is a translate of $\PPic_{X_{\mathcal{G}}/k}^{0}$, itself isomorphic to $X_{\mathcal{G}}$. Since $y$ is isolated, it follows that $X_{\mathcal{G}}$ must have rank 0. Therefore, by Conjecture \ref{thm:isolated_points_x0:exceptional_j_invariants}, the $j$-invariant of $E$ is given in Table \ref{tbl:isolated_points_x0:exceptional_j_invariants}.
\end{proof}

We have now shown that the $j$-invariant of every non-cuspidal, non-CM isolated closed point with rational $j$-invariant on a modular curve $X_{0}(n)$ must lie in the explicit finite set given in Table \ref{tbl:isolated_points_x0:exceptional_j_invariants}. Following the strategy outlined in Section \ref{sec:introduction:finding_isolated_points}, for each such $j$-invariant $j$ we now determine a finite set $S_{j}$ of modular curves $X_{0}(n)$ such that there exists an isolated point with $j$-invariant $j$ on any modular curve $X_{0}(n)$ only if there exists such a point on a modular curve in the set $S_{j}$. Such a finite set can be constructed directly from Theorem \ref{thm:isolated_points_x0:x0_sl_level_isolated}. However, we may use similar techniques to those employed in Section \ref{sec:level_7:restricting_H_and_x} to restrict this set further, and show that most $j$-invariants in Table \ref{tbl:isolated_points_x0:exceptional_j_invariants} cannot occur.

\begin{theorem}
	Assume that Conjecture \ref{thm:isolated_points_x0:exceptional_j_invariants} holds. Let $n \geq 1$, and $x \in X_{0}(n)$ be a non-cuspidal, non-CM isolated closed point with minimal representative $(E, [\alpha]_{B_{0}(n)})$, with $j(E) \in \Q$. Then the $j$-invariant $j(E)$ is given in Table \ref{tbl:isolated_points_x0:isolated_points_x0}.
\end{theorem}

\begin{proof}
	By Theorem \ref{thm:isolated_points_x0:j_invariant_is_exceptional}, the $j$-invariant $j(E)$ must be one of the 81 $j$-invariants listed in Table \ref{tbl:isolated_points_x0:exceptional_j_invariants}. Therefore, suppose that the $j$-invariant of $E$ is given in Table \ref{tbl:isolated_points_x0:exceptional_j_invariants}, but not given in Table \ref{tbl:isolated_points_x0:isolated_points_x0}. The image $G_{E, \alpha}$ of the Galois representation of the elliptic curve $E/\Q$ is computed in \cite{zywina2024}, for some fixed profinite level structure $\alpha$. In particular, we can compute the level $m$ of the commutator subgroup $[G_{E, \alpha}, G_{E, \alpha}]$, which is also recorded in Table \ref{tbl:isolated_points_x0:exceptional_j_invariants}.
	
	By Lemma \ref{thm:isolated_points_x0:x0_sl_level_isolated}, since the point $x \in X_{0}(n)$ is isolated, so is the point $y \in X_{0}((n, m))$ corresponding to the $G_{\Q}$-orbit of the geometric point $[(E, [\alpha]_{B_{0}((n, m))})] \in X_{0}((n, m))(\Qbar)$. Therefore, it suffices to show that the modular curves $X_{0}(n)$ do not contain an isolated closed point with $j$-invariant equal to $j(E)$, for all divisors $n$ of $m$.
	
	For each such divisor $n$, the closed points of $X_{0}(n)$ with $j$-invariant equal to $j(E)$ are in bijection with the double cosets $B_{0}(n) g G_{E, \alpha_{n}} \subset \GL{\Z[n]}$, by Theorem \ref{thm:modular_curves:closed_points_as_double_cosets}. As the extended adelic Galois image $G_{E, \alpha}$ of $E$ is known, we may compute the degree of each of these closed points using the aforementioned theorem. For each such closed point, this degree exceeds the genus of the modular curve $X_{0}(n)$, and so the closed point cannot be isolated, by Theorem \ref{thm:isolated_divisors:riemann_roch_criterion}.
	
	These computations are implemented in the file \texttt{isolated\char`_points\char`_x0.m}. As an example, let $E$ be the elliptic curve \href{https://www.lmfdb.org/EllipticCurve/Q/61347.bb1}{61347.bb1} with $j$-invariant $-\frac{160855552000}{1594323}$. The extended adelic Galois image $G_{E, \alpha}$ has level 78, and is given by the preimage of the subgroup
	\begin{align*}
		G_{E, \alpha_{78}} = & \left\langle \left(\begin{smallmatrix} 27 & 14 \\ 64 & 39 \end{smallmatrix}\right), \left(\begin{smallmatrix} 27 & 26 \\ 52 & 27 \end{smallmatrix}\right), \left(\begin{smallmatrix} 31 & 52 \\ 9 & 47 \end{smallmatrix}\right), \left(\begin{smallmatrix} 73 & 38 \\ 44 & 71 \end{smallmatrix}\right), \right. \\
		& \quad \left(\begin{smallmatrix} 53 & 0 \\ 0 & 53 \end{smallmatrix}\right), \left. \left(\begin{smallmatrix} 1 & 0 \\ 52 & 1 \end{smallmatrix}\right), \left(\begin{smallmatrix} 14 & 39 \\ 13 & 53 \end{smallmatrix}\right), \left(\begin{smallmatrix} 65 & 72 \\ 48 & 65 \end{smallmatrix}\right) \right\rangle \leq \GL{\Z[78]}.
	\end{align*}
	The commutator subgroup $[G_{E, \alpha}, G_{E, \alpha}]$ has level 26, and so it suffices to show that there are no isolated points with the given $j$-invariant on the modular curves $X_{0}(n)$, where $n \in \{1, 2, 13, 26\}$. The modular curves $X_{0}(1)$, $X_{0}(2)$ and $X_{0}(13)$ all have genus 0, and so there are no isolated points on these modular curves by Theorem \ref{thm:isolated_divisors:riemann_roch_criterion}. The modular curve $X_{0}(26)$ however, has genus 2.
	
	The extended mod-26 Galois image $G_{E, \alpha_{26}}$ of $E$ is given by the reduction modulo 26 of $G_{E, \alpha_{78}}$. For each conjugate $g G_{E, \alpha_{26}} g^{-1}$ of $G_{E, \alpha_{26}}$, we compute the index $[g G_{E, \alpha_{26}} g^{-1} : g G_{E, \alpha_{26}} g^{-1} \cap \pm B_{0}(26)]$, and find that it is either 18 or 24. Therefore, by Theorem \ref{thm:modular_curves:closed_points_as_double_cosets}, the closed points of $X_{0}(26)$ with $j$-invariant $-\frac{160855552000}{1594323}$ must have degree 18 or 24. In both cases, this degree exceeds the genus of $X_{0}(26)$, and so, by Theorem \ref{thm:isolated_divisors:riemann_roch_criterion}, these closed points cannot be isolated.
\end{proof}

\begin{table}
	\centering
	\caption{The $j$-invariants of the non-cuspidal, non-CM isolated closed points with rational $j$-invariant on the modular curves $X_{0}(n)$. For each $j$-invariant, the smallest modular curve $X_{0}(n)$ for which there exists an isolated closed point $x \in X_{0}(n)$ with the given $j$-invariant is listed. The genus of the modular curve $X_{0}(n)$ and the degree of the closed point $x$ are also given.}
	\label{tbl:isolated_points_x0:isolated_points_x0}
	\bgroup
	\def\arraystretch{1.2}
	\begin{tabular}{llrr}
		$j(E)$ & $X_{0}(n)$ & $g(X_{0}(n))$ & $\deg(x)$ \\ \toprule
		$-121$ & \multirow{2}*{$X_{0}(11)$} & \multirow{2}*{1} & 1 \\
		$-24729001$ & & & 1 \\ \midrule
		$-\frac{25}{2}$ & \multirow{4}*{$X_{0}(15)$} & \multirow{4}*{1} & 1 \\
		$-\frac{121945}{32}$ & & & 1 \\
		$\frac{46969655}{32768}$ & & & 1 \\
		$-\frac{349938025}{8}$ & & & 1 \\ \midrule
		$-\frac{297756989}{2}$ & \multirow{2}*{$X_{0}(17)$} & \multirow{2}*{1} & 1 \\
		$-\frac{882216989}{131072}$ & & & 1 \\ \midrule
		$\frac{3375}{2}$ & \multirow{4}*{$X_{0}(21)$} & \multirow{4}*{1} & 1 \\
		$-\frac{140625}{8}$ & & & 1 \\
		$-\frac{1159088625}{2097152}$ & & & 1 \\
		$-\frac{189613868625}{128}$ & & & 1 \\ \midrule
		$-9317$ & \multirow{2}*{$X_{0}(37)$} & \multirow{2}*{2} & 1 \\
		$-162677523113838677$ & & & 1 \\ \bottomrule
	\end{tabular}
	\egroup
\end{table}

It remains to verify the rest of Table \ref{tbl:isolated_points_x0:isolated_points_x0}, namely that for each $j$-invariant listed in the table, there exists an isolated closed point on the corresponding modular curve $X_{0}(n)$. As in Section \ref{sec:level_7:points_are_isolated}, this is a finite computation, and each case can be handled separately using ad-hoc methods.

\begin{theorem}
	Let $j(E)$ and $X_{0}(n)$ be as given in Table \ref{tbl:isolated_points_x0:isolated_points_x0}. Then, there exists an isolated closed point $x \in X_{0}(n)$ with $j$-invariant equal to $j(E)$. Moreover, $n$ is the smallest integer such that $X_{0}(n)$ contains an isolated closed point with the given $j$-invariant.
\end{theorem}

\begin{proof}
	Fix a $j$-invariant from Table \ref{tbl:isolated_points_x0:isolated_points_x0}, and let $n$ be the smallest integer such that $X_{0}(n)$ contains an isolated closed point with the given $j$-invariant. By Lemma \ref{thm:isolated_points_x0:x0_sl_level_isolated}, $n$ must divide the level of $[G_{E, \alpha}, G_{E, \alpha}]$, for any elliptic curve $E/\Q$ with the given $j$-invariant. Note that this level is given in Table \ref{tbl:isolated_points_x0:exceptional_j_invariants}.
	
	The integer $n$ given in Table \ref{tbl:isolated_points_x0:isolated_points_x0} is the smallest integer dividing the level of $[G_{E, \alpha}, G_{E, \alpha}]$ such that the modular curve $X_{0}(n)$ has genus greater than 0. Therefore, the second part of the result follows from the first part, by Theorem \ref{thm:isolated_divisors:riemann_roch_criterion}.
	
	Consider now the modular curve $X_{0}(n)$ given in Table \ref{tbl:isolated_points_x0:isolated_points_x0}. The extended adelic Galois image $G_{E, \alpha}$ is known by \cite{zywina2024}, for some elliptic curve $E/\Q$ with the given $j$-invariant and some profinite level structure $\alpha$ on $E$. Using this, we compute that the extended mod-$n$ Galois image $G_{E, \alpha_{n}}$ of $E$ is conjugate to a subgroup of $B_{0}(n) \leq \GL{\Z[n]}$. Therefore, by Theorem \ref{thm:modular_curves:closed_points_as_double_cosets}, there exists a rational point $x \in X_{0}(n)$ with the given $j$-invariant.
	
	If the curve $X_{0}(n)$ has genus 1, the explicit models of $X_{0}(n)$ in the LMFDB \cite{lmfdb} show that $X_{0}(n)$ has rank 0, and so the set $X_{0}(n)(\Q)$ is finite. If $n=37$ however, then the curve $X_{0}(n)$ has genus 2, and so, by Faltings's theorem, the set $X_{0}(n)(\Q)$ is also finite. Therefore, by Theorem \ref{thm:isolated_divisors:non_isolated_infinite}, any rational point on the curve $X_{0}(n)$ is isolated. In particular, there exists an isolated closed point $x \in X_{0}(n)$ with the given $j$-invariant.
\end{proof}

\begin{remark} \label{rmk:isolated_points_x0:higher_levels}
	Table \ref{tbl:isolated_points_x0:isolated_points_x0} solely lists the minimal integer $n$ such that the modular curve $X_{0}(n)$ contains an isolated closed point with the given $j$-invariant. The existence of other isolated closed points with the same $j$-invariant on the modular curves $X_{0}(m)$, where $m$ is a multiple of $n$, is still an open problem. In fact, it is unknown whether the set of modular curves $X_{0}(m)$ containing such an isolated point is finite or infinite.
\end{remark}
	
	\appendix
	\section{Divisor and Picard schemes of geometrically reducible varieties} \label{sec:divisor_picard_schemes}

The study of the divisor and Picard schemes of a geometrically integral variety is well-established, however much less is known in the context of geometrically reducible varieties. In this appendix, we show how the Stein factorization of a geometrically reducible variety can be used to give a link between these two settings. In particular, we show that the divisor and Picard schemes of a geometrically reducible variety are given by Weil restrictions, and deduce some consequences on the structure of the Abel map between the two schemes.

We use the notation of Section \ref{sec:isolated_divisors} throughout. Namely, let $X$ be a smooth, projective integral scheme over a number field $k$. We denote by $\DDiv_{X/k}$ the divisor scheme of $X$, by $\PPic_{X/k}$ the Picard scheme of $X$, and by $\AAbel_{X/k} : \DDiv_{X/k} \to \PPic_{X/k}$ the Abel map between the two. Let
\[
	X \overset{\pi'}{\to} \Spec K \overset{\eta}{\to} \Spec k
\]
be the Stein factorization of the structure morphism $\pi : X \to \Spec k$, where $K/k$ is a finite extension of number fields. In particular, $\eta$ is \'etale and $\pi'$ is smooth, projective and geometrically integral. We use analogous notation to denote the divisor and Picard schemes of $X/K$.

Throughout this appendix, and the body of the main paper, we will often rely on the properties of Weil restrictions under \'etale morphisms. While many such properties have already been made explicit in the literature, we have been unable to find a reference for the following statement on surjectivity.

\begin{theorem} \label{thm:divisor_picard_schemes:weil_restriction_surjectivity}
	Let $S' \to S$ be a finite \'etale morphism of schemes. Let $f' : X' \to Y'$ be a surjective morphism of $S'$-schemes, and assume that the Weil restrictions $\Res_{S'/S} X'$ and $\Res_{S'/S} Y'$ exist as schemes. Then $\Res_{S'/S}(f') : \Res_{S'/S}(X') \to \Res_{S'/S}(Y')$ is surjective. In particular, if the structure morphism $X' \to S'$ is surjective, then so is the morphism $\Res_{S'/S}(X') \to S$.
\end{theorem}

\begin{proof}
	Let $X = \Res_{S'/S}(X')$, $Y = \Res_{S'/S}(Y')$ and $f = \Res_{S'/S}(f') : X \to Y$. To show that $f$ is surjective, it suffices to check this over each geometric fiber of $Y \to S$. As Weil restriction commutes with base change, we may therefore assume that $S = \Spec k$ for an algebraically closed field $k$.
	
	Since $S' \to S$ is finite \'etale, $S'$ is isomorphic to $\Spec k^{d}$, for some natural number $d$. By definition, the map $f : X(k) \to Y(k)$ equals the map $f'^{d} : X'(k)^{d} \to Y'(k)^{d}$. Since $f'$ is surjective and $k$ is algebraically closed, the latter map is also surjective. Therefore, $f$ is surjective.
	
	The last statement follows from the fact that the morphism $\Res_{S'/S} X' \to S$ is the Weil restriction of the structure morphism $X' \to S'$, and the above.
\end{proof}

Firstly, we show that the Picard scheme $\PPic_{X/k}$ is isomorphic to the Weil restriction of the Picard scheme $\PPic_{X/K}$. This fact is stated, without proof, in much greater generality by Grothendieck in \cite[Proposition 6.1]{fga5}. However, we are only able to provide a proof when the extension $K/k$ is separable. The separable case is also treated in \cite[Lemma 3.7]{konstantinou2024}. We give a different proof, as the specifics of the isomorphism will be used later.

\begin{theorem} \label{thm:divisor_picard_schemes:picard_weil_restriction}
	Let $X, k$ and $K$ be as above. Then there exists an isomorphism of group schemes
	\[
		\varphi : \PPic_{X/k} \cong \Res_{K/k} \PPic_{X/K}.
	\]
\end{theorem}

\begin{proof}
	Since $\PPic_{X/k}$ exists as a scheme, it suffices to establish an isomorphism $\varphi_{T} : \PPic_{X/k}(T) \to \PPic_{X/K}(T_{K})$ for all $k$-schemes $T$, which is functorial in $T$.
	
	Let $T$ be a $k$-scheme. Define an fppf-cover $T' \to T$ of $T$ to be the morphism
	\[
		T' := \bigsqcup_{i \in I} T_{i} \to T,
	\]
	where $\{T_{i} \to T\}_{i \in I}$ is a covering in the fppf topology. As is discussed in \cite[253]{kleiman2005}, the set $\PPic_{X/k}(T)$ consists of invertible sheaves $\mathcal{L} \in \Pic(X_{T} \times_{T} T')$, where $T'$ is an fppf-cover of $T$, and such that there exists an fppf-cover $T'' \to T' \times_{T} T'$ such that the pullbacks of $\mathcal{L}$ with respect to $T'' \to T' \times_{T} T' \to T'$, for both projections $T' \times_{T} T' \to T'$, are isomorphic. Moreover, two such sheaves $\mathcal{L}_{1} \in \Pic(X_{T} \times_{T} T_{1})$ and $\mathcal{L}_{2} \in \Pic(X_{T} \times_{T} T_{2})$ define the same point of $\PPic_{X/k}(T)$ if and only if there exists an fppf-cover $T' \to T_{1} \times_{T} T_{2}$ such that the pullbacks of $\mathcal{L}_{1}$ and $\mathcal{L}_{2}$ to $X_{T'}$ are isomorphic. A similar description holds for $\PPic_{X/K}(T_{K})$.
	
	Consider a point $x \in \PPic_{X/k}(T)$ represented by an invertible sheaf $\mathcal{L} \in \Pic(X_{T} \times_{T} T')$. Since $X_{T} \times_{T} T' = X_{T_{K}} \times_{T_{K}} T'_{K}$, we can consider $\mathcal{L}$ as an element of $\Pic(X_{T_{K}} \times_{T_{K}} T'_{K})$. Since $T' \to T$ is an fppf-cover, so is the base change $T'_{K} \to T_{K}$. It is straightforward to check that $\mathcal{L} \in \Pic(X_{T_{K}} \times_{T_{K}} T'_{K})$ represents a point of $\PPic_{X/K}(T_{K})$, and that the latter does not depend on the choice of representative of $x$. This defines the group homomorphism $\varphi_{T} : \PPic_{X/k}(T) \to \PPic_{X/K}(T_{K})$.
	
	Let $f : S \to T$ be a morphism of $k$-schemes, and let $x \in \PPic_{X/k}(T)$ be a point represented by an invertible sheaf $\mathcal{L} \in \Pic(X_{T} \times_{T} T')$. The morphism $\PPic_{X/k}(T) \to \PPic_{X/k}(S)$ maps the point $x$ to the point $y \in \PPic_{X/k}(S)$ represented by the pullback of $\mathcal{L}$ with respect to the morphism $X_{S} \times_{S} (T' \times_{T} S) \to X_{T} \times_{T} T'$. This latter morphism is the same as the morphism $X_{S_{K}} \times_{S_{K}} (T'_{K} \times_{T} S_{K}) \to X_{T_{K}} \times_{T_{K}} T'_{K}$, since both correspond to the base change of the morphism $X_{S} \to X_{T}$ with respect to $T' \to T$. This establishes the functoriality of $\varphi_{T}$ in $T$.
	
	Consider now a point of $\PPic_{X/K}(T_{K})$ represented by an invertible sheaf $\mathcal{L} \in \Pic(X_{T} \times_{T_{K}} T')$. Let $F$ be the Galois closure of the field extension $K/k$. Since $K/k$ is a separable extension, we have $K \otimes_{k} F \cong F^{n}$, where $n = [K : k]$. Therefore, we obtain that $T_{K} \times_{T} T_{F} = \bigsqcup_{i = 1}^{n} T_{F}$. Let $U$ be the fiber product $T' \times_{T_{K}} \bigsqcup_{i = 1}^{n} T_{F}$, and denote by $U_{i}$ its fiber over the $i$-th point $\iota_{i} : T_{F} \to \bigsqcup_{i = 1}^{n} T_{F}$, so that $U = \bigsqcup_{i = 1}^{n} U_{i}$. Note that, since the morphism $T' \to T_{K}$ is an fppf-cover, so too is $U_{i} \to T_{F}$.
	
	Consider the product $\prod_{j = 1}^{n} U_{j} \to T_{F}$, which, by the above, is an fppf-cover of $T_{F}$. Since the morphism $T_{F} \to T$ is an fppf-cover, it follows that $\prod_{j = 1}^{n} U_{j}$ is also an fppf-cover of $T$. By construction, we have
	\[
		\left(\prod_{j = 1}^{n} U_{j}\right) \times_{T} T_{K} = \left(\prod_{j = 1}^{n} U_{j}\right) \times_{T_{F}} \left(\bigsqcup_{i = 1}^{n} T_{F}\right) = \bigsqcup_{i = 1}^{n} \left(\prod_{j = 1}^{n} U_{j}\right).
	\]
	The projection morphism $\bigsqcup_{i = 1}^{n} \left(\prod_{j = 1}^{n} U_{j}\right) \to \bigsqcup_{i = 1}^{n} T_{F}$ factors through $U$, with the morphism $\bigsqcup_{i = 1}^{n} \left(\prod_{j = 1}^{n} U_{j}\right) \to U$ given by the projection map $\prod_{j = 1}^{n} U_{j} \to U_{i}$ over the $i$-th point $\iota_{i} : T_{F} \to \bigsqcup_{i = 1}^{n} T_{F}$. Note that this morphism is an fppf-cover. Consider the invertible sheaf $\mathcal{L}' \in \Pic(X_{T} \times_{T} \prod_{j = 1}^{n} U_{j})$ given by the pullback of $\mathcal{L}$ with respect to the composition $\bigsqcup_{i = 1}^{n} \left(\prod_{j = 1}^{n} U_{j}\right) \to U \to T'$. By applying the above construction repeatedly, one can show that $\mathcal{L}'$ defines a point of $\PPic_{X/k}(T)$,  and that this point does not depend on the choice of $T'$ and $\mathcal{L}'$. Therefore, we obtain a map $\widetilde{\varphi}_{T} : \PPic_{X/K}(T_{K}) \to \PPic_{X/k}(T)$.
	
	The above construction can also be summarized with the following commutative diagram, in which all squares are Cartesian. We use fppf to indicate fppf-covers.
	\[\begin{tikzcd}[sep=tiny]
		& X_{T} \times_{T} \prod_{j = 1}^{n} U_{j} \arrow[rr] \arrow[dd] & & \bigsqcup_{i = 1}^{n} \left( \prod_{j = 1}^{n} U_{j} \right) \arrow[rr] \arrow[dd, "\text{fppf}"] & & \prod_{j = 1}^{n} U_{j} \arrow[dddd, "\text{fppf}"] \\
		& & & & U_{i} \arrow[dl] \arrow[dd, "\text{fppf}"] \\
		& X_{T_{K}} \times_{T_{K}} U \arrow[rr] \arrow[dl] \arrow[dd] & & U \arrow[dl, "\text{fppf}"] \arrow[dd] \\
		X_{T_{K}} \times_{T_{K}} T' \arrow[dd] \arrow[rr, crossing over] & & T' & & T_{F} \arrow[dl, "\iota_{i}"] \\
		& X_{T} \times_{T} T_{F} \arrow[ld] \arrow[rr] & & \bigsqcup_{i = 1}^{n} T_{F} \arrow[ld] \arrow[rr] & & T_{F} \arrow[ld, "\text{fppf}"] \\
		X_{T} = X_{T_{K}} \arrow[rr] & & T_{K} \arrow[rr] \arrow[from=uu, crossing over, "\text{fppf}", near start] & & T
	\end{tikzcd}\]
	
	It remains to show that $\varphi_{T}$ and $\widetilde{\varphi}_{T}$ are inverses. Throughout, we use the following fact, which follows from the explicit description of $\PPic_{X/k}(T)$ given above: let $\mathcal{L} \in \Pic(X_{T} \times T')$ be a representative for a point $x \in \PPic_{X/k}(T)$, and let $f: T'' \to T'$ be an fppf-cover of $T'$. Then the pullback $f_{X}^{\ast} \mathcal{L} \in \Pic(X_{T} \times_{T} T'')$ of $\mathcal{L}$ with respect to $f_{X} : X_{T} \times_{T} T'' \to X_{T} \times_{T} T'$ is also a representative of $x$.
	
	Let $x \in \PPic_{X/K}(T_{K})$ be a point represented by an invertible sheaf $\mathcal{L} \in \Pic(X_{T_{K}} \times_{T_{K}} T')$. Using the notation in the diagram above, the point $\varphi_{T}(\widetilde{\varphi}_{T}(x))$ is represented by the pullback of $\mathcal{L}$ with respect to the map $\bigsqcup_{i = 1}^{n} \left(\prod_{j = 1}^{n} U_{j}\right) \to T'$. By the above fact, this is precisely $x$, and so the composition $\varphi_{T} \circ \widetilde{\varphi}_{T}$ is the identity.
	
	On the other hand, let $x \in \PPic_{X/k}(T)$ be a point represented by an invertible sheaf $\mathcal{L} \in \Pic(X_{T} \times_{T} T')$. Again, we employ the notation in the diagram above, where $T'$ is replaced by $T'_{K}$. Then, the point $\widetilde{\varphi}_{T}(\varphi_{T}(x))$ is represented by the pullback of $\mathcal{L}$ with respect to the map $\bigsqcup_{i = 1}^{n} \left(\prod_{j = 1}^{n} U_{j}\right) \to T'_{K}$. By construction, we have
	\[
		U = T'_{K} \times_{T_{K}} \bigsqcup_{i = 1}^{n} T_{F} = (T' \times_{T} T_{F}) \times_{T_{F}} \bigsqcup_{i = 1}^{n} T_{F} = \bigsqcup_{i = 1}^{n} (T' \times_{T} T_{F}).
	\]
	Therefore, $U_{i} = T' \times_{T} T_{F}$ for all $i$, and so the morphism $\bigsqcup_{i = 1}^{n} \left(\prod_{j = 1}^{n} U_{j}\right) \to U$ is equal to the pullback of the morphism $\prod_{j = 1}^{n} U_{j} \to U_{1}$ with respect to the morphism $T_{K} \to T$, up to an isomorphism of $\bigsqcup_{i = 1}^{n} \left(\prod_{j = 1}^{n} U_{j}\right)$. Therefore, the point $\widetilde{\varphi}_{T}(\varphi_{T}(x))$ is represented by the pullback of $\mathcal{L}$ with respect to the map $\prod_{j = 1}^{n} U_{j} \to U_{1} \to T'$. By the above fact, it follows that $\widetilde{\varphi}_{T}(\varphi_{T}(x)) = x$, and so the composition $\widetilde{\varphi}_{T} \circ \varphi_{T}$ is the identity.
	
	Thus, $\varphi_{T}$ is an isomorphism for all $k$-schemes $T$, completing the proof.
\end{proof}

Similarly, the divisor scheme $\DDiv_{X/k}$ is isomorphic to the Weil restriction of the divisor scheme $\DDiv_{X/K}$. Moreover, this isomorphism is compatible with the isomorphism of Picard schemes just described, as the following theorem shows.

\begin{theorem} \label{thm:divisor_picard_schemes:divisor_weil_restriction}
	Let $X, k$ and $K$ be as above. Then there exists an isomorphism of schemes
	\[
		\psi : \DDiv_{X/k} \cong \Res_{K/k} \DDiv_{X/K}.
	\]
	Moreover, there is a commutative diagram
	\[\begin{tikzcd}
		\DDiv_{X/k} \arrow[d, "\psi"] \arrow[rr, "\AAbel_{X/k}"] & & \PPic_{X/k} \arrow[d, "\varphi"] \\
		\Res_{K/k} \DDiv_{X/K} \arrow[rr, "\Res_{K/k}\AAbel_{X/K}"] & & \Res_{K/k} \PPic_{X/K}.
	\end{tikzcd}\]
\end{theorem}

\begin{proof}
	To prove the first part of the theorem, it suffices to give an isomorphism $\psi_{T} : \DDiv_{X/k}(T) \to \DDiv_{X/K}(T_{K})$ for all $k$-schemes $T$, which is functorial in $T$.
	
	Let $T$ be a $k$-scheme. By definition, $\DDiv_{X/k}(T)$ is the set of relative effective Cartier divisors on $X_{T} / T$, that is to say, the set of effective Cartier divisors on $X_{T}$ which are flat over $T$. Similarly, $\DDiv_{X/K}(T_{K})$ is the set of effective Cartier divisors on $X_{T_{K}}$ which are flat over $T_{K}$. 
	
	Let $D \in \DDiv_{X/k}(T)$ be a relative effective Cartier divisor on $X_{T} / T$. As the schemes $X_{T}$ and $X_{T_{K}}$ are equal, $D$ is also an effective Cartier divisor on $X_{T_{K}}$. By definition, the composition $D \hookrightarrow X_{T} \to T_{K} \to T$ is flat. Moreover, since $K/k$ is a finite separable field extension, the morphism $T_{K} \to T$ is unramified. Therefore, the composition $D \hookrightarrow X_{T} \to T_{K}$ is also flat. In other words, $D$ is flat over $T_{K}$, and so is a relative effective Cartier divisor on $X_{T_{K}} / T_{K}$. This defines the map $\psi_{T} : \DDiv_{X/k}(T) \to \DDiv_{X/K}(T_{K})$.
	
	Let $f : S \to T$ be a morphism of $k$-schemes. The map $\DDiv_{X/k}(T) \to \DDiv_{X/k}(S)$ is given by the pullback with respect to the morphism $X_{S} \to X_{T}$. Similarly, the map $\DDiv_{X/K}(T_{K}) \to \DDiv_{X/K}(S_{K})$ is given by the pullback with respect to the morphism $X_{S_{K}} \to X_{T_{K}}$. Since these two morphisms are equal, and pullback of relative effective Cartier divisors does not depend on the base, the functoriality of the map $\psi_{T}$ follows.
	
	By construction, the map $\psi_{T}$ is injective. Moreover, since $K/k$ is a field extension, the morphism $T_{K} \to T$ is flat. Therefore, any relative effective Cartier divisor on $X_{T_{K}} / T_{K}$ is also flat over $T$. Thus, the map $\psi_{T}$ is also surjective, and hence is an isomorphism.
	
	To show the second part of the theorem, it suffices to show that the following diagram commutes for all $k$-schemes T.
	\[\begin{tikzcd}
		\DDiv_{X/k}(T) \arrow[d, "\psi_{T}"] \arrow[r, "\AAbel_{X/k}"] & \PPic_{X/k}(T) \arrow[d, "\varphi_{T}"] \\
		\DDiv_{X/K}(T_{K}) \arrow[r, "\AAbel_{X/K}"] & \PPic_{X/K}(T_{K}).
	\end{tikzcd}\]
	Let $T$ be a $k$-scheme, and $D \in \DDiv_{X/k}(T)$ be a relative effective Cartier divisor on $X_{T} / T$. By definition, the point $\AAbel_{X/k}(D) \in \PPic_{X/k}(T)$ is represented by the invertible sheaf $\O_{X_{T}}(D) \in \Pic(X_{T})$. It follows that the point $\varphi_{T}(\AAbel_{X/k}(D))$ is represented by the invertible sheaf $\O_{X_{T}}(D) \in \Pic(X_{T_{K}})$. On the other hand, since $\psi_{T}$ maps $D$ to itself, the point $\AAbel_{X/K}(\psi_{T}(D))$ is represented by the invertible sheaf $\O_{X_{T_{K}}}(D) \in \Pic(X_{T_{K}})$.  As the schemes $X_{T}$ and $X_{T_{K}}$ coincide, the result follows.
\end{proof}

This identification between the divisor and Picard schemes of a geometrically reducible variety and the Weil restriction of the corresponding schemes of the Stein factorization allows us to carry some of the results for the geometrically integral case to the geometrically reducible case. For instance, the Abel map is proper in both settings, as the following theorem shows.

\begin{theorem} \label{thm:divisor_picard_schemes:abel_proper}
	Let $X, k$ and $K$ be as above. Then the Abel map $\AAbel_{X/k} : \DDiv_{X/k} \to \PPic_{X/k}$ is proper.
\end{theorem}

\begin{proof}
	By Theorem \ref{thm:divisor_picard_schemes:divisor_weil_restriction}, it suffices to show that the Weil restriction $\Res_{K/k} \AAbel_{X/K}$ is proper. As the structure map $\pi' : X \to \Spec K$ is smooth and projective, with geometrically integral fibers, the Abel map $\AAbel_{X/K}$ is proper by \cite[Exercise 4.12]{kleiman2005}. Moreover, the morphism $\eta : \Spec K \to \Spec k$ is finite and \'etale, and so the Weil restriction $\Res_{K/k} \AAbel_{X/K}$ is proper by \cite[Proposition 6.2.9]{spec}.
\end{proof}

In much the same way, the description of the Abel map in terms of Weil restrictions leads to a concrete description of the fibers of the Abel map in the case of geometrically reducible varieties.

\begin{theorem} \label{thm:divisor_picard_schemes:abel_fiber}
	Let $X, k$ and $K$ be as above. Let $\mathcal{L}$ be an invertible sheaf on $X$, and $[\mathcal{L}] \in \PPic_{X/k}(k)$ the corresponding point of the Picard scheme. Then, the fiber $\AAbel_{X/k}^{-1}([\mathcal{L}])$ is isomorphic to $\Res_{K/k}(\mathbb{P}_{K}^{n-1})$, where $n = \dim_{K} H^{0}(X, \mathcal{L})$.
\end{theorem}

\begin{proof}
	By Theorem \ref{thm:divisor_picard_schemes:divisor_weil_restriction}, we have a commutative diagram
	\[\begin{tikzcd}
		\DDiv_{X/k} \arrow[d, "\psi"] \arrow[rr, "\AAbel_{X/k}"] & & \PPic_{X/k} \arrow[d, "\varphi"] \\
		\Res_{K/k} \DDiv_{X/K} \arrow[rr, "\Res_{K/k}\AAbel_{X/K}"] & & \Res_{K/k} \PPic_{X/K}.
	\end{tikzcd}\]
	As both $\varphi$ and $\psi$ are isomorphisms, the fiber $\AAbel_{X/k}^{-1}([\mathcal{L}])$ is isomorphic to the fiber $(\Res_{K/k} \AAbel_{X/K})^{-1}(\varphi([\mathcal{L}]))$. By \cite[Proposition A.5.2 (3)]{conrad2015} and \cite[Proposition A.5.7]{conrad2015}, the fiber $(\Res_{K/k} \AAbel_{X/K})^{-1}(\varphi([\mathcal{L}]))$ is isomorphic to the Weil restriction $\Res_{K/k}(\AAbel_{X/K}^{-1}(x))$, where $x \in \PPic_{X/K}(K)$ is the point corresponding to $\varphi([\mathcal{L}])$ under the isomorphism $\Res_{K/k} \PPic_{X/K}(k) \cong \PPic_{X/K}(K)$. In particular, using the explicit description of $\varphi$, we have $x = [\mathcal{L}] \in \PPic_{X/K}(K)$. Therefore, we have
	\[
		\AAbel_{X/k}^{-1}([\mathcal{L}]) \cong \Res_{K/k}(\AAbel_{X/K}^{-1}([\mathcal{L}])).
	\]
	
	By \cite[Proposition 8.2.7]{bosch2012}, since $X/K$ is geometrically integral, the fiber $\AAbel_{X/K}^{-1}([\mathcal{L}])$ is represented by $\mathbb{P}(\mathcal{F})$, for some $\O_{\Spec K}$-module $\mathcal{F}$. Moreover, every $\O_{X}$-module is cohomologically flat over $\Spec K$ in dimension zero, as $K$ is a field. Therefore, the aforementioned result also shows that $\mathcal{F}$ is locally free, and is isomorphic to the dual of $\pi_{\ast}(\mathcal{L})$, where $\pi : X \to \Spec K$ is the structure map. Thus, we obtain that $\mathbb{P}(\mathcal{F}) \cong \mathbb{P}^{n-1}_{K}$, where $n$ is the dimension of the $K$-vector space $\pi_{\ast}(\mathcal{L})$. By \cite[Corollary II.5.2]{mumford1970}, it follows that $n = \dim_{K} H^{0}(X, \mathcal{L})$, as required.
\end{proof}

Finally, one can give a description of the image $\AAbel_{X/k}(\DDiv_{X/k}(k))$ as a subset of $\PPic_{X/k}(k)$. While not as explicit as in the case of geometrically integral varieties, this is description still proves to be very useful in Section \ref{sec:isolated_divisors}.

\begin{theorem} \label{thm:divisor_picard_schemes:abel_rational_image}
	Let $X, k$ and $K$ be as above. Then there exists a subgroup $G \leq \PPic_{X/k}(k)$ such that
	\[
		\AAbel_{X/k}(\DDiv_{X/k}(k)) = \AAbel_{X/k}(\DDiv_{X/k})(k) \cap G.
	\]
	Moreover, the quotient group $\PPic_{X/k}(k) / G$ is torsion.
\end{theorem}

\begin{proof}
	By Theorem \ref{thm:divisor_picard_schemes:divisor_weil_restriction}, we have the following commutative diagram
	\[\begin{tikzcd}
		\DDiv_{X/k}(k) \arrow[d, "\psi_{k}"] \arrow[r, "\AAbel_{X/k}"] & \PPic_{X/k}(k) \arrow[d, "\varphi_{k}"] \\
		\DDiv_{X/K}(K) \arrow[r, "\AAbel_{X/K}"] & \PPic_{X/K}(K),
	\end{tikzcd}\]
	where both $\varphi_{k}$ and $\psi_{k}$ are isomorphisms. Therefore, we have
	\[
		\AAbel_{X/k}(\DDiv_{X/k}(k)) = \varphi_{k}^{-1}(\AAbel_{X/K}(\DDiv_{X/K}(K))). \tag{$\ast$}
	\]
	Since $X/K$ is geometrically integral, by \cite[Proposition 8.1.4]{bosch2012}, we obtain the following exact sequence
	\[
		0 \to \Pic_{X/K}(K) \to \PPic_{X/K}(K) \to \Br(K) \to \Br(X).
	\]
	In particular, the group of $K$-points of the relative Picard functor $\Pic_{X/K}(K)$ embeds into the group of $K$-points of the Picard scheme $\PPic_{X/K}(K)$. Moreover, the quotient group $\PPic_{X/K}(K) / \Pic_{X/K}(K)$ is a subgroup of the Brauer group $\Br(K)$ of $K$, and hence is a torsion group.
	
	Since the Abel map $\AAbel_{X/K} : \DDiv_{X/K} \to \PPic_{X/K}$ factors through the relative Picard functor $\Pic_{X/K}$, we have
	\[
		\AAbel_{X/K}(\DDiv_{X/K}(K)) \subseteq \AAbel_{X/K}(\DDiv_{X/K})(K) \cap \Pic_{X/K}(K).
	\]
	Let $x \in \AAbel_{X/K}(\DDiv_{X/K})(K) \cap \Pic_{X/K}(K)$, and consider the fiber $\AAbel_{X/K}^{-1}(x)$. Since $x \in \Pic_{X/K}(K)$, $x$ represents the class of an invertible sheaf $\mathcal{L}$ on $X$. By Theorem \ref{thm:divisor_picard_schemes:abel_fiber}, the fiber $\AAbel_{X/K}^{-1}(x)$ is isomorphic to $\mathbb{P}_{K}^{n-1}$, for some $n \geq 0$. Moreover, since $x \in \AAbel_{X/K}(\DDiv_{X/K})(K)$, the fiber $\AAbel_{X/K}^{-1}(x)$ is not empty, and so $n \geq 1$. In particular, the set $\AAbel_{X/K}^{-1}(x)(K)$ is non-empty, and so $x \in \AAbel_{X/K}(\DDiv_{X/K}(K))$. Therefore, we have shown that
	\[
		\AAbel_{X/K}(\DDiv_{X/K}(K)) = \AAbel_{X/K}(\DDiv_{X/K})(K) \cap \Pic_{X/K}(K).
	\]
	Using $(\ast)$, since $\varphi_{k}^{-1}$ is an isomorphism, we have
	\[
		\AAbel_{X/k}(\DDiv_{X/k}(k)) = \varphi_{k}^{-1}(\AAbel_{X/K}(\DDiv_{X/K})(K)) \cap \varphi_{k}^{-1}(\Pic_{X/K}(K)).
	\]
	Moreover, denoting by $G$ the subgroup $\varphi_{k}^{-1}(\Pic_{X/K}(K))$ of $\PPic_{X/k}(k)$, the quotient group $\PPic_{X/k}(k) / G \cong \PPic_{X/K}(K) / \Pic_{X/K}(K)$ is torsion.
	
	By construction, we have $\AAbel_{X/K}(\DDiv_{X/K})(K) = \Res_{K/k}(\AAbel_{X/K}(\DDiv_{X/K}))(k)$ and
	\begin{align*}
		\varphi_{k}^{-1} \! \left(\AAbel_{X/K}(\DDiv_{X/K})(K)\right) & = \varphi^{-1} \! \left(\Res_{K/k}(\AAbel_{X/K}(\DDiv_{X/K}))(k)\right) \\
		& = \varphi^{-1} \! \left(\Res_{K/k}(\AAbel_{X/K}(\DDiv_{X/K}))\right) \! (k),
	\end{align*}
	where $\varphi : \PPic_{X/k} \to \Res_{K/k} \PPic_{X/K}$ is the isomorphism constructed in Theorem \ref{thm:divisor_picard_schemes:divisor_weil_restriction}. Since $K/k$ is a finite separable extension of fields, by Theorem \ref{thm:divisor_picard_schemes:weil_restriction_surjectivity} it follows that
	\[
		\Res_{K/k}(\AAbel_{X/K}(\DDiv_{X/K})) = \Res_{K/k} \AAbel_{X/K}(\Res_{K/k} \DDiv_{X/K}).
	\]
	Therefore, using Theorem \ref{thm:divisor_picard_schemes:divisor_weil_restriction} once more, we obtain that
	\[
		\varphi^{-1} \! \left(\Res_{K/k}(\AAbel_{X/K}(\DDiv_{X/K}))\right) = \AAbel_{X/k}(\DDiv_{X/k}).
	\]
	Thus, we have
	\[
		\AAbel_{X/k}(\DDiv_{X/k}(k)) = \AAbel_{X/k}(\DDiv_{X/k})(k) \cap G,
	\]
	where $\PPic_{X/k}(k)/G$ is torsion, as required.
\end{proof}
	
	\printbibliography
\end{document}